\documentclass{amsart}

\newtheorem{thm}{Theorem}[section]
\newtheorem{lemma}{Lemma}[section]
\newtheorem{prop}{Proposition}[section]
\newtheorem{coro}{Corollary}[section]

\newtheorem{assume}{Assumption}[section]
\newtheorem{remark}{Remark}[section]

\newtheorem{alphthm}{Theorem}[section]

\newtheorem{alphlemma}{Lemma}[section]

\newcommand{\lct}{\; \raisebox{-.96ex}{$\stackrel{\textstyle <}{\sim}$} \;}
\newcommand{\gct}{\; \raisebox{-.96ex}{$\stackrel{\textstyle >}{\sim}$} \;}
\newcommand{\mbb}{\mathbb}

\begin{document}

\title[Weighted restriction estimates]{Weighted restriction estimates \\
                                       using polynomial partitioning}

\author{Bassam Shayya}
\address{Department of Mathematics\\
         American University of Beirut\\
         Beirut\\
         Lebanon}
\email{bshayya@aub.edu.lb}




\begin{abstract}
We use the polynomial partitioning method of Guth~\cite{guth:poly} to prove
weighted Fourier restriction estimates in $\mbb R^3$ with exponents $p$ that
range between $3$ and $3.25$, depending on the weight. As a corollary to our
main theorem, we obtain new (non-weighted) local and global restriction
estimates for compact $C^\infty$ surfaces $S \subset \mbb R^3$ with strictly
positive second fundamental form. For example, we establish the global
restriction estimate $\| Ef \|_{L^p(\mbb R^3)} \lct \| f \|_{L^q(S)}$ in
the full conjectured range of exponents for $p > 3.25$ (up to the sharp
line), and the global restriction estimate
$\| Ef \|_{L^p(\Omega)} \lct \| f \|_{L^2(S)}$ for $p>3$ and certain
sets $\Omega \subset \mbb R^3$ of infinite Lebesgue measure. As a corollary
to our main theorem, we also obtain new results on the decay of spherical
means of Fourier transforms of positive compactly supported measures on
$\mbb R^3$ with finite $\alpha$-dimensional energies.
\end{abstract}

\maketitle

\section{Introduction}

There has been a surge of activity in recent years of using polynomial
partitioning methods to study some important and open problems in
combinatorics and harmonic analysis. One striking example is Guth's recent
paper~\cite{guth:poly}, where polynomial partitioning was used to obtain
progress on the Fourier restriction problem in $\mbb R^3$.

Let $S$ be a compact $C^\infty$ surface in $\mbb R^3$ with strictly positive
second fundamental form. In a recent paper~\cite{guth:poly}, Guth
used polynomial partitioning to obtain a new local restriction estimate on
$S$: to every $\epsilon > 0$ there is a constant $C_\epsilon$ (which also
depends on $S$) such that
\begin{equation}
\label{3pt25}
\int_{|x| \leq R} |E f(x)|^{3.25} dx
\leq C_\epsilon R^\epsilon \| f \|_{L^\infty(S)}^{3.25}
\end{equation}
for all $f \in L^\infty(S)=L^\infty(\sigma)$, where $\sigma$ is the surface
measure on $S$ and $E f$ is the Fourier extension (or adjoint restriction)
operator on $S$ defined as
\begin{displaymath}
E f(x)= E_S f(x) = \widehat{f d\sigma}(x)
                 = \int e^{-2 \pi i x \cdot \xi} f(\xi) d\sigma(\xi).
\end{displaymath}
Using Tao's $\epsilon$-removal lemma, (\ref{3pt25}) can be turned into a
global restriction estimate:
\begin{equation}
\label{3pt25global}
\int_{\mbb R^3} |E f(x)|^p dx
\leq C \, \| f \|_{L^\infty(S)}^p \hspace{0.44in} (p > 3.25)
\end{equation}
for all $f \in L^\infty(S)$.

The restriction conjecture in harmonic analysis asserts that the exponent
$3.25$ in (\ref{3pt25}) and (\ref{3pt25global}) can be replaced by $3$.
Guth's theorem is the current best known result on this conjecture in
$\mbb R^3$. The purpose of this paper is to lower the $3.25$ exponent by
considering weighted variants of (\ref{3pt25}). For example, as a
consequence of the results of this paper, one obtains the estimates
\begin{equation}
\label{3ptzero}
\int_{B_R} |E f(x)|^3 \chi_{\Omega_1}(x) dx
\leq C_\epsilon R^\epsilon \| f \|_{L^2(S)}^3
\end{equation}
for all $f \in L^2(S)$, and
\begin{equation}
\label{3ptfourteen}
\int_{B_R} |E f(x)|^{22/7} \chi_{\Omega_2}(x) dx
\leq C_\epsilon R^\epsilon \| f \|_{L^2(S)}^3 \| f \|_{L^\infty(S)}^{1/7}
\end{equation}
for all $f \in L^\infty(S)$, where $B_R$ is any ball in $\mbb R^3$ of radius
$R$, $\chi_{\Omega_1}$ is the characteristic function of the set
\begin{displaymath}
\Omega_1 = \{ (x_1,x_2,x_3) \in \mbb R^3 : |x_3| \leq |(x_1,x_2)|^{-1/2} \},
\end{displaymath}
and $\chi_{\Omega_2}$ is the characteristic function of the set
\begin{displaymath}
\Omega_2 = \{ (x_1,x_2,x_3) \in \mbb R^3 : |x_3| \leq 1 \}.
\end{displaymath}
We refer the reader to the paragraph following the statement of our main
theorem, Theorem \ref{mainjj}, for the proofs of (\ref{3ptzero}) and
(\ref{3ptfourteen}).

Some of our local restriction estimates can be turned into global ones. For
example, we establish the global version of (\ref{3ptfourteen}):
\begin{displaymath}
\int_{\mbb R^3} |E f(x)|^p \chi_{\Omega_2}(x) dx
\leq C \, \| f \|_{L^{44/21}(S)}^p \hspace{0.44in} (p > 22/7)
\end{displaymath}
for all $f \in L^{44/21}(S)$. For certain types of sets, we get global
estimates for the full range of exponents $p > 3$. For example, we prove
that
\begin{displaymath}
\int_{\mbb R^3} |E f(x)|^p \chi_{\Omega}(x) dx
\leq C \, \| f \|_{L^2(S)}^p \hspace{0.44in} (p > 3)
\end{displaymath}
for all $f \in L^2(S)$, where
\begin{displaymath}
\Omega = \cup_{m=0}^\infty \cup_{n=0}^\infty
\big( \mbb R \times [-1,1]^2 + (0,m^4,n^4) \big).
\end{displaymath}
We also obtain the following improvement on (\ref{3pt25global}):
\begin{equation}
\label{3pt25newglobal}
\int_{\mbb R^3} |E f(x)|^p dx
\leq C \, \| f \|_{L^q(S)}^p \hspace{0.44in}
\big( p > 3.25, \; q' < p/2 \big)
\end{equation}
for all $f \in L^q(S)$, where $q'$ is the exponent conjugate to $q$. Up to
the end point $q'=p/2$, the range of the $q$ exponent in
(\ref{3pt25newglobal}) is known to be the best possible.

All the results of this paper concerning global Fourier restriction
estimates are stated in Corollary \ref{global}.

Let $M(\mbb R^3)$ be the space of all complex Borel measures on $\mbb R^3$,
and $S \subset \mbb R^3$ be, as above, a compact $C^\infty$ surface with
strictly positive second fundamental form. The second application of the
main theorem of this paper concerns the decay properties of the $L^q$ norms
\begin{displaymath}
\| \widehat{\mu}(R \cdot) \|_{L^q(S)}
= \Big( \int |\widehat{\mu}(R\xi)|^q d\sigma(\xi) \Big)^{1/q}
\end{displaymath}
as $R \to \infty$, where $\mu \in M(\mbb R^3)$ is a positive compactly
supported measure with finite $\alpha$-dimensional energy $I_\alpha(\mu)$
(see (\ref{defofenergy}) for the definition of $I_\alpha(\mu)$).
This is an important topic for the study of
distance sets in geometric measure theory, as we explain in Section 3 below.

For $q = 2$, Erdo\v{g}an \cite{mbe:birestfal} proved the decay
estimate
\begin{displaymath}
\| \widehat{\mu}(R \cdot) \|_{L^2(\mbb S^2)} \leq C_\epsilon R^\epsilon
R^{-(\alpha/4)-(1/8)} \, \sqrt{I_\alpha(\mu)}
\end{displaymath}
for $3/2 \leq \alpha \leq 5/2$, where $\mbb S^2$ is the unit sphere in
$\mbb R^3$ (see (\ref{l2erdogan}) for the $n$-dimensional version of this
estimate), and used it to get the best known result on the distance set
problem in $\mbb R^3$, but no better decay estimate was known for
$1 \leq q < 2$. In this paper, we get a better decay estimate for the range
of exponents $1 \leq q \leq p_0$, where $p_0=4(4\alpha+3)/(10\alpha+3)$ and
$3/2 < \alpha < 5/2$. We also reprove Erdo\v{g}an's estimate when
$3/2 \leq \alpha < 2$, obtaining a proof of his distance set result that is
based on polynomial partitioning. All these results are stated in Corollary
\ref{decaysphm}.

Let $\mu$ be a positive and compactly supported measure in $M(\mbb R^3)$
satisfying
\begin{displaymath}
\sup_{x \in \mbb R^3} \sup_{r >0} \frac{\mu(B(x,r))}{r^\alpha} < \infty.
\end{displaymath}
(We refer the reader to the first paragraph of Section 4 for an explanation
of how this condition relates to the condition $I_\alpha(\mu) < \infty$.)
The third application of our main theorem establishes $L^p(\mu)$ bounds on
exponential sums of the form $\sum_{l=1}^N a_l e^{2\pi i R w_l \cdot x}$,
where $w_1, \ldots, w_N \in S$ are $R^{-1}$-separated. These results are
stated in Corollary \ref{expsum}.

\subsection{The main theorem}

Suppose $3/2 \leq \alpha \leq 3$ and $H$ is a non-negative measurable
function on $\mbb R^3$. We define $A_\alpha(H)$ to be the infimum of the set
\begin{displaymath}
\Big\{ \lambda \in [0,\infty] : \int_{B(x_0,R)} H(x) dx \leq \lambda
R^\alpha \mbox{ for all } x_0 \in \mbb R^3 \mbox{ and } R \geq 1 \Big\}.
\end{displaymath}
We also define
\begin{displaymath}
A_{\alpha,p}(H)= \max \big[ A_\alpha(H), A_\alpha(H)^{1-\frac{p}{4}} \big]
\end{displaymath}
and
\begin{displaymath}
{\mathcal A}_{\alpha,p}(H)
= \max \big[ A_\alpha(H), A_\alpha(H)^{2-\frac{p}{2}} \big].
\end{displaymath}
Note that if $H_\tau$ is a translate of $H$, defined by
$H_\tau(x)=H(x+\tau)$, then $A_\alpha(H)=A_\alpha(H_\tau)$ (and, of course,
the same is true for $A_{\alpha,p}(H)$ and ${\mathcal A}_{\alpha,p}(H)$).

We say $H$ is a weight of dimension $\alpha$ if $\| H \|_{L^\infty} \leq 1$
and $A_\alpha(H) < \infty$. In this case, any translate $H_\tau$ of $H$
is also a weight of dimension $\alpha$.

We alert the reader that since $A_\beta(H) \leq A_\alpha(H)$ if
$\beta \geq \alpha$, a weight $H$ of dimension $\alpha$ is also a weight of
dimension $\beta$. So the phrase ``$H$ is a weight of dimension $\alpha$''
is just a way of expressing in words the inequality $A_\alpha(H) < \infty$,
and is not meant to assign the specific dimension $\alpha$ to the function
$H$.

The aim of this paper is to prove the following theorem.

\begin{thm}
\label{mainjj}
Let $S \subset \mbb R^3$ be a compact $C^\infty$ surface with strictly
positive second fundamental form, and $H$ be a weight of dimension $\alpha$.

{\rm (i)} Suppose $3/2 \leq \alpha < 5/2$. Then to every $\epsilon > 0$
there is a constant $C_\epsilon(\alpha,S)$ such that
\begin{displaymath}
\int_{B(0,R)} |E f(x)|^p H(x)dx
\leq C_\epsilon(\alpha,S) R^\epsilon A_{\alpha,p}(H)
     \| f \|_{L^2(S)}^3 \| f \|_{L^\infty(S)}^{p-3}
\end{displaymath}
for all $f \in L^\infty(S)$ and $R \geq 1$, where
\begin{displaymath}
p= 2 \frac{4 \alpha + 3}{2 \alpha + 3}.
\end{displaymath}

{\rm (ii)} Suppose $3/2 \leq \alpha < 2$. Then to every $\epsilon > 0$
there is a constant $C_\epsilon(\alpha,S)$ such that
\begin{displaymath}
\int_{B(0,R)} |E f(x)|^3 H(x)dx
\leq C_\epsilon(\alpha,S) R^\epsilon A_{\alpha,3}(H)
     R^{\frac{1}{4}(\alpha - \frac{3}{2})} \| f \|_{L^2(S)}^3
\end{displaymath}
for all $f \in L^2(S)$ and $R \geq 1$.

{\rm (iii)} Suppose $5/2 \leq \alpha \leq 3$. Let
$2 \leq \gamma < (11/2) - \alpha$. Then to every $\epsilon > 0$ there is a
constant $C_\epsilon(\alpha,\gamma,S)$ such that
\begin{displaymath}
\int_{B(0,R)} |E f(x)|^p H(x)dx
\leq C_\epsilon(\alpha,\gamma,S) R^\epsilon {\mathcal A}_{\alpha,p}(H)
     \| f \|_{L^2(S)}^{\gamma} \| f \|_{L^\infty(S)}^{p-\gamma}
\end{displaymath}
for all $f \in L^\infty(S)$ and $R \geq 1$, where $p=13/4$.
\end{thm}

For example, if $H=\chi_{\Omega_1}$, then
\begin{displaymath}
\int_{B_R} |Ef(x)|^3 H(x) dx= \int_{B(0,R)} \big| E
\big( e^{-2\pi i \tau \cdot \xi} f \big) (x) \big|^3
H(x+\tau)dx,
\end{displaymath}
where $\tau$ is the center of $B_R$. Applying part (i) (or part (ii)) of
Theorem \ref{mainjj} with $\alpha=3/2$, we prove (\ref{3ptzero}). Similarly,
taking $H=\chi_{\Omega_2}$ and applying part (i) of Theorem \ref{mainjj}
with $\alpha=2$, we prove (\ref{3ptfourteen}).

It would be very interesting to know if the estimates in Theorem
\ref{mainjj} have counterparts in dimension $n=2$ or $n \geq 4$. For
example, if $H$ is a weight on $\mbb R^n$ of dimension $\alpha=n/2$, and
$p=2n/(n-1)$, then do we have
\begin{displaymath}
\int_{B(0,R)} |E f(x)|^p H(x)dx
\leq C_\epsilon R^\epsilon A_{\alpha,p}(H) \| f \|_{L^2(S)}^p
\end{displaymath}
for all $f \in L^2(S)$, $R \geq 1$, and $\epsilon > 0$?

The main obstacle in generalizing Theorem \ref{mainjj} to higher dimensions
is the current unavailability of the needed Kakeya information (in the form
of Lemma \ref{Lemma3.6}) in dimensions four and more. We refer the reader to
Conjecture 11.6 in \cite{guth:poly2} for more information about this
important topic. As for an explanation of the reason why the methods of this
paper do not seem to apply in dimension $n=2$, we refer the reader to Remark
\ref{soandso}.

\subsection{Methodology}

We start with a couple of definitions.

If ${\mathcal L}$ is a set of lines in $\mbb R^3$, and $r \geq 2$ is an
integer, then the set of $r$-rich points of ${\mathcal L}$ is defined as
\begin{displaymath}
P_r({\mathcal L})= \{ x \in \mbb R^3 : x \mbox{ belongs to at least } r
                      \mbox{ lines from } {\mathcal L} \}.
\end{displaymath}

We denote the zero set of a polynomial $Q$ by $Z(Q)$, and we say $Q$ is
non-singular if $\nabla Q(x) \not= 0$ for all $x \in Z(Q)$.

To prove Theorem \ref{mainjj}, we employ Guth's polynomial partitioning
method from \cite{guth:poly}. The proof is carried along nine main steps.

Step one translates the geometric properties of the surface $S$ into a
decomposition, called the wave packet decomposition, which is applicable to
the functions $f \in L^2(S)$ and allows one who works locally in the ball
$B_R$ of center 0 and radius $R$ to think of $Ef$ as being essentially 
supported on long thin tubes of radius $R^{(1/2)+\delta}$ for some
parameter $\delta$ (which is positive and rather small).

Step two associates to every polynomial $P$ on $\mbb R^3$ of degree at most 
$D$, which is a product of non-singular polynomials, a partitioning of 
$\mbb R^3$ as follows. We know that $\mbb R^3 \setminus Z(P)$ is a disjoint 
union of at most $CD^3$ open sets $O_i$, which are often called cells. We
also know that a line can intersect at most $D+1$ of the cells $O_i$. We let 
$W$ be the $R^{(1/2)+\delta}$-neighborhood of $Z(P)$. Then 
$\mbb R^n \setminus W$ is a disjoint union of the open sets
$O_i'=O_i \setminus W$, and a tube of radius $R^{(1/2)+\delta}$ can 
intersect at most $D+1$ of the modified cells $O_i'$. 

Step three organizes those tubes from step one that meet $W$ into two 
groups: the transverse tubes are those tubes that intersect $W$ 
transversally, and the tangential tubes are those that lie in $W$ over a 
long stretch. 

Step four identifies the part of $Ef$ corresponding to those tubes that
point in different directions. Guth calls this the broad part of $Ef$. The 
relation between the broad part of $Ef$ and the long thin tubes that support 
$Ef$ resembles the relation between the set $P_r({\mathcal L})$, as defined
at the start of this subsection, and the lines of ${\mathcal L}$. This is 
the main reason behind the polynomial method becoming as important in 
restriction theory as it has become in incidence geometry. We refer the 
reader to Subsections 0.4 and 0.5 of \cite{guth:poly} for a thorough 
discussion about the common features between bounding the number of $r$-rich 
points of ${\mathcal L}$ and estimating the broad part of $Ef$.

Step five formulates a theorem, Theorem \ref{biltobr}, that estimates the 
$L^p(Hdx)$ norm of the broad part of $Ef$ on $B_R$ conditional on having a 
favorable bound on the contribution coming from the tangential tubes. The 
conditional bound on the tangential tubes is uniform over all polynomials 
$P$ satisfying $\mbox{Deg}(P) \leq D$ (for some specified $D$), and $P$ is a 
product of non-singular polynomials. This conditional formulation of Theorem 
\ref{biltobr} resembles the formulation of incidence geometry theorems in 
\cite{guth:math307j} and \cite{gk:erdos} that estimate $|P_r({\mathcal L})|$ 
conditional on having a favorable bound on the number of lines of 
${\mathcal L}$ that lie in $Z(P)$ (with the conditional bound being uniform 
over all $P$ satisfying appropriate properties).

Step six proves Theorem \ref{biltobr}. We first find (using algebraic 
topology) an appropriate polynomial $P$ such that the modified cells $O_i'$,
which were associated to $P$ in step two, contribute roughly equally to the
$L^p(Hdx)$ norm of the broad part of $Ef$ over $B_R$. This will then allow 
us to bound the contribution that comes from the cells by induction on the 
``size'' of the function $f$, and the contribution from the transverse tubes 
by induction on the radius $R$.

Step seven bounds the contribution from the tangential tubes. The tangential 
tubes gather rather close to an algebraic surface and dealing with them 
becomes roughly a two dimensional problem. In the incidence geometry 
setting, to bound the number of lines of ${\mathcal L}$ that lie in $Z(P)$,
one uses the available geometric properties of the lines of ${\mathcal L}$.
In the Fourier restriction setting, to bound the contribution from the
tubes tangent to $Z(P)$, one uses the available Kakeya information. In this 
paper, the Kakeya information we use are contained in Lemma \ref{Lemma3.6}, 
which was proved in \cite{guth:poly} by adapting Wolff's hairbrush argument 
to the polynomial partitioning setting.

Step eight inserts the bounds from step seven into Theorem \ref{biltobr}
obtaining estimates on the $L^p(Hdx)$ norm of the broad part of $Ef$ 
over $B_R$ (see Theorems \ref{thejapp} and \ref{themapp}).

Step nine uses parabolic scaling and induction on the radius $R$ to 
upgrade the estimates we now have on the $L^p(Hdx)$ norm of the broad part 
of $Ef$ on $B_R$ into estimates on the $L^p(Hdx)$ norm of $Ef$ itself,
proving Theorem \ref{mainjj}. 

Steps one to four are more or less identical to the treatment in 
\cite{guth:poly}. Steps five to nine differ from the treatment in 
\cite{guth:poly} in the following aspects.

The estimate on the broad part of $Ef$ (steps five, six, and eight) is 
established in \cite{guth:poly} in the non-weighted setting (i.e., $H(x)=1$ 
for all $x \in \mbb R^3$)  and for functions $f$ that obey the condition
\begin{displaymath}
\int_{B(\xi_0,R^{-1/2}) \cap S} |f(\xi)|^2 d\sigma(\xi) \leq \frac{1}{R}
\end{displaymath}
for all $\xi_0 \in S$. In this paper, we let $b \geq 1$ be a parameter and
establish the estimate on the broad part in the weighted setting for 
functions $f$ that obey the condition
\begin{displaymath}
\int_{B(\xi_0,R^{-1/2}) \cap S} |f(\xi)|^2 d\sigma(\xi)
\leq \frac{1}{R^{(b+1)/2}}
\end{displaymath}
for all $\xi_0 \in S$. To prove parts (i) and (iii) of Theorem \ref{mainjj},
we let $b=1$ later in the argument. To prove part (ii) of Theorem
\ref{mainjj}, we pick for $b$ a large value that depends on $\epsilon$. We
alert the reader to the fact that even though the estimate in the case
$\alpha=3/2$ is stated in both parts (i) and (ii) of Theorem \ref{mainjj},
its proof belongs to part (ii), and hence requires the large value of $b$.

In bounding the contribution from the tangential tubes (step seven), we
adjust the corresponding argument from \cite{guth:poly} to take into
consideration the dimensionality of the weight $H$, which is reflected in
the inequality $\int_{B(x_0,R)} H(x) dx \leq A_\alpha(H) R^\alpha$. This is
where we lower the value of the exponent $p$ from $13/4$ to
$2(4\alpha+3)/(2\alpha+3)$. 

The induction argument that \cite{guth:poly} uses to upgrade the estimate on 
the broad part of $Ef$ to an estimate on $Ef$ itself (step nine), assumes in 
the induction hypothesis that the desired estimate holds for all the 
surfaces in $\mbb R^3$ that have the same geometric properties as the given 
surface $S$, and then uses parabolic scaling. The induction hypothesis in
this paper assumes that the desired estimate holds not only for all the 
surfaces in $\mbb R^3$ that have the same geometric properties as $S$, but 
also for all the weights on $\mbb R^3$ of dimension $\alpha$. Then, to carry 
the induction out, the parabolic scaling argument gets adjusted accordingly. 
This turns out to be a little more involved than the corresponding argument 
in \cite{guth:poly}.

\subsection{Notation}

Throughout this paper, a closed ball in $\mbb R^3$ of center $x$ and radius
$r$ will be denoted by $B(x,r)$. A closed ball in $\mbb R^2$ of center
$\omega$ and radius $r$ will be denoted by $B^2(x,r)$. For example,
$B^2(0,1)$ is the closed unit ball in the plane.

If $A$ and $B$ are two positive quantities, then $A \lct B$ means that
$A \leq C B$ for a suitable constant $C$. If $A \lct B$ and $B \lct A$,
then we write $A \sim B$.

If $\phi$ is a function on $\mbb R^n$ and $r$ is a positive number, then
$\phi_r$ will denote the function defined by
$\phi_r(x)=r^{-n} \phi(r^{-1}x)$. Also, if $\Theta$ is a ball in $\mbb R^n$,
then $r \Theta$ will denote the ball of the same center as $\Theta$ and $r$
times the radius.

\subsection{Acknowledgment}

The author wishes to thank the anonymous referee for many insightful
suggestions concerning the structure of this paper.

\section{Global restriction estimates}

Some of our local restriction estimates can be turned into global estimates.
The tool for doing this is Tao's $\epsilon$-removal lemma from
\cite{tt:removal}. When one goes over the proof of Tao's $\epsilon$-removal
lemma (especially as presented in \cite{bg:bgmethod}), one sees that the
proof can be carried over to the weighted setting of this paper when $H$ is
the characteristic function of a set of the form
\begin{displaymath}
\Omega_{a,b} = \left\{
\begin{array}{l}
\hspace{-0.1in} \cup_{(m,n) \in \mbb Z^2} \Big( \mbb R \times [-1,1]^2
+ \big( 0,(\mbox{sgn} \, m) |m|^{1/a},(\mbox{sgn}\, n) |n|^{1/b} \big) \Big)
\mbox{ if $0< a, b \leq 1$,} \\
\hspace{-0.1in} \cup_{m \in \mbb Z} \Big( \mbb R \times [-1,1]^2
+ \big( 0,(\mbox{sgn} \, m) |m|^{1/a},0 \big) \Big)
\mbox{\hspace{0.41in} if $0 < a \leq 1$ and $b=0$,} \\
\hspace{-0.1in} \cup_{n \in \mbb Z} \Big( \mbb R \times [-1,1]^2
+ \big( 0,0,(\mbox{sgn} \, n) |n|^{1/b} \big) \Big)
\mbox{\hspace{0.528in} if $a=0$ and $0 < b \leq 1$,} \\
\end{array} \right.
\end{displaymath}
where $\mbox{sgn} \, m= m/|m|$ if $m \not= 0$, and $\mbox{sgn} \, 0= 1$. We
note that
\begin{displaymath}
\int_{B(x_0,R)} \chi_{\Omega_{a,b}}(x) dx \lct R^{1+a+b}
\end{displaymath}
for all $x_0 \in \mbb R^3$ and $R \geq 1$, so that
$A_\alpha(\chi_{\Omega_{a,b}}) \lct 1$ if $a+b=\alpha-1$. Also,
$\Omega_{1,1}= \mbb R^3$.

The restriction conjecture (in its global form) in $\mbb R^3$ asserts that
the estimate
\begin{displaymath}
\| E f \|_{L^p(\mbb R^3)} \leq C(p,q,S) \| f \|_{L^q(S)}
\end{displaymath}
holds whenever $p>3$, $(2/p)+(1/q) \leq 1$, and $f \in L^q(S)$. This 
estimate immediately implies that the Fourier transform of any function
$f \in L^{(3/2)-\epsilon}(\mbb R^3)$ can be restricted to $S$. So long as
this conjecture remains unsolved, it is natural to ask if there is a set
$\Omega$ in $\mbb R^3$ of infinite Lebesgue measure such that the Fourier
transform of any function $f \in L^{(3/2)-\epsilon}(\Omega)$ can be
restricted to $S$. Taking $\Omega=\Omega_{1/4,1/4}$, part (i) of the
following corollary tells us that this is indeed the case.

\begin{coro}
\label{global}
Let $S \subset \mbb R^3$ be a compact $C^\infty$ surface with strictly
positive second fundamental form.

{\rm (i)} Suppose $3/2 \leq \alpha < 5/2$, $p_0=4(4\alpha+3)/(10\alpha+3)$,
and $p > 2(4\alpha+3)/(2\alpha+3)$. Then
\begin{displaymath}
\| E f \|_{L^p(\Omega_{a,b})} \leq C(\alpha,p,S) \| f \|_{L^{p_0'}(S)}
\end{displaymath}
for all $f \in L^{p_0'}(S)$ provided $a+b=\alpha-1$, where $p_0'$ is the
exponent conjugate to $p_0$.

{\rm (ii)} Suppose $5/2 \leq \alpha \leq 3$, $1 \leq p_0 < 13/(2+2\alpha)$,
and $p > 13/4$. Then
\begin{displaymath}
\| E f \|_{L^p(\Omega_{a,b})} \leq C(\alpha,p_0,p,S) \| f \|_{L^{p_0'}(S)}
\end{displaymath}
for all $f \in L^{p_0'}(S)$ provided $a+b=\alpha-1$.

{\rm (iii)} We have
\begin{displaymath}
\| E f \|_{L^p(\mbb R^3)} \leq C(p,q,S) \| f \|_{L^q(S)}
\end{displaymath}
whenever $p>13/4$, $(2/p)+(1/q) < 1$, and $f \in L^q(S)$.
\end{coro}

Part (iii) of Corollary \ref{global} is a modest improvement on Guth's
global restriction estimate from \cite{guth:poly}; in Guth's theorem, the
norm on the right-hand side of the inequality is an $L^\infty$ norm. At any
rate, part (iii) of Corollary \ref{global} proves a restriction theorem in
$\mbb R^3$ for $p > 3.25$ in the full conjectured range of exponents up to
the sharp line $(2/p)+(1/q)= 1$. Also, part (i) of Corollary \ref{global}
proves the global version of (\ref{3ptfourteen}):
\begin{displaymath}
\| E f \|_{L^p(\Omega_2)}= \| E f \|_{L^p(\Omega_{1,0})}
\lct \| f \|_{L^{44/21}(S)} \hspace{0.44in} (p > 22/7).
\end{displaymath}

\section{Applications in geometric measure theory}

Fourier restriction theory has well-known and important implications to an
area of study that lies on the boundary between harmonic analysis and
geometric measure theory, and revolves around Falconer's distance set
conjecture.

Let $K$ be a compact subset of $\mbb R^n$, $n \geq 2$. The distance set of
$K$ is defined as
\begin{displaymath}
\Delta(K)= \{ |x-y| : x, y \in K \}.
\end{displaymath}
Falconer's conjecture asserts that if $K$ has Hausdorff dimension greater
than $n/2$, then $\Delta(K)$ has positive (one-dimensional) Lebesgue
measure.

Falconer initiated the study of the connection between Hausdorff dimension
and distance sets in~\cite{f:distsets} and used Fourier analysis (via
potential theory) to show that if the Hausdorff dimension of $K$ is greater
than $(n+1)/2$, then $\Delta(K)$ has positive Lebesgue measure. The Fourier
analytic approach to studying this problem was developed further by Mattila
in~\cite{pm:distsets}, where getting information about the Lebesgue measure
of $\Delta(K)$ was linked to obtaining favorable decay estimates as
$R \to \infty$ on integrals of the form
\begin{displaymath}
\int |\widehat{\mu}(R \, \xi)|^2 d\sigma_{n-1}(\xi)
\end{displaymath}
for $\mu \in M(\mbb R^n)$, where $\sigma_{n-1}$ is the surface measure on
the unit sphere $\mbb S^{n-1} \subset \mbb R^n$ and $M(\mbb R^n)$ is the
space of all complex Borel measures on $\mbb R^n$. Of course, favorable
decay estimates on such integrals cannot be obtained for general
$\mu \in M(\mbb R^n)$. In fact, on one end of the spectrum we have Dirac
measures, for which the above integral is equal to
$\sigma_{n-1}(\mbb S^{n-1})$ for all $R$. On the other end of the spectrum
we have the absolutely continuous measures with Schwartz densities, whose
Fourier transforms decay like $C_N R^{-N}$ for any positive integer $N$.

Suppose $K \subset \mbb R^n$ is compact and has Hausdorff dimension
$\alpha$. It is well known (e.g., see \cite{tw:book}) that if
$\alpha < \beta < n$, then $K$ supports a probability measure $\mu$ that
satisfies
\begin{displaymath}
\mu(B(x,r)) \leq C \, r^\beta
\end{displaymath}
for a suitable constant $C$ and all $x \in \mbb R^n$ and $r > 0$. It is also
well known that this condition implies that the measure's
$\alpha$-dimensional energy $I_\alpha(\mu)$, defined as
\begin{equation}
\label{defofenergy}
I_\alpha(\mu)= \int \int \frac{1}{|x-y|^\alpha} d\mu(x) d\mu(y),
\end{equation}
is finite. The $\alpha$-dimensional energy has the following Fourier
representation:
\begin{displaymath}
I_\alpha(\mu)= c_\alpha \int_{\mbb R^n} |\widehat{\mu}(\eta)|^2
|\eta|^{\alpha-n} d\eta,
\end{displaymath}
where $c_\alpha$ is a constant that depends only on $\alpha$ and $n$.
Moving to polar coordinates, we see that
\begin{displaymath}
\int_{\mbb R^n} |\widehat{\mu}(\eta)|^2 |\eta|^{\alpha-n} d\eta
= \int_0^\infty
  \Big( \int |\widehat{\mu}(R \, \xi)|^2 d\sigma_{n-1}(\xi) \Big)
  R^{\alpha-1} dR,
\end{displaymath}
and so we expect the finiteness of the $\alpha$-dimensional energy of $\mu$
to lead to some control over
$\int |\widehat{\mu}(R \, \xi)|^2 d\sigma_{n-1}(\xi)$ as $R \to \infty$.
Naturally, the tighter the control we have over these integrals, the better
the result we obtain on Falconer's conjecture by using the methods of
\cite{pm:distsets}.

In view of Mattila's work, Bourgain brought restriction theory into the
picture in~\cite{jb:dist}, where he showed that Falconer's $(n+1)/2$ result
follows from the Tomas-Stein restriction estimate, and used the better
restriction estimates that were available in dimensions two and three to
improve on Falconer's result. Bourgain showed that if $K \subset \mbb R^2$
has Hausdorff dimension greater than $13/9$, then $\Delta(K)$ has positive
Lebesgue measure, and that if $K \subset \mbb R^3$ has Hausdorff dimension
greater than $1091/546=1.998...$, then $\Delta(K)$ has positive Lebesgue
measure.

The next improvement came in~\cite{w:circdecay}, where Wolff showed that if
$K \subset \mbb R^2$ has dimension greater than $4/3$, then $\Delta(K)$ has
positive Lebesgue measure. Wolff got this result in the plane by obtaining
a sharp (up to the endpoint\footnote{It is not known whether the estimate
(\ref{l2wolff}) holds without the $R^\epsilon$ factor.}) decay estimate on
the circular means $\int |\widehat{\mu}(R \, \xi)|^2 d\sigma_1(\xi)$. Wolff
proved that if $1 \leq \alpha < 2$, and $\mu$ is a positive measure in
$M(\mbb R^2)$ with support in the unit disc and finite $\alpha$-dimensional
energy, then
\begin{equation}
\label{l2wolff}
\int |\widehat{\mu}(R \, \xi)|^2 d\sigma_1(\xi)
\leq C_\epsilon \frac{R^\epsilon}{R^{\frac{\alpha}{2}}} I_\alpha(\mu)
\end{equation}
for all $R \geq 1$. Wolff also considered the $L^q$ circular means
$\int |\widehat{\mu}(R \, \xi)|^q d\sigma_1(\xi)$ and observed that for
$q>2$ one cannot do better than interpolating between the above $L^2$
estimate and the trivial $L^\infty$ estimate. The case $1 \leq q < 2$,
however, presented a different challenge. The only estimate in this case
came from applying H\"{o}lder's inequality and using (\ref{l2wolff}), and it
remains an open problem to determine whether or not better estimates are
available. As was also explained in~\cite{w:circdecay}, estimates for the
$L^1$ means are particularly interesting as they are related to the open
problem of evaluating the infimum of the Hausdorff dimension of
$\beta$-sets. We refer the reader to~\cite{w:circdecay}
(and~\cite{bv:randomised}) for more details.

Using Tao's bilinear restriction estimate~\cite{tt:paraboloid}, Erdo\v{g}an
studied the same problem in higher dimensions and proved in
\cite{mbe:birestfal} that if \footnote{There are available estimates for the
$L^2$ spherical means of the Fourier transforms of such measures in the
complementary range $\alpha \in (0,n/2) \cup ((n+2)/2,n)$, which are known
to be sharp (up to the endpoint) only for $0 < \alpha \leq (n-1)/2$. In
dimension $n=2$, however, these estimates are known to be sharp (again up to
the endpoint) for all $0 < \alpha < 2$; see~\cite{pm:distsets},
~\cite{w:circdecay}, and~\cite{ps:sphericalavg}.}
$n/2 \leq \alpha \leq (n+2)/2$, and $\mu \in M(\mbb R^n)$ is positive with
$\mbox{supp} \, \mu$ contained in the unit ball and $I_\alpha(\mu)< \infty$,
then
\begin{equation}
\label{l2erdogan}
\int |\widehat{\mu}(R \, \xi)|^2 d\sigma_{n-1}(\xi) \leq C_\epsilon
\frac{R^\epsilon}{R^{\frac{\alpha}{2}+\frac{n}{4}-\frac{1}{2}}}
I_\alpha(\mu)
\end{equation}
for all $R \geq 1$. This estimate gives the currently best known result on
Falconer's conjecture: if $K \subset \mbb R^n$ has Hausdorff dimension
greater that $(n/2)+(1/3)$, then $\Delta(K)$ has positive Lebesgue measure.
Once again, for $1 \leq q < 2$, the best known estimate on the $L^q$
spherical means is the one we get from H\"{o}lder's inequality and
(\ref{l2erdogan}). For example, in dimension $n=3$, we have
\begin{displaymath}
\Big( \int |\widehat{\mu}(R \, \xi)|^q d\sigma_2(\xi) \Big)^{1/q}
\leq C_\epsilon \frac{R^\epsilon}{R^{\frac{\alpha}{4}+\frac{1}{8}}}
\sqrt{I_\alpha(\mu)}
\end{displaymath}
for all $R \geq 1$, provided $3/2 \leq \alpha \leq 5/2$.

As a corollary to Theorem \ref{mainjj}, we obtain the following result
concerning the $L^q$ norm of $\widehat{\mu}(R \, \cdot)$ on $S$ for
$R \geq 1$.

\begin{coro}
\label{decaysphm}
Let $S \subset \mbb R^3$ be a compact $C^\infty$ surface with strictly
positive second fundamental form, and $\mu$ be a positive measure in
$M(\mbb R^3)$ with support in the closed unit ball and finite
$\alpha$-dimensional energy.

{\rm (i)} Suppose $3/2 \leq \alpha < 5/2$, $p=2(4\alpha+3)/(2\alpha+3)$, and
$p_0=2p/(2p-3)=4(4\alpha+3)/(10\alpha+3)$. Then to every $\epsilon > 0$
there is a constant $C_\epsilon(\alpha,S)$ such that
\begin{displaymath}
\| \widehat{\mu}(R \, \cdot) \|_{L^{p_0}(S)}
\leq C_\epsilon(\alpha,S) R^\epsilon R^{-\alpha/p} \sqrt{I_\alpha(\mu)}
\end{displaymath}
for all $R \geq 1$.

{\rm (ii)} Suppose $3/2 \leq \alpha < 2$. Then to every $\epsilon > 0$
there is a constant $C_\epsilon(\alpha,S)$ such that
\begin{displaymath}
\| \widehat{\mu}(R \cdot) \|_{L^2(S)}
\leq C_\epsilon(\alpha,S) R^\epsilon R^{-(\alpha/4)-(1/8)}
     \sqrt{I_\alpha(\mu)}
\end{displaymath}
for all $R \geq 1$.

{\rm (iii)} Suppose $5/2 \leq \alpha < 13/5$ and
$1 \leq p_0 < 13/(2+2\alpha)$. Then to every $\epsilon >0$ there is a
constant $C_\epsilon(\alpha,p_0,S)$ such that
\begin{displaymath}
\| \widehat{\mu}(R \, \cdot) \|_{L^{p_0}(S)} \leq
C_\epsilon(\alpha,p_0,S) R^\epsilon R^{-4\alpha/13} \sqrt{I_\alpha(\mu)}
\end{displaymath}
for all $R \geq 1$.
\end{coro}

The estimates in parts (i) and (iii) of Corollary \ref{decaysphm} are new
and form one of the main results of this paper. This is the first time an
estimate on the $L^1$ spherical means of $\widehat{\mu}$ goes beyond what is
known on the $L^2$ spherical means.

Part (iii) is, in fact, true for $5/2 \leq \alpha < 3$. We state it as
above, because in the regime $13/5 \leq \alpha < 3$ it becomes inferior to
the known estimate $\| \widehat{\mu}(R \, \cdot) \|_{L^2(S)} \lct$
$R^{(1-\alpha)/2 } \sqrt{I_\alpha(\mu)}$ (see \cite{ps:sphericalavg} or
\cite{tw:book}).

As we mentioned before the statement of the corollary, the result of part
(ii) of Corollary \ref{decaysphm} has previously been obtained in
\cite{mbe:birestfal} by using Tao's bilinear restriction estimate from
\cite{tt:paraboloid}. It gives the best known result on Falconer's
conjecture in $\mbb R^3$, and now has a proof that is based on polynomial
partitioning.

After this paper was written, the author learned about the paper
\cite{rogerluca}, which obtains new results about the decay rate of
$\| \widehat{\mu}(R \, \cdot) \|_{L^2(\mbb S^{n-1})}$. In dimension $n=3$,
the results of \cite{rogerluca} are better than 
Corollary
\ref{decaysphm} when $\alpha > (8+\sqrt{85})/7 \approx 2.46$ and $S$ is the
unit sphere $\mbb S^2$.

\section{An estimate for exponential sums}

We can impose a slightly different condition on the measure than having a
finite $\alpha$-dimensional energy. Namely,
\begin{equation}
\label{condforexpsum}
\sup_{x \in \mbb R^3} \sup_{r>0} \frac{\mu(B(x,r))}{r^\alpha} < \infty.
\end{equation}
This condition on the measure is morally stronger than having finite
$\alpha$-dimensional energy. In fact, a result that the author learned from
Wolff's paper \cite{w:circdecay} (which is also stated below as Lemma
\ref{soilemma}) says that, for $R \geq 1$, a positive compactly supported
measure $\mu \in M(\mbb R^3)$ with finite $\alpha$-dimensional energy can be
decomposed as a sum of $O(1+\log R)$ measures $\mu_j$ satisfying
\begin{equation}
\label{morallystronger}
\sup_{x \in \mbb R^3} \sup_{r \geq R^{-1}} \frac{\mu(B(x,r))}{r^\alpha}
< \infty.
\end{equation}
Under (\ref{condforexpsum}), $\mu$ itself satisfies (\ref{morallystronger})
with a uniform bound for all $R \geq 1$.

The next corollary says that under (\ref{condforexpsum}), Theorem
\ref{mainjj} gives $L^p(\mu)$ bounds on exponential sums of the form
$\sum_{l=1}^N a_l e^{2\pi i R w_l \cdot x}$, where $w_1, \ldots, w_N \in S$
are $R^{-1}$-separated.

\begin{coro}
\label{expsum}
Let $S \subset \mbb R^3$ be a compact $C^\infty$ surface with strictly
positive second fundamental form, and $\mu$ be a positive measure in
$M(\mbb R^3)$ with support in the closed unit ball. Also, let
\begin{displaymath}
{\mathcal C}_\alpha(\mu)
= \sup_{x \in \mbb R^3} \sup_{r>0} \frac{\mu(B(x,r))}{r^\alpha}.
\end{displaymath}

{\rm (i)} Suppose $3/2 \leq \alpha < 5/2$ and $p=2(4\alpha+3)/(2\alpha+3)$.
Then to every $\epsilon > 0$ there is a constant $C_\epsilon(\alpha,S)$ such
that
\begin{displaymath}
\int \Big| \sum_{l=1}^N a_l e^{2\pi i R w_l \cdot x} \Big|^p d\mu(x)
\leq C_\epsilon(\alpha,S) \frac{R^\epsilon R^{2p}}{R^{\alpha+3}}
     {\mathcal C}_\alpha(\mu) \Big( \sum_{l=1}^N |a_l|^2 \Big)^{3/2}
                              \Big( \max_l |a_l| \Big)^{p-3}
\end{displaymath}
whenever $R \geq 1$, $w_1, \ldots, w_N \in S$ are $R^{-1}$-separated, and
$a_1, \ldots, a_N \in \mbb C$.

{\rm (ii)} Suppose $5/2 \leq \alpha \leq 3$,
$2 \leq \gamma < (11/2) - \alpha$, and $p=13/4$. Then to every
$\epsilon > 0$ there is a constant $C_\epsilon(\alpha,\gamma,S)$ such that
\begin{eqnarray*}
\lefteqn{\int \Big| \sum_{l=1}^N a_l e^{2\pi i R w_l \cdot x} \Big|^p
         d\mu(x)} \\
& \leq & C_\epsilon(\alpha,\gamma,S)
         \frac{R^\epsilon R^{2p}}{R^{\alpha+\gamma}}
         {\mathcal C}_{\alpha}(\mu)
         \Big( \sum_{l=1}^N |a_l|^2 \Big)^{\gamma/2}
         \Big( \max_l |a_l| \Big)^{p-\gamma}
\end{eqnarray*}
whenever $R \geq 1$, $w_1, \ldots, w_N \in S$ are $R^{-1}$-separated, and
$a_1, \ldots, a_N \in \mbb C$.
\end{coro}

Parts (i) and (ii) of Corollary \ref{expsum} are sharp (up to the
$R^\epsilon$ factor) in the following sense. Take $\mu$ to be the
restriction of an Ahlfors $\alpha$-regular measure $\nu$ to the unit ball,
$a_l=1$ for all $l$, and $N \sim R^2$. Being Ahlfors $\alpha$-regular means
that there are positive constants $C_1$ and $C_2$ such that
$C_1 r^\alpha \leq \nu(B(x,r)) \leq C_2 r^\alpha$ for all $x \in \mbb R^3$
and $r > 0$. Then ${\mathcal C}_\alpha(\mu) \leq C_2$ and Corollary
\ref{expsum} implies that
\begin{displaymath}
\frac{R^{2p}}{R^\alpha} \lct \int_{B(0,cR^{-1})}
\Big| \sum_{l=1}^N e^{2\pi i R w_l \cdot x} \Big|^p d\mu(x)
\leq \int \Big| \sum_{l=1}^N e^{2\pi i R w_l \cdot x} \Big|^p d\mu(x)
\lct R^\epsilon \frac{R^{2p}}{R^\alpha},
\end{displaymath}
where $c$ is an appropriately small constant. For example, when $\alpha=3/2$
this becomes
\begin{displaymath}
R^{9/2}
\lct \int \Big| \sum_{l=1}^N e^{2\pi i R w_l \cdot x} \Big|^3 d\mu(x)
\lct R^\epsilon R^{9/2}.
\end{displaymath}

\section{Preliminaries for the proofs of the corollaries}

The global restriction estimates of Corollary \ref{global} will be proved by
combining the local estimates of Theorem \ref{mainjj} with Tao's
$\epsilon$-removal lemma from \cite{tt:removal}. To apply the
$\epsilon$-removal lemma, however, it will be convenient to free the
estimates in parts (i) and (iii) of Theorem \ref{mainjj} from the
$L^\infty$-norm of $f$. It will also be convenient to write the estimates in
their dual form. This is the goal of the following theorem; which will also
be important to the proof of Corollary \ref{decaysphm}.

\begin{thm}
\label{dualform}
Let $S \subset \mbb R^3$ be a compact $C^\infty$ surface with strictly
positive second fundamental form, and $H$ be a weight of dimension $\alpha$
with $A_\alpha(H) \leq 1$.

{\rm (i)} Suppose $3/2 \leq \alpha < 5/2$, $p=2(4 \alpha +3)/(2 \alpha +3)$,
and $p_0=4(4 \alpha +3)/(10 \alpha +3)$. Then to every $\epsilon > 0$
there is a constant $C_\epsilon(\alpha,S)$ such that
\begin{displaymath}
\| {\mathcal R} f \|_{L^{p_0}(S)}
\leq C_\epsilon(\alpha,S) R^\epsilon \| f \|_{L^{p'}(\chi_{B(0,R)}Hdx)}
\end{displaymath}
whenever $R \geq 1$ and $f \in L^{p'}(\chi_{B(0,R)}Hdx)$, where
${\mathcal R} f = \widehat{f H} \big|_S$ and $p'$ is the exponent conjugate
to $p$.

{\rm (ii)} Suppose $5/2 \leq \alpha \leq 3$, $p=13/4$, and
$1 \leq p_0 < 13/(2+2\alpha)$. Then to every $\epsilon > 0$ there is a
constant $C_\epsilon(\alpha,p_0,S)$ such that
\begin{displaymath}
\| {\mathcal R} f \|_{L^{p_0}(S)}
\leq C_\epsilon(\alpha,p_0,S) R^\epsilon \| f \|_{L^{p'}(\chi_{B(0,R)}Hdx)}
\end{displaymath}
whenever $R \geq 1$ and $f \in L^{p'}(\chi_{B(0,R)}Hdx)$, where
${\mathcal R} f = \widehat{f H} \big|_S$ and $p'=13/9$.
\end{thm}

\begin{proof}
Since $A_\alpha(H) \leq 1$, parts (i) and (iii) of Theorem \ref{mainjj} can
be combined as
\begin{displaymath}
\int_{B(0,R)} |E g(x)|^{\bar{p}} H(x)dx \lct R^\epsilon
\| g \|_{L^2(S)}^{\bar{\gamma}}
\| g \|_{L^\infty(S)}^{\bar{p}-\bar{\gamma}}
\end{displaymath}
for all $g \in L^\infty(S)$, where
\begin{displaymath}
\bar{p}= \left\{ \begin{array}{ll}
                 2(4\alpha+3)/(2\alpha+3) &
                 \mbox{if $\frac{3}{2} \leq \alpha < \frac{5}{2}$,} \\ \\
                 13/4 & \mbox{if $\frac{5}{2} \leq \alpha \leq 3$,}
                 \end{array} \right.
\hspace{0.25in} \mbox{and} \hspace{0.25in}
\bar{\gamma}= \left\{ \begin{array}{ll}
                      3      &
                    \mbox{if $\frac{3}{2} \leq \alpha < \frac{5}{2}$,} \\ \\
                      \gamma & \mbox{if $\frac{5}{2} \leq \alpha \leq 3$.}
                 \end{array} \right.
\end{displaymath}
By the duality relation of the Fourier transform and H\"{o}lder's
inequality, the above estimate tells us that
\begin{eqnarray}
\label{dualityrelation}
\Big| \int {\mathcal R} f(\xi) g(\xi) d\sigma(\xi) \Big|
& \leq & \| f \|_{L^{\bar{p}'}(\chi_{B(0,R)}Hdx)}
         \| Eg \|_{L^{\bar{p}}(\chi_{B(0,R)}Hdx)} \nonumber \\
& \lct & R^{\epsilon/\bar{p}} \| f \|_{L^{\bar{p}'}(\chi_{B(0,R)}Hdx)}
         \| g \|_{L^2(S)}^{\bar{\gamma}/\bar{p}}
         \| g \|_{L^\infty(S)}^{1-(\bar{\gamma}/\bar{p})}
\end{eqnarray}
for all $f \in L^{\bar{p}'}(\chi_{B(0,R)}Hdx)$ and $g \in L^\infty(S)$.

We will now use (\ref{dualityrelation}) to prove the theorem. We will do
this by estimating the $\sigma$-measure of the sets
$\{ \xi \in S : |{\mathcal R}f(\xi)| > \lambda \}$ for
$0 < \lambda \leq \| f \|_{L^1(\chi_{B(0,R)}Hdx)}$. For such $\lambda$ and
for $l \in \mbb N$, we set
\begin{displaymath}
A_l= A_l(\lambda)=\{ \xi \in S : 2^{l-1} \lambda < |{\mathcal R}f(\xi)|
                                 \leq 2^l \lambda \}.
\end{displaymath}
Clearly,
$\{ \xi \in S : |{\mathcal R}f(\xi)| > \lambda \} \subset \cup_l A_l$.
Inserting $\overline{{\mathcal R}f(\xi)} \, \chi_{A_l}(\xi)$ for $g(\xi)$ in
(\ref{dualityrelation}), we obtain
\begin{displaymath}
\Big( \int_{A_l} |{\mathcal R}f(\xi)|^2 d\sigma(\xi)
\Big)^{1-\frac{\bar{\gamma}}{2\bar{p}}}
\lct R^{\epsilon/\bar{p}} \| f \|_{L^{\bar{p}'}(\chi_{B(0,R)}Hdx)}
     (2^l \lambda)^{1-(\bar{\gamma}/\bar{p})},
\end{displaymath}
which implies
\begin{displaymath}
\sigma(A_l)
\lct R^{2\epsilon/(2\bar{p}-\bar{\gamma})}
\| f \|_{L^{\bar{p}'}(\chi_{B(0,R)}Hdx)}^{2\bar{p}/(2\bar{p}-\bar{\gamma})}
(2^l \lambda)^{-2\bar{p}/(2\bar{p}-\bar{\gamma})}.
\end{displaymath}
Since $3 < \bar{p} \leq 13/4$ and $2 \leq \bar{\gamma} \leq 3$, it
follows that
\begin{displaymath}
\sigma(\{ \xi \in S : |{\mathcal R}f(\xi)| > \lambda \})
\leq \sum_l \sigma(A_l)
\lct \left(
\frac{R^{\epsilon/\bar{p}} \| f\|_{L^{\bar{p}'}(\chi_{B(0,R)}Hdx)}}{\lambda}
\right)^{p_0},
\end{displaymath}
where $p_0=2\bar{p}/(2\bar{p}-\bar{\gamma})$. Of course, we also have the
trivial bound
\begin{displaymath}
\sigma(\{ \xi \in S : |{\mathcal R}f(\xi)| > \lambda \}) \leq \sigma(S)
\lct 1.
\end{displaymath}

We can now bound our integral. We let
\begin{displaymath}
\lambda_0=R^{\epsilon/\bar{p}} \| f \|_{L^{\bar{p}'}(\chi_{B(0,R)}Hdx)}
\hspace{0.25in} \mbox{and} \hspace{0.25in}
\lambda_1=\| f \|_{L^1(\chi_{B(0,R)}Hdx)},
\end{displaymath}
and we observe that
\begin{displaymath}
\int_0^{\lambda_1} \sigma(\{ \xi \in S : |{\mathcal R}f(\xi) > \lambda \})
\lambda^{p_0-1} d\lambda \lct \int_0^{\lambda_0} \lambda^{p_0-1} d\lambda
= \frac{\lambda_0^{p_0}}{p_0}
\end{displaymath}
if $\lambda_1 \leq \lambda_0$, and
\begin{eqnarray*}
\int_0^{\lambda_1} \sigma(\{ \xi \in S : |{\mathcal R}f(\xi) > \lambda \})
\lambda^{p_0-1} d\lambda
& \lct & \int_0^{\lambda_0} \lambda^{p_0-1} d\lambda
+ \lambda_0^{p_0} \int_{\lambda_0}^{\lambda_1} \frac{d\lambda}{\lambda} \\
& \lct & \lambda_0^{p_0} \Big( 1+\log \frac{\lambda_1}{\lambda_0} \Big)
\end{eqnarray*}
if $\lambda_1 > \lambda_0$. But, by H\"{o}lder's inequality,
$\lambda_1 \leq |B(0,R)|^{1/\bar{p}} R^{-\epsilon/\bar{p}} \lambda_0 \lct
R \lambda_0$, so
\begin{displaymath}
\int_0^{\lambda_1} \sigma(\{ \xi \in S : |{\mathcal R}f(\xi) > \lambda \})
\lambda^{p_0-1} d\lambda \lct \lambda_0^{p_0} (1 + \log R).
\end{displaymath}
Thus
\begin{displaymath}
\Big( \int |{\mathcal R}f(\xi)|^{p_0} d\sigma(\xi) \Big)^{1/p_0}
\lct R^\epsilon \| f \|_{L^{\bar{p}'}(\chi_{B(0,R)}Hdx)}.
\end{displaymath}

When $3/2 \leq \alpha < 5/2$, we have $\bar{p}=2(4\alpha+3)/(2\alpha+3)$ and
$\bar{\gamma}=3$, so that $p_0=4(4\alpha+3)/(10\alpha+3)$. This proves part
(i) of the theorem.

When $5/2 \leq \alpha \leq 3$, we have $\bar{p}=13/4$ and
$2 \leq \bar{\gamma}=\gamma < (11/2)-\alpha$, so that
$13/9 \leq p_0 < 13/(2+2\alpha)$. This proves part (ii).
\end{proof}

We now prove a lemma that is important for our results concerning the decay
of the $L^q(S)$ means of Fourier transforms of measures, as well as for our
estimate on exponential sums. This lemma will allow
us to bound $\| Ef(R \, \cdot) \|_{L^p(\mu)}$ by $\| Eg \|_{L^p(Hdx)}$ with
$|g| \leq |f|$, and for an appropriate weight $H$ that is supported in the
ball $B(0,2R)$.

\begin{lemma}
\label{bdmubyh}
Suppose $\mu \in M(\mbb R^3)$ is positive and supported in $B(0,1)$,
$0 < \alpha \leq 3$, $R \geq 1$, and
\begin{displaymath}
{\mathcal C}_{\alpha,R}(\mu)
= \sup_{x \in \mbb R^3} \sup_{r \geq R^{-1}} \frac{\mu(B(x,r))}{r^\alpha}.
\end{displaymath}
Then there is a weight $H$ (which depends on $R$) of dimension $\alpha$ such
that
\\
{\rm (i)} $A_\alpha(H) \leq |B(0,1)|$
\\
{\rm (ii)} to every function $f \in L^1(S)$ there is a function
$g \in L^1(S)$ such that $|g| \leq |f|$ and
\begin{displaymath}
\int |Ef(R x)|^p d\mu(x)
\leq C_p \frac{{\mathcal C}_{\alpha,R}(\mu)}{R^\alpha}
     \int_{B(0,2R)} |Eg(y)|^p H(y) dy
\end{displaymath}
for $p \geq 1$, where $C_p$ is a constant that only depends on $p$.
\end{lemma}

\begin{proof}
We let $\phi$ be a Schwartz function on $\mbb R^3$ such that
$|\phi| \geq 1$ on $S$ and $\mbox{supp} \, \widehat{\phi} \subset B(0,1)$,
and define the function $g$ on $S$ by $g=f/\phi$. Then $|g| \leq |f|$ and
\begin{displaymath}
\int |E f(R x)|^p d\mu(x) =
\int \big| \widehat{\phi} \ast \widehat{g d\sigma}(R x) \big|^p d\mu(x)
= \int \big| \widehat{\phi} \ast (E g)(R x) \big|^p d\mu(x).
\end{displaymath}
By H\"{o}lder's inequality, it follows that
\begin{eqnarray*}
\int |E f(R x)|^p d\mu(x)
& \leq & \| \widehat{\phi} \|_{L^1}^{p-1}
         \int |E g|^p \ast \big| \widehat{\phi} \, \big|(R x) d\mu(x) \\
&  =   & \| \widehat{\phi} \|_{L^1}^{p-1}
         \int \int |E g(y)|^p \big| \widehat{\phi}(Rx-y) \big| dy d\mu(x).
\end{eqnarray*}
Interchanging the order of integration, we arrive at the inequality
\begin{displaymath}
\int |E f(R x)|^p d\mu(x)
\leq \| \widehat{\phi} \|_{L^1}^{p-1} \| \widehat{\phi} \|_{L^\infty}
     {\mathcal C}_{\alpha,R}(\mu) R^{-\alpha} \int |E g(y)|^p H(y) dy,
\end{displaymath}
where $H(y)= \| \widehat{\phi} \|_{L^\infty}^{-1}
{\mathcal C}_{\alpha,R}(\mu)^{-1} R^\alpha \int |\widehat{\phi}(Rx-y)|
d\mu(x)$. Since $\widehat{\phi}$ as $\mu$ is supported in $B(0,1)$, it
follows that $H$ is supported in $B(0,1+R) \subset B(0,2R)$, and hence
\begin{displaymath}
\int |E f(R x)|^p d\mu(x)
\leq \| \widehat{\phi} \|_{L^1}^{p-1} \| \widehat{\phi} \|_{L^\infty}
     {\mathcal C}_{\alpha,R}(\mu) R^{-\alpha}
     \int_{B(0,2R)} |E g(y)|^p H(y) dy.
\end{displaymath}

It remains to show that $H$ is a weight of dimension $\alpha$ and
$A_\alpha(H) \leq |B(0,1)|$. Clearly,
\begin{displaymath}
H(y)= \| \widehat{\phi} \|_{L^\infty}^{-1} {\mathcal C}_{\alpha,R}(\mu)^{-1}
R^\alpha \int_{B(R^{-1} y,R^{-1})} |\widehat{\phi}(Rx-y)| d\mu(x) \leq 1
\end{displaymath}
for all $y$, where we have used the fact that
$\mu(B(R^{-1} y,R^{-1})) \leq {\mathcal C}_{\alpha,R}(\mu) R^{-\alpha}$.
Also,
\begin{eqnarray*}
\int_{B(x_0,r)} H(y) dy
& = & \| \widehat{\phi} \|_{L^\infty}^{-1} {\mathcal C}_{\alpha,R}(\mu)^{-1}
  R^\alpha \int \int \chi_{B(x_0,r)}(y) |\widehat{\phi}(Rx-y)| d\mu(x) dy \\
& = & \| \widehat{\phi} \|_{L^\infty}^{-1} {\mathcal C}_{\alpha,R}(\mu)^{-1}
  R^\alpha \int |\widehat{\phi}(u)| \int \chi_{B(x_0,r)}(Rx-u) d\mu(x) du \\
& = & \| \widehat{\phi} \|_{L^\infty}^{-1} {\mathcal C}_{\alpha,R}(\mu)^{-1}
   R^\alpha \int |\widehat{\phi}(u)|\, \mu(B((u+x_0) R^{-1},r R^{-1})) du \\
& \leq & \| \widehat{\phi} \|_{L^\infty}^{-1} \| \widehat{\phi} \|_{L^1}
         \, r^\alpha
\end{eqnarray*}
provided $r \geq 1$. Thus $A_\alpha(H) \leq
\| \widehat{\phi} \|_{L^\infty}^{-1} \| \widehat{\phi} \|_{L^1} \leq
|B(0,1)|$.
\end{proof}

\section{Proofs of the corollaries}

We are now in position to prove our global restriction estimates.

\begin{proof}[Proof of Corollary \ref{global}]
Let $3/2 \leq \alpha \leq 3$, and $H$ be a weight of dimension $\alpha$ with
$A_\alpha(H) \leq 1$. We combine parts (i) and (ii) of Theorem
\ref{dualform} as
\begin{equation}
\label{stop0}
\| {\mathcal R} f \|_{L^{p_0}(S)}
\lct R^\epsilon \| f \|_{L^s(\chi_{B(0,R)} H dx)}
\end{equation}
with the understanding that $p_0=4(4\alpha+3)/(10\alpha+3)$ and
$s'=2(4\alpha+3)/(2\alpha+3)$ if $3/2 \leq \alpha < 5/2$, and
$13/9 \leq p_0 < 13/(2+2\alpha)$ and $s'=13/4$ if $5/2 \leq \alpha \leq 3$.
We note that in both cases $1 < s \leq p_0 \leq 2$.

Following \cite{tt:removal}, we would like to upgrade (\ref{stop0}) to
become valid for all functions $f \in L^s(\chi_V H dx)$ whenever $V$ is a
union of a sparse family of balls. This means $V= \cup_{l=1}^N B(x_l,R)$
with $|x_l - x_m| \geq (RN)^C$ if $l \not= m$, where $C$ is a suitably large
constant.

For $f \in L^s(\chi_V H dx)$ and $1 \leq l \leq N$, we let $f_l$ be the
restriction of $f$ to $B(x_l,R)$, and we define the function
$g_l \in L^s(\chi_{B(0,R)} H dx)$ by $g_l(x)=f_l(x+x_l)$. Then
\begin{displaymath}
{\mathcal R} f_l(\xi)= \widehat{f_l H}(\xi) =
e^{-2 \pi i \xi \cdot x_l} \Big( g_l H(\cdot+x_l) \widehat{\Big) \;} (\xi)
\end{displaymath}
for all $\xi \in S$. Since $1 < p_0 \leq 2$ and the balls $B(x_l,R)$ are
sparse, Lemma 3.2 of \cite{tt:removal} as reformulated in \cite{bg:bgmethod}
(see inequality (11) on page 1289 of \cite{bg:bgmethod}) now tells us that
\begin{displaymath}
\| {\mathcal R} f \|_{L^{p_0}(S)}
= \Big\| \sum_{l=1}^N {\mathcal R} f_l \Big\|_{L^{p_0}(S)} \lct
\Big( \sum_{l=1}^N
\Big\| \Big( g_l H(\cdot+x_l) \widehat{\Big) \;} \Big\|_{L^{p_0}(S)}^{p_0}
\Big)^{1/p_0}.
\end{displaymath}
Since $A_\alpha(H(\cdot+x_l))=A_\alpha(H)$, (\ref{stop0}) gives
\begin{displaymath}
\Big\| \Big( g_l H(\cdot+x_l) \widehat{\Big) \;} \Big\|_{L^{p_0}(S)}
\lct R^\epsilon \| g_l \|_{L^s(\chi_{B(0,R)}H(\cdot+x_l)dx)}.
\end{displaymath}
But $\| g_l \|_{L^s(\chi_{B(0,R)}H(\cdot+x_l)dx)}=
\| f \|_{L^s(\chi_{B(x_l,R)}Hdx)}$, so
\begin{eqnarray*}
\| {\mathcal R} f \|_{L^{p_0}(S)}
& \lct & R^\epsilon \Big( \sum_{l=1}^N
         \| f \|_{L^s(\chi_{B(x_l,R)}Hdx)}^{p_0} \Big)^{1/{p_0}} \\
& \leq & R^\epsilon \Big( \sum_{l=1}^N
         \| f \|_{L^s(\chi_{B(x_l,R)}Hdx)}^s \Big)^{1/s} \\
&  =   & R^\epsilon \| f \|_{L^s(\chi_V Hdx)},
\end{eqnarray*}
where we have used the fact that $p_0 \geq s$.

The next step is to upgrade (\ref{stop0}) to become valid for all
$f \in L^s(\chi_E Hdx)$ whenever $E$ is a finite union of $c$-cubes. For
this we need the following lemma.

\begin{alphlemma}[Tao~\cite{tt:removal}]
\label{ccubes}
Suppose $E$ is the union of $c$-cubes and $0 < \delta < 1$. Then there are
$O(\delta^{-1} |E|^\delta)$ sets $V_k$ that cover $E$ such that each $V_k$
is a union of a sparse collections of balls of radius
$O(|E|^{C^{1/\delta}})$.
\end{alphlemma}

Given a function $f \in L^s(\chi_E Hdx)$, then writing $f=\sum_k f_k$ with
$f_k$ supported in $V_k$ and using Minkowski's inequality, we see that
\begin{displaymath}
\| {\mathcal R} f \|_{L^{p_0}(S)} \lct \delta^{-1}
|E|^\delta |E|^{\epsilon C^{1/\delta}} \| f \|_{L^s(\chi_E Hdx)}.
\end{displaymath}
Taking $\delta \sim 1/\log(1/\epsilon)$, this becomes
\begin{equation}
\label{estoncubes}
\| {\mathcal R} f \|_{L^{p_0}(S)}
\lct |E|^{C/\log(1/\epsilon)} \| f \|_{L^s(\chi_E Hdx)}
\end{equation}
provided $A_\alpha(H) \lct 1$.

We now take $H$ to be the characteristic function of $2\Omega_{a,b}$, where
$2\Omega_{a,b}$ is the same as $\Omega_{a,b}$ but with the cylinder
$\mbb R \times [-1,1]^2$ replaced by $\mbb R \times [-2,2]^2$, and proceed
as in \cite{bg:bgmethod}. As we mentioned before, the argument in
\cite{bg:bgmethod} is based on \cite{tt:removal}.

Suppose $1 \leq r < s$. Let $\tilde{E}$ be a subset of $2\Omega_{a,b}$ which
is a (possibly infinite) union of $c$-cubes, where $c$ is a constant that
will be determined later, and let
$f \in L^1(2\Omega_{a,b}) \cap L^r(2\Omega_{a,b})$ be a function that is
constant on each of the $c$-cubes of $\tilde{E}$, vanishes on
$2\Omega_{a,b} \setminus \tilde{E}$, and satisfies $\| f \|_{L^r} \leq 1$.
For $k \in \mbb Z$, we set $E_k= \{ 2^{-k-1} \leq |f| < 2^{-k} \}$ and
$f_k= \chi_{E_k} f$. Then each $E_k$ is a finite union of $c$-cubes, and
$2^{-kr} |E_k| \lct 1$ for all $k$. We note that, since $|E_k| \geq c^3$,
the last inequality implies that the set of $k$ for which
$E_k \not= \emptyset$ is bounded from below by a constant that depends on
$c$ and $r$. Applying (\ref{estoncubes}) with $H=\chi_{2\Omega_{a,b}}$,
$f=f_k$, and $E=E_k$, we get
\begin{eqnarray*}
\lefteqn{\| {\mathcal R} f_k \|_{L^{p_0}(S)}
         \; \lct \; |E_k|^{C/\log(1/\epsilon)}
                    \| f_k \|_{L^s(2\Omega_{a,b})}} \\
& & \; \lct \; 2^{-k} |E_k|^{C/\log(1/\epsilon)} |E_k|^{1/s}
    \; \lct \; 2^{-k} 2^{k r \, C/\log(1/\epsilon)} 2^{kr/s}.
\end{eqnarray*}
Summing over $k$, we arrive at
\begin{equation}
\label{cnstonetilde}
\| {\mathcal R} f \|_{L^{p_0}(S)} \lct 1
\end{equation}
provided
\begin{displaymath}
\frac{C}{\log(1/\epsilon)} + \frac{1}{s} < \frac{1}{r}.
\end{displaymath}

For $\tau \in [-c/2,c/2]^3$, we now let ${\mathcal L}_\tau$ be the
intersection of the lattice $c \, \mbb Z^3 + \tau$ with $\Omega_{a,b}$.
We suppose $\{ \lambda_n \}$ is a sequence in
$l^1({\mathcal L}_\tau) \cap l^r({\mathcal L}_\tau)$, and let $\chi_n$ be
the characteristic function of $[-c/2,c/2]^3+n$. If
$f=\sum_n \lambda_n \chi_n$, then (\ref{cnstonetilde}) tells us that
\begin{displaymath}
\| {\mathcal R} f \|_{L^{p_0}(S)}
\lct \Big( \sum_n |\lambda_n|^r \Big)^{1/r}.
\end{displaymath}
But $\widehat{f}(\xi)= \widehat{\chi_0}(\xi)
\sum_n \lambda_n e^{-2 \pi i \xi \cdot n}$, so, choosing $c$ small enough
for $|\widehat{\chi_0}|$ to be positive on $S$, we get
\begin{displaymath}
\left( \int \Big| \sum_{n \in {\mathcal L}_\tau}
\lambda_n e^{-2 \pi i \xi \cdot n} \Big|^{p_0} d\sigma(\xi) \right)^{1/p_0}
\lct \left( \sum_{n \in {\mathcal L}_\tau} |\lambda_n|^r \right)^{1/r}.
\end{displaymath}
Averaging over $\tau$ and letting $p$ be the exponent conjugate to $r$, we
get the estimate in part (i) of the corollary if $3/2 \leq \alpha < 5/2$,
and the estimate in part (ii) if $5/2 \leq \alpha \leq 3$.

When $\alpha=3$ and $a=b=1$, the estimate in part (ii) of the corollary
becomes
\begin{displaymath}
\| Ef \|_{L^p(\mbb R^3)}= \| Ef \|_{L^p(\Omega_{1,1})}
\lct \| f \|_{L^{p_0'}(S)}
\end{displaymath}
whenever $p > 13/4$ and $1 \leq p_0 < 13/8$. Interpolating this with the
trivial estimate $\| Ef \|_{L^\infty(\mbb R^3)} \leq \| f \|_{L^1(S)}$, we
prove part (iii) of the corollary.
\end{proof}

For the proof of Corollary \ref{decaysphm}, we will also need the following
lemma from \cite{w:circdecay} that connects the
${\mathcal C}_{\alpha,R}(\mu)$ of Lemma \ref{bdmubyh} to the
$\alpha$-dimensional energy of $\mu$.

\begin{alphlemma}[Wolff~\cite{w:circdecay}]
\label{soilemma}
Let $\mu \in M(\mbb R^3)$ be a positive measure with support in $B(0,1)$,
$0 < \alpha < 3$, and $R \geq 1$. Then we can decompose $\mu$ as a sum of
$O(1+\log R)$ measures $\mu_j$ so that for each $j$,
\begin{displaymath}
\| \mu_j \| \, {\mathcal C}_{\alpha,R}(\mu_j) \lct I_\alpha(\mu)
\end{displaymath}
with an implicit constant that depends only on $\alpha$.
\end{alphlemma}

\begin{proof}[Proof of Corollary \ref{decaysphm}]
As in the proof of Theorem \ref{dualform}, we let
\begin{displaymath}
\bar{p}= \left\{ \begin{array}{ll}
                 2(4\alpha+3)/(2\alpha+3) &
                 \mbox{if $\frac{3}{2} \leq \alpha < \frac{5}{2}$,} \\ \\
                 13/4 & \mbox{if $\frac{5}{2} \leq \alpha \leq 3$.}
                 \end{array} \right.
\end{displaymath}

Writing $\mu= \sum_j \mu_j$ as in Lemma \ref{soilemma}, we see by
H\"{o}lder's inequality that
\begin{eqnarray*}
\int |Ef(R x)| d\mu(x)
&   =  & \sum_j \int |Ef(R x)| d\mu_j(x) \\
& \leq & \sum_j \| \mu_j \|^{1-(1/\bar{p})}
         \Big( \int |Ef(R x)|^{\bar{p}} d\mu_j(x) \Big)^{1/\bar{p}} \\
& \leq & \| \mu \|^{1-(2/\bar{p})} \sum_j
     \Big( \| \mu_j \| \int |Ef(R x)|^{\bar{p}} d\mu_j(x) \Big)^{1/\bar{p}}
\end{eqnarray*}
for all $f \in L^1(S)$, where we have used the fact that $\bar{p} >2$. But
by Lemma \ref{bdmubyh} and Lemma \ref{soilemma}, for each such $f$ there is
a function $g$ with $|g| \leq |f|$ such that
\begin{eqnarray*}
\| \mu_j \| \int |Ef(R x)|^{\bar{p}} d\mu_j(x)
& \lct & \| \mu_j \| \, {\mathcal C}_{\alpha,R}(\mu_j) R^{-\alpha}
         \int_{B(0,2R)} |Eg(y)|^{\bar{p}} H(y) dy \\
& \lct & I_\alpha(\mu) R^{-\alpha} \int_{B(0,2R)} |Eg(y)|^{\bar{p}} H(y) dy.
\end{eqnarray*}
Summing over $j$, we get
\begin{eqnarray*}
\lefteqn{\int |Ef(R x)| d\mu(x)} \\
& \lct & (1+\log R) \| \mu \|^{1-(2/\bar{p})} I_\alpha(\mu)^{1/\bar{p}}
         R^{-\alpha/\bar{p}}
         \Big( \int_{B(0,2R)} |Eg(y)|^{\bar{p}} H(y) dy \Big)^{1/\bar{p}}.
\end{eqnarray*}
Since $\mbox{supp} \, \mu \subset B(0,1)$, we have
$\| \mu \|^2 \lct I_\alpha(\mu)$, and the above estimate becomes
\begin{displaymath}
\int |Ef(R x)| d\mu(x) \lct R^\epsilon
I_\alpha(\mu)^{1/2} R^{-\alpha/\bar{p}}
\Big( \int_{B(0,2R)} |Eg(y)|^{\bar{p}} H(y) dy \Big)^{1/\bar{p}}.
\end{displaymath}
By part (i) of Lemma \ref{bdmubyh}, we know that $A_\alpha(H) \lct 1$, so
Theorem \ref{dualform} (in its dual form) tells us that
\begin{displaymath}
\Big( \int_{B(0,2R)} |Eg(y)|^{\bar{p}} H(y) dy \Big)^{1/\bar{p}}
\lct R^\epsilon \| g \|_{L^{p_0'}(S)}.
\end{displaymath}
Thus
\begin{equation}
\label{lpbarineq}
\Big| \int \widehat{\mu}(R \, \xi) f(\xi) d\sigma(\xi) \Big|
\lct R^{2\epsilon} \frac{\sqrt{I_\alpha(\mu)}}{R^{\alpha/\bar{p}}} \,
\| f \|_{L^{p_0'}(S)}
\end{equation}
for all $f \in L^{p_0'}(S)$ and $R \geq 1$. By duality, (\ref{lpbarineq})
proves parts (i) and (iii) of Corollary~\ref{decaysphm}.

If we use part (ii) of Theorem \ref{mainjj} instead of Theorem
\ref{dualform}, and follow the same steps as in the proof of
(\ref{lpbarineq}), we arrive at the inequality
\begin{displaymath}
\Big| \int \widehat{\mu}(R \, \xi) f(\xi) d\sigma(\xi) \Big|
\lct R^{2\epsilon} \frac{\sqrt{I_\alpha(\mu)}}{R^{\alpha/3}}
R^{\frac{1}{12}(\alpha - \frac{3}{2})} \| f \|_{L^2(S)}
\end{displaymath}
for all $f \in L^2(S)$ and $R \geq 1$. Inserting
$\overline{\widehat{\mu}(R \, \xi)}$ for $f(\xi)$, the inequality becomes
\begin{displaymath}
\Big( \int |\widehat{\mu}(R \, \xi)|^2 d\sigma(\xi) \Big)^{1/2}
\lct R^{2\epsilon} \frac{\sqrt{I_\alpha(\mu)}}{R^{(\alpha/4)+(1/8)}},
\end{displaymath}
which is part (ii) of Corollary \ref{decaysphm}.
\end{proof}

We now move to prove our result on exponential sums.

\begin{proof}[Proof of Corollary \ref{expsum}]
Let $\bar{p}$ and $\bar{\gamma}$ be as in the proof of Theorem
\ref{dualform}. Since
${\mathcal C}_{\alpha,R}(\mu) \leq {\mathcal C}_\alpha(\mu)$ for all
$R \geq 1$, parts (i) and (iii) of Theorem \ref{mainjj} together with
Lemma \ref{bdmubyh} tell us that
\begin{displaymath}
\int |Ef(Rx)|^{\bar{p}} d\mu(x) \lct \frac{R^\epsilon}{R^\alpha}
{\mathcal C}_\alpha(\mu) \| f \|_{L^2(S)}^{\bar{\gamma}}
\| f \|_{L^\infty(S)}^{\bar{p}-\bar{\gamma}}
\end{displaymath}
for all $f \in L^\infty(S)$. If $F$ is an $L^\infty$ function on the
$1/R$-neighborhood of $S$, then the above estimate implies that
\begin{equation}
\label{thick}
\int |\widehat{F}(Rx)|^{\bar{p}} d\mu(x) \lct {\mathcal C}_\alpha(\mu)
\frac{R^\epsilon R^{\bar{\gamma}/2}}{R^{\alpha+\bar{p}}}
\| F \|_{L^2}^{\bar{\gamma}} \| F \|_{L^\infty}^{\bar{p}-\bar{\gamma}}.
\end{equation}
Proving this is a standard argument (e.g., see Proposition 4.3 of
\cite{taovv:bilinear}). For example, if $S$ is the unit sphere, then
\begin{eqnarray*}
\widehat{F}(Rx)
& = & \int_{1-\frac{1}{R} \leq |\xi| \leq 1+\frac{1}{R}}
      e^{-2 \pi i Rx \cdot \xi} F(\xi) d\xi \\
& = & \int_{1-\frac{1}{R}}^{1+\frac{1}{R}} \int_S
      e^{-2 \pi i Rx \cdot r \theta} F(r \theta) d\sigma(\theta) r^2 dr \\
& = & \int_{1-\frac{1}{R}}^{1+\frac{1}{R}} E \big( F(r \, \cdot) \big) (rRx)
      r^2 dr,
\end{eqnarray*}
so that
\begin{eqnarray*}
\big\| \widehat{F}(R \, \cdot) \big\|_{L^{\bar{p}}(\mu)}
& \leq & \int_{1-\frac{1}{R}}^{1+\frac{1}{R}}
   \big\| E \big( F(r \, \cdot) \big)(rR \, \cdot) \big\|_{L^{\bar{p}}(\mu)}
         r^2 dr \\
& \lct & {\mathcal C}_\alpha(\mu)^{\frac{1}{\bar{p}}}
         \frac{R^{\frac{\epsilon}{\bar{p}}}}{R^{\frac{\alpha}{\bar{p}}}}
         \int_{1-\frac{1}{R}}^{1+\frac{1}{R}}
         \| F(r \, \cdot) \|_{L^2(S)}^{\frac{\bar{\gamma}}{\bar{p}}}
         \| F(r \, \cdot) \|_{L^\infty(S)}^{1-\frac{\bar{\gamma}}{\bar{p}}}
         r^2 dr \\
& \lct & {\mathcal C}_\alpha(\mu)^{\frac{1}{\bar{p}}}
         \frac{R^{\frac{\epsilon}{\bar{p}}}}{R^{\frac{\alpha}{\bar{p}}}}
         \| F \|_{L^\infty}^{1-\frac{\bar{\gamma}}{\bar{p}}}
         \frac{R^{\frac{\bar{\gamma}}{2\bar{p}}}}{R}
\Big( \int_{1-\frac{1}{R}}^{1+\frac{1}{R}} \| F(r \, \cdot) \|_{L^2(S)}^2
      r^2 dr \Big)^{\frac{\bar{\gamma}}{2\bar{p}}} \\
&   =  & {\mathcal C}_\alpha(\mu)^{\frac{1}{\bar{p}}}
         \frac{R^{\frac{\epsilon}{\bar{p}}+\frac{\bar{\gamma}}{2\bar{p}}}}
              {R^{\frac{\alpha}{\bar{p}}+1}}
         \| F \|_{L^2}^{\frac{\bar{\gamma}}{\bar{p}}}
         \| F \|_{L^\infty}^{1-\frac{\bar{\gamma}}{\bar{p}}},
\end{eqnarray*}
where we have used H\"{o}lder's inequality and the fact that
$\bar{\gamma} < 2\bar{p}$.

Now suppose $w_1, \ldots, w_N \in S$ are such that
$|w_l-w_{l'}| \sim R^{-1}$. Let $\phi$ be a $C_0^\infty$ function on
$\mbb R^3$ with the property that $|\widehat{\phi}| \geq 1$ on the unit
ball. Then
\begin{eqnarray*}
\int \Big| \sum_{l=1}^N a_l e^{-2 \pi i R w_l \cdot x} \Big|^{\bar{p}}
d\mu(x)
& \leq & \int \Big| \sum_{l=1}^N a_l e^{-2 \pi i w_l \cdot (Rx)}
              \widehat{\phi_{R^{-1}}}(Rx) \Big|^{\bar{p}} d\mu(x) \\
&   =  & \int \Big| \sum_{l=1}^N \widehat{\psi_l}(Rx) \Big|^{\bar{p}}
         d\mu(x),
\end{eqnarray*}
where
$\widehat{\psi_l}(y)=a_l e^{-2\pi i w_l\cdot y} \widehat{\phi_{R^{-1}}}(y)$,
i.e.\ $\psi_l(\xi)= a_l \phi_{R^{-1}}(\xi-w_l)= a_l R^3 \phi(R(\xi - w_l))$.
Applying (\ref{thick}) with $F=\sum_{l=1}^N \psi_l$, we get
\begin{displaymath}
\int \Big| \sum_{l=1}^N a_l e^{-2 \pi i w_l \cdot x} \Big|^{\bar{p}} d\mu(x)
\lct {\mathcal C}_\alpha(\mu)
     \frac{R^\epsilon R^{\bar{\gamma}/2}}{R^{\alpha+\bar{p}}}
     \| F \|_{L^2}^{\bar{\gamma}}
     \| F \|_{L^\infty}^{\bar{p}-\bar{\gamma}}.
\end{displaymath}
Since $\phi$ is compactly supported and the $w_l$ are $R^{-1}$-separated, we
have
\begin{displaymath}
\| F \|_{L^2}^2 \lct R^3 \sum_{l=1}^N |a_l|^2
\hspace{0.25in} \mbox{and} \hspace{0.25in}
\| F \|_{L^\infty} \lct R^3 \max_l |a_l|.
\end{displaymath}
Thus
\begin{displaymath}
\int \Big| \sum_{l=1}^N a_l e^{-2 \pi i w_l \cdot x} \Big|^{\bar{p}} d\mu(x)
\lct {\mathcal C}_\alpha(\mu)
     \frac{R^\epsilon R^{2\bar{p}}}{R^{\alpha+\bar{\gamma}}}
     \Big( \sum_{l=1}^N |a_l|^2 \Big)^{\bar{\gamma}/2}
     \Big( \max_l |a_l| \Big)^{\bar{p}-\bar{\gamma}}.
\end{displaymath}
\end{proof}

\section{The wave packet decomposition}

The wave packet decomposition is an important tool for studying the
restriction problem. It translates the geometric condition (having strictly
positive second fundamental form) imposed on the surface $S$ into a way of
writing a function $f \in L^2(S)$ as a sum of simpler functions which are
are almost orthogonal to each other and whose Fourier transforms are
essentially supported on tubes. This idea was originated by Bourgain in
\cite{jb:besitype} and was further developed by several authors (see
\cite{tw:conesub}, \cite{tt:paraboloid}, \cite{tt:parkcity}, and
\cite{guth:poly}). Our presentation of this topic is tightly based on
\cite{guth:poly}. More precisely, we follow the proof of Proposition 2.6 in
\cite{guth:poly}, but we organize the arguments and state the results in a
slightly different manner.

\begin{prop}
\label{flatwave}
Suppose $\Theta$ is a closed ball in $\mbb R^n$ of radius $\rho \leq 1$,
$\delta > 0$, and $\{ D \}$ is a countable collection of closed balls in
$\mbb R^n$ of radius $\rho^{-1-\delta}$ satisfying
\begin{displaymath}
1 \leq \sum_D \chi_{(3/4)D^0} \leq C
\end{displaymath}
for some constant $C$. Also, suppose that $\Phi \in C^L((4/3)\Theta)$
satisfies $|\nabla^l \Phi| \leq C_L \rho^{-l}$ for $0 \leq l \leq L$ and
some constant $C_L$, $N > n$ is a positive integer, and
$L \geq n(4+\delta+2/\delta)+N(1+2/\delta)$. Then to every function
$f \in L^2(\mbb R^n)$ with $\mbox{\rm supp } \!\! f \subset \Theta$ there is
a sequence $\{ f_D \}$ in $L^2(\mbb R^n)$ with the following properties.
\\
{\rm (i)} Each $f_D$ is supported in $(4/3)\Theta$, $f= \sum_D f_D$ in
$L^1((4/3)\Theta)$, and
\begin{displaymath}
\sum_D \int |f_D|^2 d\omega \leq \| f \|_{L^2(\Theta)}^2.
\end{displaymath}
{\rm (ii)} We have
\begin{displaymath}
\sum_{D : \, z \not\in D} |\widehat{\Phi f_D}(z)|
\lct \rho^N \| f \|_{L^1(\Theta)}
\end{displaymath}
for all $z \in \mbb R^n$, where the implicit constant depends only on $L$,
$C_L$, $C$, and $n$.
\\
{\rm (iii)} If $D_1$ and $D_2$ are disjoint, then
\begin{displaymath}
\Big| \int \Phi f_{D_1} \overline{f_{D_2}} \; d\omega \Big|
\lct \rho^N \| f \|_{L^1(\Theta)}^2,
\end{displaymath}
where the implicit constant depends only on $L$, $C_L$, and $n$.
\end{prop}

\begin{proof}
We start by letting $\psi_\Theta \in C^\infty(\mbb R^n)$ be such that
$0 \leq \psi_\Theta \leq 1$, $\psi_\Theta= 1$ on $(5/4) \Theta$,
$\psi_\Theta$ has support in $(4/3) \Theta$, and
$|\nabla^l \psi_\Theta| \leq c_l \rho^{-l}$ for all $l$, where the $c_l$ are
constants that depend only on $l$ and $n$. We also let
$\Psi_\Theta= \Phi \, \psi_\Theta$. Clearly,
\begin{displaymath}
|\nabla^l \Psi_\Theta| \lct \rho^{-l}
\end{displaymath}
for $0 \leq l \leq L$, and
\begin{displaymath}
|\widehat{\Psi_\Theta}(z)| \lct \frac{|\Theta|}{(1 + \rho |z|)^L}
\end{displaymath}
for all $z \in \mbb R^n$, where the implicit constants depend only on $L$,
$C_L$, and $n$.

We then let $\{ \phi_D \}$ be a  partition of unity of $\mbb R^n$
subordinate to the open cover $\{ (3/4)D^0 \}$, and define the functions
$f_D$ by
\begin{displaymath}
f_D= \psi_\Theta (\varphi_D \ast f),
\end{displaymath}
where $\varphi_D$ is the inverse Fourier transform of $\phi_D$; in other
words, $\varphi_D$ is given by
$\varphi_D(\omega)=\widehat{\phi_D}(-\omega)$.

(i) We clearly have $\mbox{\rm supp } \!\! f_D \subset (4/3)\Theta$ for all
$D$. Since $\sum_D \phi_D = 1$, it follows that $\sum_D \phi_D \widehat{f}$
converges to $\widehat{f}$ in $L^2(\mbb R^n)$, so $\sum_D \varphi_D \ast f$
converges to $f$ in $L^2(\mbb R^n)$ (by Plancherel), and so
\begin{displaymath}
\sum_D f_D= \sum_D \psi_\Theta (\varphi_D \ast f)
          = \psi_\Theta \sum_D \varphi_D \ast f
\end{displaymath}
converges to $\psi_\Theta f= f$ in $L^1(\mbb R^n)$ (by Cauchy-Schwarz).
Also,
\begin{eqnarray*}
\lefteqn{\sum_D \int |f_D|^2 d\omega
\; \leq \, \sum_D \int |\varphi_D \ast f|^2 d\omega
\; = \, \sum_D \int |\phi_D|^2 |\widehat{f}|^2 dz} \\
& & \leq \; \int \Big( \sum_D |\phi_D| \Big)^2 |\widehat{f}|^2 dz
\; = \; \int |\widehat{f}|^2 dz \; = \; \int |f|^2 d\omega,
\end{eqnarray*}
where we have used Plancherel's theorem.

(ii) We have
\begin{displaymath}
\widehat{\Phi f_D}
= \widehat{\Psi_\Theta} \ast \big(\phi_D \widehat{f} \, \big),
\end{displaymath}
so
\begin{eqnarray*}
|\widehat{\Phi f_D}(z)|
&   =  & \Big| \int \widehat{\Psi_\Theta}(z-y) \phi_D(y) \widehat{f}(y) dy
         \Big| \\
& \lct & \int_{\mbox{\tiny supp} \, \phi_D}
         \frac{|\Theta|}{(1 + \rho |z-y|)^L} |\phi_D(y)| \,
         \| f \|_{L^1(\Theta)} dy \\
& \leq & \frac{\| f \|_{L^1(\Theta)} |\Theta| \, |\mbox{supp} \, \phi_D|}
              {\big( 1+\rho \, \mbox{dist}(z,\mbox{supp} \, \phi_D) \big)^L}
\end{eqnarray*}
for all $z \in \mbb R^n$. Since $|\Theta| \sim \rho^n$ and
$\mbox{supp} \, \phi_D$ is contained in $(3/4)D$ (which is a ball of radius
$(3/4)\rho^{-1-\delta})$), it follows that
\begin{displaymath}
|\widehat{\Phi f_D}(z)|
\lct \frac{\| f \|_{L^1(\Theta)} \rho^{-n \delta}}
          {\big( 1+\rho \, \mbox{dist}(z,\mbox{supp} \, \phi_D) \big)^L}
\end{displaymath}
for all $z \in \mbb R^n$. If $z \not\in D$, then
$\mbox{dist}(z,\mbox{supp} \, \phi_D) > \rho^{-1-\delta}/4$, and it follows
that
\begin{eqnarray}
|\widehat{\Phi f_D}(z)|
\label{nodeponC}
& \lct & \frac{\| f \|_{L^1(\Theta)} \rho^{-n \delta}}
              {\big( \rho+\rho \,\mbox{dist}(z,\mbox{supp} \,\phi_D) \big)^N
               \big( 1+ \rho \rho^{-1-\delta}/4 \big)^{L-N}} \nonumber \\
& \leq & \frac{4^{L-N} \| f \|_{L^1(\Theta)}
               \rho^{-n \delta - N + \delta (L-N)}}
              {\big( 1 + \mbox{dist}(z,\mbox{supp} \, \phi_D) \big)^N}
         \nonumber \\
& \leq & \frac{4^L \| f \|_{L^1(\Theta)} \rho^N}
              {\big( 1 + \mbox{dist}(z,\mbox{supp} \, \phi_D) \big)^N}
\end{eqnarray}
provided $L \geq n+N+2N/\delta$, where the implicit constants depend only on
$L$, $C_L$, and $n$. Letting
\begin{displaymath}
{\mathcal D}_j = \{ D : 2^j \rho^{-1-\delta} <
\mbox{dist}(z,\mbox{supp} \,\phi_D) \leq 2^{j+1} \rho^{-1-\delta} \}
\end{displaymath}
for $j=-2,-1, \ldots$, we then see that
\begin{eqnarray*}
\sum_{D : \, z \not\in D} |\widehat{\Phi f_D}(z)|
& \lct & \rho^N \| f \|_{L^1(\Theta)} \sum_{D : \, z \not\in D}
         \big( 1+ \mbox{dist}(z,\mbox{supp} \, \phi_D) \big)^{-N} \\
& \leq & \rho^N \| f \|_{L^1(\Theta)} \sum_{j=-2}^\infty
         \sum_{D \in {\mathcal D}_j}
         \big( 1+ \mbox{dist}(z,\mbox{supp} \, \phi_D) \big)^{-N} \\
& \leq & \rho^N \| f \|_{L^1(\Theta)} \sum_{j=-2}^\infty
         \sum_{D \in {\mathcal D}_j}
         \big( 1+ 2^j \rho^{-1-\delta} \big)^{-N} \\
& \lct & \rho^N \| f \|_{L^1(\Theta)},
\end{eqnarray*}
where the implicit constants depend only on $L$, $C_L$, $C$, and $n$.

(iii) Suppose $D_1$ and $D_2$ are disjoint. Then, by Plancherel's theorem,
\begin{displaymath}
\int \Phi f_{D_1} \overline{f_{D_2}} \, d \omega
= \int \Phi \, f_{D_1} \, \overline{\psi_\Theta} \,
       \overline{(\varphi_{D_2} \ast f)} \, d\omega \\
= \int \big( \Phi \, \overline{\psi_\Theta} f_{D_1}\widehat{\big) \,}
       \;\; \overline{\phi_{D_2} \widehat{f} \,} \, dz.
\end{displaymath}
The estimates in part (ii) that lead to (\ref{nodeponC}) apply to
$\big( \Phi \, \overline{\psi_\Theta} f_{D_1}\widehat{\big) \,}$ (since
$\Phi \, \overline{\psi_\Theta}$ has the same smoothness and decay
properties as $\Phi$), so (by (\ref{nodeponC}))
\begin{displaymath}
\Big| \big( \Phi \, \overline{\psi_\Theta} f_{D_1}\widehat{\big)\,}(z) \Big|
\lct \rho^{N'} \| f \|_{L^1(\Theta)}
\end{displaymath}
for all $x \in D_1^c$ (and hence for all $x \in D_2$) provided
$L \geq n+N'+2N'/\delta$, and so
\begin{eqnarray*}
\lefteqn{\Big| \int \Phi f_{D_1} \overline{f_{D_2}} \, d\omega \Big|
\; \lct \; \rho^{N'} \| f \|_{L^1(\Theta)}
           \int_{\mbox{\tiny supp} \, \phi_{D_2}}
           |\phi_{D_2}| \, \| f \|_{L^1(\Theta)} dz} \\
& & \leq \; \rho^{N'} \| f \|_{L^1(\Theta)}^2 |D_2|
   \; \sim  \; \rho^{N'} \| f \|_{L^1(\Theta)}^2 \rho^{-n(1+\delta)}
   \; \lct  \; \rho^N \| f \|_{L^1(\Theta)}^2
\end{eqnarray*}
provided $N' \geq N+n(1+\delta)$.
\end{proof}

We now consider a $C^L$ function $h: \Omega \to \mbb R$ defined on some open
set $\Omega \subset \mbb R^n$, and we assume the following bounds on the
first and second order partial derivatives of $h$:
\begin{equation}
\label{hxxyyxy}
\left\{ \begin{array}{ll}
        \frac{1}{2} \leq \partial_i^2 h(\omega) \leq \frac{3}{2} &
        \mbox{ for $i=1,\ldots, n$,} \\ \\
        \big| \partial_i \partial_j h(\omega) \big| \leq \frac{1}{4(n-1)} &
        \mbox{ if $i \not= j$,} \\ \\
        | \nabla h(\omega) | \leq \frac{7}{4} &
        \end{array} \right.
\end{equation}
for all $\omega \in \Omega$.

We let $B$ be a closed ball in $\Omega$ of center $\omega_0$ and radius
$r \leq 1/12$, $\theta$ be the graph of $h$ over $B$, and $3\theta$ be the
graph of $h$ over $3B$. We are going to show that in an appropriate
orthonormal system of coordinates, $\theta$ is contained in the graph of a
$C^L$ function $h_0$, defined in a ball of radius $\sim r$, such that both
$h_0$ and $\nabla h_0$ vanish at the center of the ball, and the graph of
$h_0$ is contained in $3\theta$.

Let ${\mathcal T}_{(\omega_0,h(\omega_0))} \theta$ be the tangent plane to
$\theta$ at $(\omega_0,h(\omega_0))$. If $\omega$ is a point on the boundary
of $B$ and
$(\Delta \omega, \Delta h)=(\omega-\omega_0,h(\omega)-h(\omega_0))$, then
Pythagoras' theorem shows that the projection of $(\Delta \omega, \Delta h)$
onto ${\mathcal T}_{(\omega_0,h(\omega_0))} \theta$ has length
\begin{displaymath}
\rho_0= \left( r^2 + (\Delta h)^2 -
        \frac{\big( \Delta h-\nabla h(\omega_0) \cdot \Delta \omega \big)^2}
             {1+|\nabla h(\omega_0)|^2} \right)^{1/2}.
\end{displaymath}
By (\ref{hxxyyxy}), $|\Delta h| \leq (7/4)r$, so
$\rho_0 \leq (\sqrt{65}/4)r$. On the other hand, Taylor's theorem and
(\ref{hxxyyxy}) tell us that
\begin{displaymath}
|\Delta h - \nabla h(\omega_0) \cdot \Delta \omega|
\leq \frac{1}{2} \Big( \frac{3}{2} |\Delta \omega|^2
                       + \frac{1}{4} |\Delta \omega|^2 \Big) < r^2,
\end{displaymath}
so $\rho_0 >r \sqrt{1-r^2}$, and so
$\sqrt{15} \; r < 4 \rho_0 \leq \sqrt{65} \; r$ (because $r < 3r \leq 1/4$).
This shows that if we let $\theta'$ and $(3\theta)'$ be the projections of
$\theta$ and $3 \theta$, respectively, onto
${\mathcal T}_{(\omega_0,h(\omega_0))} \theta$, and if we dilate $\theta'$
around $(\omega_0,h(\omega_0))$ by a factor of $4/3$, then the resulting set
will be contained in $(3 \theta)'$. More precisely, letting $\Theta$ be the
ball in ${\mathcal T}_{(\omega_0,h(\omega_0))} \theta$ of center
$(\omega_0,h(\omega_0))$ and radius $\rho = (4/\sqrt{15}) \rho_0$ (note that
$r < \rho \leq \sqrt{13/3} \; r$), we have
\begin{displaymath}
\theta' \subset \Theta \subset (4/3) \Theta \subset (3 \theta)'
\end{displaymath}
(because $(4/3)\rho \leq (4/3)(\sqrt{13/3} \,) r < (\sqrt{15}/4)(3r)$).
Therefore, in an appropriate orthonormal system of coordinates, $\theta$ is
contained in the graph of a $C^L$ function $h_0: (4/3)\Theta \to \mbb R$
such that both $h_0$ and $\nabla h_0$ vanish at the center of $\Theta$, and
the graph of $h_0$ is contained in $3\theta$. For the rest of this section,
all the implicit constants will depend on the $C^L$ norm of $h_0$.

Let $S_0$ be the graph of $h_0$. We know that
$\theta \subset S_0 \subset 3\theta$. Let $f$ be a function in $L^2(S_0)$
with support in $\theta$. We would now like to obtain the wave packet
decomposition of $f$. We start by replacing the function $f(\omega)$ in
Proposition \ref{flatwave} by the function
$f(\omega,h_0(\omega)) J(\omega)$, where
$J(\omega)=\sqrt{1+|\nabla h_0(\omega)|^2}$. To each member of the
collection $\{ D \}$, we associate a tube $T \subset \mbb R^{n+1}$ defined
by $T= D \times \mbb R$. We alert the reader that this definition of $T$ is
in the new system of coordinates that comes with the function $h_0$; in the
original system of coordinates the tube $T$ is perpendicular to the tangent
plane ${\mathcal T}_{(\omega_0,h(\omega_0))} \theta$. We then define the
function $f_T \in L^2(S_0)$ by
\begin{displaymath}
f_T(\omega,h_0(\omega))=
\frac{\big( f(\cdot,h_0(\cdot)) \, J \big)_D(\omega)}{J(\omega)}.
\end{displaymath}

Next, we apply parts (i) and (ii) of Proposition \ref{flatwave} with
$\Phi(\omega)= e^{-2 \pi i z_{n+1} h_0(\omega)}$. Since
$|\partial_j h_0| \lct \rho$ for $j=1,\ldots,n$, we need
$|z_{n+1}| \leq \rho^{-2}$ in order to satisfy the requirement
$|\nabla^l \Phi| \leq C_L \rho^{-l}$. To free the condition
$|z_{n+1}| \leq \rho^{-2}$ from depending on the choice of the orthonormal
coordinates, we require $|(z,z_{n+1})| \leq \rho^{-2} \sim r^{-2}$.

We know that
$f(\cdot,h_0(\cdot)) \, J = \sum_D \big( f(\cdot,h_0(\cdot)) \, J \big)_D$
in $L^1((4/3)\Theta)$, so $f= \sum_T f_T$ in $L^1(S_0)$, and
\begin{eqnarray*}
\sum_T \int |f_T|^2 d\sigma & = &
\sum_D \int \frac{\big|\big( f(\cdot,h_0(\cdot))\, J \big)_D(\omega)\big|^2}
                 {J(\omega)} d\omega \\
& \leq & \sum_D \int \big| \big( f(\cdot,h_0(\cdot))\, J \big)_D(\omega)
                     \big|^2 d\omega \\
& \leq & \int |(f(\omega,h_0(\omega)) J(\omega)|^2 d\omega \\
& \lct & \int |f|^2 d\sigma.
\end{eqnarray*}
We also have
\begin{displaymath}
Eg(z,z_{n+1})= \Big( \Phi \; g(\cdot,h_0(\cdot)) \, J \widehat{\Big) \;}(z)
\end{displaymath}
for all $g \in L^1(S_0)$, so
\begin{eqnarray*}
\lefteqn{\sum_{T : \, (z,z_{n+1}) \not\in T} |Ef_T(z,z_{n+1})|
         \; = \; \sum_{D : \, z \not\in D}
\Big( \Phi \, \big( f(\cdot,h_0(\cdot)) \, J \big)_D \widehat{\Big) \;}} \\
& & \lct \; \rho^N \| f(\cdot,h_0(\cdot)) \, J \|_{L^1(\Theta)}
    \; = \; \rho^N \| f \|_{L^1(\theta)} \; = \; \rho^N \| f \|_{L^1(S_0)}.
\end{eqnarray*}

The functions $f_T$ are almost orthogonal in $L^2(S_0)$. To see this, we
apply part (iii) of Proposition \ref{flatwave} with
$\Phi(\omega)= 1/J(\omega)$ to get
\begin{eqnarray*}
\lefteqn{\int f_{T_1} \overline{f_{T_2}} \, d\sigma
         \; =  \; \int \Phi(\omega) \,
                  \big( f(\cdot,h_0(\cdot))\, J \big)_{D_1}(\omega) \;
\overline{\big( f(\cdot,h_0(\cdot)) \, J \big)_{D_2}(\omega)} \; d\omega} \\
&  & \lct \; \rho^N \| f(\cdot,h_0(\cdot)) \, J \|_{L^1(\Theta)}^2
     \; = \; \rho^N \| f \|_{L^1(\theta)}^2
     \; = \; \rho^N \| f \|_{L^1(S_0)}^2.
\end{eqnarray*}

We summarize the above discussion in the following proposition, which is a
reformulation of Proposition 2.6 in \cite{guth:poly}.

\begin{prop}
\label{wave}
Suppose $S$ is a compact $C^L$ surface in $R^{n+1}$ given as the graph of
a function $h$ that satisfies {\rm (\ref{hxxyyxy})}, $\delta > 0$, $N > n/2$
is a positive integer, and
\begin{displaymath}
L \geq n \Big( 4 + 2\delta + \frac{1}{\delta} \Big)
       + (2N) \Big( 1 + \frac{1}{\delta} \Big).
\end{displaymath}
Let  $\theta$ be a cap on $S$ of center $\xi_0$ and radius
$r=R^{-1/2} \leq 1/12$, and $v(\theta)$ be the unit normal vector of $S$ at
$\xi_0$.

Then there is a countable collection $\tilde{\mbb T}(\theta)=\{ T \}$ of
finitely overlapping tubes in $\mbb R^{n+1}$ of radius $R^{(1/2)+\delta}$,
which are parallel to $v(\theta)$, such that the following holds. To every
function $f \in L^2(S)$ with $\mbox{\rm supp} \, f \subset \theta$ there is
a sequence $\{ f_T \}$ in $L^2(S)$ with the following properties.
\\
{\rm (i)} Each $f_T$ is supported in $3\theta$,
$f= \sum_{T \in \tilde{\mbb T}(\theta)} f_T$ in
$L^1(S)$, and
\begin{displaymath}
\sum_{T \in \tilde{\mbb T}(\theta)} \int |f_T|^2 d\sigma
\lct \| f \|_{L^2(S)}^2,
\end{displaymath}
where the implicit constant depends only on the $C^1$ norm of $h$.
\\
{\rm (ii)} We have
\begin{displaymath}
\sum_{T \in \tilde{\mbb T}(\theta) : \, x \not\in T} |E f_T(x)|
\lct R^{-N} \| f \|_{L^1(S)}
\end{displaymath}
for all $x \in \mbb R^{n+1}$ with $|x| \leq R$, where the implicit constant
depends only on $L$, $N$, $n$, and the $C^L$ norm of $h$.
\\
{\rm (iii)} If $T_1, T_2 \in \tilde{\mbb T}(\theta)$ are disjoint, then
\begin{displaymath}
\Big| \int f_{T_1} \overline{f_{T_2}} \; d\sigma \Big|
\lct R^{-N} \| f \|_{L^1(S)}^2,
\end{displaymath}
where the implicit constant depends only on $L$, $N$, $n$, and the $C^L$
norm of $h$.
\\
{\rm (iv)} Let $\mbb T(\theta)=\{ T \in \tilde{\mbb T}(\theta) : T \cap
B(0,R) \not=\emptyset \}$. Then
\begin{displaymath}
\Big| Ef(x) - \sum_{T \in \mbb T(\theta)} Ef_T(x) \Big|
\lct R^{-N} \| f \|_{L^1(\theta)}
\end{displaymath}
for all $x \in B(0,R)$.
\end{prop}

We note that in applying Proposition \ref{flatwave} to get Proposition
\ref{wave}, we have replaced $\delta$ by $2\delta$ and $N$ by $2N$. We also
note that part (iv) of Proposition \ref{wave} is an immediate consequence of
part(ii).

For each $T \in \tilde{\mbb T}(\theta)$, the function $f_T$ is called a wave
packet. The equality $f= \sum_{T \in \tilde{\mbb T}(\theta)} f_T$ (which
holds in $L^1(S)$) is called the wave packet decomposition of $f$. The
functions that we shall be dealing with for the rest of the paper are
defined on a surface $S \subset \mbb R^3$. This means that we shall be using
Proposition \ref{flatwave} in $\mbb R^2$, and Proposition \ref{wave} in
$\mbb R^3$.

\section{Guth's polynomial partitioning method}

In this section, and for the rest of the paper, we make the following
assumption on the surface $S$.

\begin{assume}
\label{graphofh}
The surface $S$ is the graph of a function $h: B^2(0,1) \to \mbb R$ that
satisfies the following conditions.
\\
{\rm (i)} There is an integer $L \geq 3$ such that $h \in C^L(B^2(0,1))$.
\\
{\rm (ii)} We have $h(0)= \nabla h(0)= 0$.
\\
{\rm (iii)} For all $\omega \in B^2(0,1)$, both eigenvalues of the Hessian
            $\partial^2 h(\omega)$ lie in the open interval $(3/4,5/4)$.
\\
{\rm (iv)} We have $\| \nabla^l h \|_{L^\infty(B^2(0,1))} < 10^{-9}$ for
           $3 \leq l \leq L$.
\end{assume}

Once we have proved that to every $\epsilon > 0$ there is a positive integer
$L_\epsilon$ such that Theorem \ref{mainjj} holds for all surfaces that
satisfy (i)--(iv) with $L \geq L_\epsilon$, the result for general
$C^\infty$ compact surfaces with strictly positive second fundamental form
would follow by a standard parabolic scaling argument. We refer the reader
to the last paragraph of Subsection 2.3 in \cite{guth:poly} for a very nice
outline of this argument.

The information we have about the eigenvalues of $\partial^2 h$ give the
following bounds on the second order partial derivatives of $h$:
\begin{displaymath}
\left\{ \begin{array}{ll}
        \frac{1}{2} < \partial_i^2 h(\omega) < \frac{3}{2} &
        \mbox{ for $i=1, 2$,} \\ \\
        \big| \partial_i \partial_j h(\omega) \big| < \frac{1}{4} &
        \mbox{ if $i \not= j$}
        \end{array} \right.
\end{displaymath}
for all $\omega \in B^2(0,1)$. These bounds tell us that
$\| \partial^2 h(\omega) \| < 7/4$ for all $\omega \in B^2(0,1)$, where
$\| \partial^2 h(\omega) \|$ is the operator norm of $\partial^2 h(\omega)$.
Condition (ii) then implies that
\begin{equation}
\label{bdongrad}
|\nabla h(\omega)| < (7/4) |\omega|
\hspace{0.25in} \forall \; \omega \in B^2(0,1).
\end{equation}
Thus $h$ satisfies (\ref{hxxyyxy}) (with $n=2$) on some open set $\Omega$
that contains the closed unit ball $B^2(0,1)$. Thus Proposition \ref{wave}
applies to functions on $S$. For this reason, in this and the next four 
sections, we will let $N > 1$ and $\delta$ be, respectively, a positive 
integer and a positive number that satisfy the standing hypothesis
\begin{equation}
\label{LdeltaN}
L \geq 2 \Big( 4 + 2\delta + \frac{1}{\delta} \Big)
       + (2N) \Big( 1 + \frac{1}{\delta} \Big).
\end{equation}
At the end of the argument, it will be clear that to find the positive
integer $L_\epsilon$ that was mentioned above (following the statement of
Assumption \ref{graphofh}), we need to take $N$ to be a large absolute
constant, say $N=1000$, and $\delta$ to be small relative to $\epsilon$, say
$\delta=\epsilon^2$.

Let $P$ be a polynomial in $n$ real variables of degree $D$, and $Z(P)$ be
the zero set of $P$. A connected component of $\mbb R^n \setminus Z(P)$ is
called a cell. If $\{ O_i \}$ are all the cells of $\mbb R^n\setminus Z(P)$,
then  $|\{ i \}| \leq C_n D^n$ for some constant $C_n$ that only depends on
the dimension $n$. (For a proof of this bound on the number of cells, we
refer the reader to Milnor \cite{m:nofcells}.) A line in $\mbb R^n$, 
however, can intersect at most $D+1$ cells. This relationship between lines 
and polynomials lies at the heart of what is now referred to in incidence 
geometry and harmonic analysis as the polynomial method.

The polynomial method helped resolve a number of longstanding combinatorial
problems in incidence geometry concerning the intersection patterns of lines
in Euclidean space as well as in vector spaces over finite fields. In
harmonic analysis, the combinatorial issues concern the intersection pattern
of tubes with a fixed radius $\rho$. A tube can enter much more than $D+1$
of the cells $\{ O_i \}$. To go around this difficulty, Guth defined in
\cite{guth:poly} the cell-wall $W$ as the $\rho$-neighborhood of $Z(P)$ and
considered the sets $O_i'= O_i \setminus W$. Guth then observed that if a
tube enters $O_i'$, then its core line will enter $O_i$, so the tube can
enter at most $D+1$ of the modified cells $\{ O_i' \}$.

Another property of polynomials that lies at the heart of the polynomial
method (when working in Euclidean space rather than in vector spaces over
finite fields) is that given a degree $D$ and a non-negative integrable
function $F$, one can use the topology of $\mbb R^n$ to find a polynomial
$P$ of degree at most $D$ such that the integrals $\int_{O_i} F dx$ are
essentially equal.

\begin{alphthm}[Corollary 1.7 in \cite{guth:poly}]
\label{hamsand}
Let $F$ be a non-negative function in $L^1(\mbb R^n)$. Then to every $D$
there is a non-zero polynomial $P$ of degree at most $D$ such that $P$ is a
product of non-singular polynomials\footnote{Recall that a polynomial $Q$ is
non-singular if $\nabla Q(x) \not= 0$ for all $x \in Z(Q)$.},
$\mbb R^n \setminus Z(P)$ is a disjoint union of $\sim D^n$ cells $O_i$, and
the integrals $\int_{O_i} F(x) dx$ agree up to a factor of $2$.
\end{alphthm}

Given a function $f \in L^1(S)$, our goal is to estimate the $L^p(Hdx)$ norm
of $Ef$ over the ball $B_R$ in $\mbb R^3$ of center 0 and radius $R$. We 
shall think of $B_R$ as lying in physical space and of the surface $S$ as 
lying in frequency space. In physical space, $B_R$ inherits from $\mbb R^3$ 
a partition into cells $O_i'$ and a cell-wall $W$ coming from a polynomial
$P$, as described above. The only condition we impose on $P$ for now is that
it is a product of non-singular polynomials in three real variables; we will
not use Theorem \ref{hamsand} until we arrive at the proof of Theorem
\ref{biltobr}. In frequency space, we cover $S$ by a collection
$\{ \theta \}$ of finitely-overlapping caps each of radius $R^{-1/2}$, and
we write $f=\sum_\theta f_\theta$ with $f_\theta$ supported in $\theta$ and
such that
$(\mbox{supp} \, f_\theta) \cap (\mbox{supp} \, f_{\theta'})= \emptyset$ if
$\theta \not= \theta'$. Applying Proposition \ref{wave} of the previous
section to each $f_\theta$, we obtain a wave packet decomposition of $f$:
\begin{displaymath}
f= \sum_\theta \sum_{T \in \tilde{\mbb T}(\theta)} \big( f_{\theta} \big)_T
\end{displaymath}
with the equality holding in $L^1(S)$.

We point out that the covering $\{ \theta \}$ of $S$ by $R^{-1/2}$-caps will
be fixed throughout the argument, so, in order to simplify the notation, we
write $f_T$ for $\big( f_{\theta} \big)_T$.

Going back to physical space, applying part (iv) of Proposition \ref{wave}
to each $f_\theta$ and summing shows that
\begin{equation}
\label{wave(ii)all}
Ef(x)= \sum_{T \in \mbb T} Ef_T(x) + O \big( R^{-N} \| f \|_{L^1(S)} \big)
\end{equation}
for all $x \in B_R$, where
\begin{displaymath}
\mbb T= \cup_\theta \mbb T(\theta).
\end{displaymath}
Recall from Proposition \ref{wave} that each $T \in \mbb T$ is a tube of
radius $R^{(1/2)+\delta}$, so the cell-wall is
\begin{displaymath}
W= N_{R^{(1/2)+\delta}} Z(P).
\end{displaymath}

If $\mbb T_i$ is a subset of $\mbb T$, we set
\begin{displaymath}
f_i= \sum_{T \in \mbb T_i} f_T.
\end{displaymath}
We shall often denote a function on $S$ which is supported in a cap $\tau$
by $f_\tau$. In this case, we shall write $f_{\tau,i}$ for $(f_\tau)_i$ and
$f_{\tau,T}$ for $(f_\tau)_T$, so that
\begin{displaymath}
f_{\tau,i}= \sum_{T \in \mbb T_i} f_{\tau,T}.
\end{displaymath}
If we happen to have $f= \sum_\tau f_\tau$, then\footnote{One can easily see
from the proof of Proposition \ref{flatwave} that $(f+g)_T=f_T+g_T$.}
\begin{equation}
\label{eglemma}
f_i = \sum_{T \in \mbb T_i} f_T = \sum_{T \in \mbb T_i} \sum_\tau f_{\tau,T}
= \sum_\tau \sum_{T \in \mbb T_i} f_{\tau,T} = \sum_\tau f_{\tau,i}.
\end{equation}

Since the $f_{\tau,T}$ are almost orthogonal (by part (iii) of Proposition
\ref{wave}), one expects $\sum_{T \in \mbb T_i} \| f_{\tau,T} \|_{L^2(S)}^2$
to be smaller than $\| f_\tau \|_{L^2(S)}^2$. The next lemma makes this
precise.

\begin{alphlemma}[Lemma 2.7 in \cite{guth:poly}]
\label{Lemma2.7}
Suppose $\{ \mbb T_i \}_{i \in I}$ is a family of subsets of $\mbb T$, $k$
is a positive integer, and $\tau \subset S$. If each tube
$T \in \cup_{i \in I} \mbb T_i$ belongs to at most $k$ of the subsets
$\{ \mbb T_i \}_{i \in I}$, then
\begin{displaymath}
\sum_{i \in I} \int_{3 \theta} |f_{\tau,i}|^2 d\sigma(\xi)
\lct k \int_{10 \theta} |f_\tau|^2 d\sigma(\xi)
\end{displaymath}
for all $\theta$. Also,
\begin{displaymath}
\sum_{i \in I} \int_S |f_{\tau,i}|^2 d\sigma(\xi)
\lct k \int_S |f_\tau|^2 d\sigma(\xi).
\end{displaymath}
\end{alphlemma}

We now fix a specific family $\{ \mbb T_i \}_{i \in I}$ of subsets of
$\mbb T$. We define
\begin{displaymath}
\mbb T_i= \{ T \in \mbb T : T \cap O_i' \not= \emptyset \}.
\end{displaymath}
Recalling Guth's motivation for introducing the modified cells $O_i'$, we
know that a tube $T$ of radius $R^{(1/2)+\delta}$ can enter at most $D+1$ of
these cells. So, a tube $T \in \mbb T$ can belong to at most $D+1$ of the
sets $\mbb T_i$. For later reference, we state this fact in the following
lemma.

\begin{alphlemma}[Lemma 3.2 in \cite{guth:poly}]
\label{fundthmalg}
A tube $T \in \mbb T$ can belong to at most $D+1$ of the sets $\mbb T_i$.
\end{alphlemma}

The integral of $|E f|^p H$ on the cells $O_i' \cap B_R$ will be controlled
using induction. To control the integral of $|E f|^p H$ on $W \cap B_R$, we
cover $B_R$ with $\sim R^{3 \delta}$ balls $B_j$ of radius $R^{1-\delta}$.
If $B_j \cap W \not= \emptyset$, then the tubes of $\mbb T$ will be
separated into two groups: the tubes that are tangent to $Z(P)$ in $B_j$,
and the tubes that are transverse to $Z(P)$ in $B_j$. Here are the details.

Let $Z_0(P)$ be the set of all non-singular points of $Z(P)$. We denote by
$\mbb T_{j,\mbox{\rm \tiny tang}}$ the set of all $T \in \mbb T$ satisfying
\begin{displaymath}
\left\{ \begin{array}{l}
        T \cap W \cap B_j \not= \emptyset \\
        \mbox{Angle}(v(T),T_zZ(P)) \leq R^{-(1/2)+2\delta}
        \;\; \forall \; z \in Z_0(P) \cap 2B_j \cap 10T,
        \end{array} \right.
\end{displaymath}
where $v(T)$ is the unit vector in the direction of the tube $T$, and by
$\mbb T_{j,\mbox{\rm \tiny trans}}$ the set of all $T \in \mbb T$ satisfying
\begin{displaymath}
\left\{ \begin{array}{l}
        T \cap W \cap B_j \not= \emptyset \\
        \exists \; z \in Z_0(P) \cap 2B_j \cap 10T
        \mbox{ such that Angle}(v(T),T_zZ(P)) > R^{-(1/2)+2\delta}.
        \end{array} \right.
\end{displaymath}
Any tube $T \in \mbb T$ that intersects $W \cap B_j$ lies in exactly one of
$\mbb T_{j,\mbox{\tiny tang}}$ and $\mbb T_{j,\mbox{\tiny trans}}$. For a
proof of this fact, we refer the reader to the paragraph immediately
following Definitions 3.3 and 3.4 in \cite{guth:poly}. More importantly, we
have the following two remarkable results of \cite{guth:poly}:

\begin{alphlemma}[Lemma 3.5 in \cite{guth:poly}]
\label{Lemma3.5}
If $P$ has degree at most $D$, then a tube $T \in \mbb T$ can belong to at
most $\mbox{\rm Poly}(D)$ different sets
$\mbb T_{j,\mbox{\rm \tiny trans}}$.
\end{alphlemma}

\begin{alphlemma}[Lemma 3.6 in \cite{guth:poly}]
\label{Lemma3.6}
If $P$ has degree at most $D$, then, for each $j$, the number of different
$\theta$ such that
$\mbb T_{j,\mbox{\rm \tiny tang}} \cap \mbb T(\theta) \not= \emptyset$ is at
most $D^2 R^{(1/2)+O(\delta)}$.
\end{alphlemma}

We let
\begin{displaymath}
f_{\tau,j,\mbox{\tiny tang}}
= \sum_{T \in \mbb T_{j,\mbox{\tiny tang}}} f_{\tau,T}
\hspace{0.25in} \mbox{and} \hspace{0.25in}
f_{j,\mbox{\tiny tang}}= \sum_\tau f_{\tau,j,\mbox{\tiny tang}},
\end{displaymath}
and
\begin{displaymath}
f_{\tau,j,\mbox{\tiny trans}}
= \sum_{T \in \mbb T_{j,\mbox{\tiny trans}}} f_{\tau,T}
\hspace{0.25in} \mbox{and} \hspace{0.25in}
f_{j,\mbox{\tiny trans}}= \sum_\tau f_{\tau,j,\mbox{\tiny trans}}.
\end{displaymath}
Recall that $S$ is covered by $\sim K^2$ caps $\tau$ of diameter $1/K$. If
$I$ is any subset of these caps, we let
\begin{displaymath}
f_{I,j,\mbox{\tiny trans}}= \sum_{\tau \in I} f_{\tau,j,\mbox{\tiny trans}}.
\end{displaymath}
The contribution to $\int_{B_j \cap W}|Ef(x)|^p H(x)dx$ coming from
the transverse tubes will be controlled by using induction to estimate
\begin{displaymath}
\sum_I \int_{B_j \cap W} |Ef_{I,j,\mbox{\tiny trans}}(x)|^p H(x)dx,
\end{displaymath}
where the sum runs over all $I \subset \{ \tau \}$. To control the
contribution coming from the tangential tubes, we make a further definition.

We say that two caps $\tau_1$ and $\tau_2$ are non-adjacent if the distance
between them is $\geq 1/K$. We define
\begin{displaymath}
\mbox{Bil}_{P,\delta} E f_{j,\mbox{\tiny tang}} =
\sum_{\tau_1, \tau_2 \, \mbox{\tiny non-adjacent}}
|E f_{\tau_1,j,\mbox{\tiny tang}}|^{1/2}
|E f_{\tau_2,j,\mbox{\tiny tang}}|^{1/2}.
\end{displaymath}
We call the function $\mbox{Bil}_{P,\delta} E f_{j,\mbox{\tiny tang}}$ the
tangential part of $Ef$ with respect to the polynomial $P$ and the parameter
$\delta$. Of course, this function also depends on $R$, the ball $B_j$, and
the decomposition $\{ \tau \}$ of $S$, but $R$, $B_j$ and $\{ \tau \}$ will
often appear elsewhere in the estimates involving
$\mbox{Bil}_{P,\delta} E f_{j,\mbox{\tiny tang}}$, so to simplify the
notation we only emphasize the dependence of this function on $P$ and
$\delta$ (see the statement of Theorem \ref{biltobr} below).

It is not clear how controlling the tangential part of $Ef$ on $B_j \cap W$
can lead to controlling the contribution to
$\int_{B_j \cap W}|Ef(x)|^p H(x)dx$ coming from the tangential tubes. One of
the important ideas of \cite{guth:poly} is that controlling
$\mbox{Bil}_{P,\delta} E f_{j,\mbox{\tiny tang}}$ on $B_j \cap W$ actually
leads to controlling the contribution of the tangential tubes to the $L^p$
norm of the {\it broad part} of $Ef$ on $B_j \cap W$.

The definition of the broad part of $Ef$ and its relation to
$\mbox{Bil}_{P,\delta} E f_{j,\mbox{\tiny tang}}$ is the subject of our next
section.

Guth  estimated the $L^p(dx)$ norm of
$\mbox{Bil}_{P,\delta} E f_{j,\mbox{\tiny tang}}$ over $B_j \cap W$ by using
C\'{o}rdoba's $L^4$ argument (see Lemma 3.10 in \cite{guth:poly}, or Lemma
\ref{Lemma3.10} below), and used this estimate to derive an estimate on the
broad part of $Ef$, before going back to $Ef$ itself.

In the next section, we formulate a general theorem, Theorem \ref{biltobr}, 
saying that if one has a favorable bound on the $L^p(Hdx)$ norm of
$\mbox{Bil}_{P,\delta} E f_{j,\mbox{\tiny tang}}$ over $B_j \cap W$, then
one gets a favorable estimate on the broad part of $Ef$ over $B_R$.

In Section 11, we establish bounds on various $L^p(Hdx)$ norms of
$\mbox{Bil}_{P,\delta} E f_{j,\mbox{\tiny tang}}$ over $B_j \cap W$. We then
insert these bounds into Theorem \ref{biltobr} and arrive at our desired
estimates on the $L^p(Hdx)$ norm of the broad part of $Ef$ on $B_R$.

\section{The broad part and its interaction with the zero set of $P$}

We are now in good shape to present Guth's definition of the broad part of
$Ef$. We let $m$ and $K$ be constants, and we think of $K$ as being rather
large. We consider a covering $\{ B^2(\omega,r) \}$ of $B^2(0,1)$ such that
the centers $\omega$ are $K^{-1}$-separated, and the radius $r$ satisfies
the inequalities $1/K \leq r \leq \sqrt{m}/K$. If a point $\omega_0$ belongs
to $M$ of these balls, then the centers of the $M$ balls lie in
$B^2(\omega_0,\sqrt{m}/K)$. Since the $\omega$ are $1/K$-separated, it
follows that $M \leq c \, m$ for some absolute constant $c$. Letting $\tau$
be the graph of $h$ over $B^2(\omega,r)$ (i.e., $\tau$ is a cap on $S$ of
center $(\omega,h(\omega))$ and radius $r$), we obtain a covering of $S$ by
caps $\tau$ such that each point of $S$ lies in at most $c \, m$ different
caps. We shall refer to $\{ \tau \}$ as a decomposition of $S$ of
multiplicity $m$. We write $f= \sum_\tau f_\tau$ with
$\mbox{supp} \, f_\tau \subset \tau$, but we do not insist that
$(\mbox{supp} \, f_\tau) \cap (\mbox{supp} \, f_{\tau'})= \emptyset$ if
$\tau \not= \tau'$, as we did with the decomposition
$f= \sum_\theta f_\theta$ above.

Suppose $0 < \beta \leq 1$. We say the point $x \in \mbb R^3$ is
$\beta$-broad for $E f$ if
\begin{displaymath}
\max_\tau |E f_\tau(x)| < \beta |E f(x)|.
\end{displaymath}
Since $|Ef(x)| \lct K^2 \max_\tau |E f_\tau(x)|$, a necessary condition for
the existence of $\beta$-broad points is that $\beta \gct K^{-2}$. Also, if
$f=f_\tau$ for some $\tau$, then (since $\beta \leq 1$) no point of physical
space can satisfy the above inequality, so another necessary condition for
the existence of $\beta$-broad points is that $f$ is not supported in just
one of the caps $\tau$.

We now define the $\beta$-broad part of $Ef$ to be the function
$\mbox{Br}_\beta E f : \mbb R^3 \to [0,\infty)$ given by
\begin{displaymath}
\mbox{Br}_\beta E f(x) = \left\{
\begin{array}{ll}
|E f(x)| & \mbox{ if $x$ is $\beta$-broad for $E f$,} \\
0        & \mbox{ otherwise.}
\end{array} \right.
\end{displaymath}
Clearly,
\begin{displaymath}
|E f(x)|^p
\leq \mbox{Br}_\beta E f(x)^p + \frac{1}{\beta^p} \sum_\tau |E f_\tau(x)|^p
\end{displaymath}
for all $x \in \mbb R^3$ and $p > 0$. Guth's strategy in \cite{guth:poly} is
to estimate $\int_{B_R} \mbox{Br}_\beta E f(x)^p dx$ (for appropriate
$\beta$) by using polynomial partitioning and induction, and estimate
$\int_{B_R} |E f_\tau(x)|^p dx$ by using parabolic scaling and induction.

Recall from (\ref{wave(ii)all}) that $Ef \sim \sum_{T \in \mbb T} Ef_T$ on
$B_R$. Since $Ef_T$ is essentially supported on the tube $T$, to estimate
$\int_{O_i' \cap B_R} |Ef(x)|^p H(x)dx$, one may replace the function $f$ by
the function $f_i=\sum_{T \in \mbb T_i} f_T$. We remind the reader that
\begin{displaymath}
\mbb T_i= \{ T \in \mbb T : T \cap O_i' \not= \emptyset \}.
\end{displaymath}
The next lemma says that this is also true for the broad part of $Ef$. We
include the proof of this lemma, as well as the proof of Lemma
\ref{Lemma3.8} below, because of their importance to the flow of the
argument. Also, because of some minor differences between the statements
here and the corresponding statements in \cite{guth:poly}; for example, the
form of the error terms, and the conditions on $K$ and $R$.

\begin{alphlemma}[Lemma 3.7 in \cite{guth:poly}]
\label{Lemma3.7}
Suppose $\epsilon, K >0$, $K^{-\epsilon} \leq \beta \leq 1$, and
$R \geq C K^\epsilon$. If $x \in O_i'$ and $C$ is sufficiently large, then
\begin{displaymath}
\mbox{\rm Br}_\beta E f(x) \leq \, \mbox{\rm Br}_{2\beta} E f_i(x) +
O \Big( R^{-N+1} \sum_\tau \| f_\tau \|_{L^1(S)} \Big).
\end{displaymath}
\end{alphlemma}

\begin{proof}[Proof {\rm (\cite{guth:poly})}]
From (\ref{wave(ii)all}), we know that
\begin{displaymath}
E f_\tau(x)= \sum_{T \in \mbb T} E f_{\tau,T}(x)
             + O \big( R^{-N} \| f_\tau \|_{L^1(S)} \big).
\end{displaymath}
The point $x$ is in $O_i'$. If $x \in T$, then $T$ must intersect $O_i'$,
and it follows that $T \in \mbb T_i$. If $x \not\in T$ and
$T \in \mbb T(\theta)$, then part (ii) of Proposition \ref{wave} tells us
that $|E f_{\tau,T}(x)| \lct R^{-N} \| f_\tau \|_{L^1(\theta)}$, so
\begin{displaymath}
\sum_{T \not\in \mbb T_i} |E f_{\tau,T}(x)|
=\sum_\theta \sum_{T \in \mbb T(\theta) \setminus \mbb T_i}|E f_{\tau,T}(x)|
\lct \sum_\theta R^{-N} \| f_\tau \|_{L^1(\theta)}
= R^{-N} \| f_\tau \|_{L^1(S)},
\end{displaymath}
and so
\begin{displaymath}
E f_\tau(x)
= \sum_{T \in \mbb T_i} E f_{\tau,T}(x)
  + O \big( R^{-N} \| f_\tau \|_{L^1(S)} \big)
= E f_{\tau,i}(x) + O \big( R^{-N} \| f_\tau \|_{L^1(S)} \big).
\end{displaymath}
Summing over $\tau$ and using (\ref{eglemma}), we get
\begin{displaymath}
E f(x)= E f_i(x) + O \Big( R^{-N} \sum_\tau \| f_\tau \|_{L^1(S)} \Big).
\end{displaymath}
Since we can assume that
$|E f(x)| > R^{-N+1} \sum_\tau \| f_\tau \|_{L^1(S)}$, it follows that
\begin{displaymath}
|E f_i(x)| > \frac{1}{2} R^{-N+1} \sum_\tau \| f_\tau \|_{L^1(S)}.
\end{displaymath}
(We have $R \geq C K^\epsilon \geq C$, so for the last inequality we need
$C'/C \leq 1/2$, where $C'$ is the implicit constant in the error term in
the last equality.) We can also assume that $x$ is $\beta$-broad for $E f$.
Under these assumptions, it remains to show that $x$ is $(2 \beta)$-broad
for $E f_i$. In other words, we have to show that for each $\tau$,
\begin{displaymath}
|E f_{\tau,i}(x)| \leq 2 \beta |E f_i(x)|.
\end{displaymath}
The equality before the last tells us that
\begin{displaymath}
|E f_{\tau,i}(x)|
\leq |E f_\tau(x)| + O \big( R^{-N} \| f_\tau \|_{L^1(S)} \big)
\leq \beta |E f(x)| + O \big( R^{-N} \| f_\tau \|_{L^1(S)} \big).
\end{displaymath}
The last equality then tells us that
\begin{displaymath}
|E f_{\tau,i}(x)|
\leq \beta |E f_i(x)|
     + O \Big( R^{-N} \sum_\tau \| f_\tau \|_{L^1(S)} \Big)
\end{displaymath}
(recall that $\beta \leq 1$). Thus
\begin{eqnarray*}
|E f_{\tau,i}(x)|
& \leq & \beta |E f_i(x)| + \frac{1}{R} \,
         O \Big( R^{-N+1} \sum_\tau \| f_\tau \|_{L^1(S)} \Big) \\
& \leq & \beta |E f_i(x)| + \frac{K^{-\epsilon}}{C} \,
         O \Big( R^{-N+1} \sum_\tau \| f_\tau \|_{L^1(S)} \Big) \\
& \leq & \beta |E f_i(x)| + \frac{K^{-\epsilon}}{C} (2 C''|E f_i(x)|) \\
& =    & 2 \beta |E f_i(x)|
\end{eqnarray*}
provided $C$ is sufficiently large, where $C''$ is the implicit constant in
the error term.
\end{proof}

The next lemma connects the broad part of $Ef$, $\mbox{Br}_\beta E f$, to 
the tangential part of $Ef$, 
$\mbox{Bil}_{P,\delta} E f_{j,\mbox{\tiny tang}}$. We remind the reader that 
$N$ and $\delta$ satisfy (\ref{LdeltaN}).

\begin{alphlemma}[Lemma 3.8 in \cite{guth:poly}]
\label{Lemma3.8}
Suppose $0 < \epsilon \leq 2$, $K \geq \sqrt[98]{10}$,
$K^{-\epsilon} \leq \beta \leq 1$, $\beta m \leq 10^{-5}$, and
$R \geq C K^\epsilon$. If $x \in B_j \cap W$ and $C$ is sufficiently large,
then
\begin{eqnarray*}
\mbox{\rm Br}_\beta E f(x)
& \leq & \frac{5}{4}
         \sum_I \mbox{\rm Br}_{2\beta} E f_{I,j,\mbox{\rm \tiny trans}}(x)
 + K^{100} \, \mbox{\rm Bil}_{P,\delta} E f_{j,\mbox{\rm \tiny tang}}(x) \\
& & + \; O \Big( R^{-N+1} \sum_\tau \| f_\tau \|_{L^1(S)} \Big).
\end{eqnarray*}
\end{alphlemma}

\begin{proof}[Proof {\rm (\cite{guth:poly})}]
We can assume that $|E f(x)| > R^{-N+1} \sum_\tau \| f_\tau \|_{L^1(S)}$
and $x$ is $\beta$-broad for $E f$. Let $I$ be the set of the caps $\tau$
such that $|E f_{\tau,j,\mbox{\tiny tang}}(x)| \leq K^{-100} |E f(x)|$. We
can also assume that $I^c$ does not contain two non-adjacent caps. Then
$I^c$ consists of at most $10^4 m$ caps. In fact, since the centers of the
caps are $K^{-1}$-separated, and the radius of each cap is at most
$\sqrt{m} K^{-1}$, we have $|I^c| \lct (\sqrt{m} K^{-1})^2 / (K^{-1})^2 =m$.
Since $x$ is $\beta$-broad for $E f$, and $\beta m \leq 10^{-5}$, it follows
that
\begin{displaymath}
\sum_{\tau \in I^c} |E f_\tau(x)| \leq 10^4 m \beta |E f(x)|
\leq \frac{1}{10} |E f(x)|,
\end{displaymath}
so that $|E f(x)| \leq (10/9) |E f_I(x)|$. If $x \in T$, then $T$ must
intersect $B_j \cap W$, and it follows that $T$ belongs to
$\mbb T_{j,\mbox{\tiny trans}}$ or $\mbb T_{j,\mbox{\tiny tang}}$. If
$x \not\in T$ and $T \in \mbb T(\theta)$, then
$|E f_{\tau,T}(x)| \lct R^{-N} \| f_\tau \|_{L^1(\theta)}$, so
\begin{eqnarray*}
\sum_{T \not\in
            \mbb T_{j,\mbox{\tiny trans}} \cup \mbb T_{j,\mbox{\tiny tang}}}
|E f_{\tau,T}(x)|
& = & \sum_\theta \sum_{T \in \mbb T(\theta) \setminus
          (\mbb T_{j,\mbox{\tiny trans}} \cup \mbb T_{j,\mbox{\tiny tang}})}
           |E f_{\tau,T}(x)| \\
& \lct & \sum_\theta R^{-N} \| f_\tau \|_{L^1(\theta)} \\
& \lct & R^{-N} \| f_\tau \|_{L^1(S)},
\end{eqnarray*}
and so (recall that
$T_{j,\mbox{\tiny tang}} \cap T_{j,\mbox{\tiny trans}}= \emptyset$)
\begin{eqnarray}
\label{equalitybeforethelast}
E f_\tau(x)
& = & \!\!\!\!\! \sum_{T \in T_{j,\mbox{\tiny trans}}} E f_{\tau,T}(x)
      + \sum_{T \in T_{j,\mbox{\tiny tang}}} E f_{\tau,T}(x)
      + O \big( R^{-N} \| f_\tau \|_{L^1(S)} \big) \nonumber \\
& = & E f_{\tau,j,\mbox{\tiny trans}}(x) + E f_{\tau,j,\mbox{\tiny tang}}(x)
      + O \big( R^{-N} \| f_\tau \|_{L^1(S)} \big).
\end{eqnarray}
Summing over $\tau \in I$, we see that
\begin{displaymath}
E f_I(x)= E f_{I,j,\mbox{\tiny trans}}(x) + E f_{I,j,\mbox{\tiny tang}}(x)
          + O \Big( R^{-N} \sum_\tau \| f_\tau \|_{L^1(S)} \Big).
\end{displaymath}
By the definition of $I$,
\begin{displaymath}
|E f_{I,j,\mbox{\tiny tang}}(x)| \leq
\sum_{\tau \in I} |E f_{\tau,j,\mbox{\tiny tang}}(x)| \leq
\sum_{\tau \in I} K^{-100} |E f(x)| \leq K^{-98} |E f(x)|,
\end{displaymath}
so
\begin{displaymath}
\frac{9}{10} |E f(x)| \leq |E f_I(x)|
\leq |E f_{I,j,\mbox{\tiny trans}}(x)| + K^{-98} |E f(x)|
     + O \Big( R^{-N} \sum_\tau \| f_\tau \|_{L^1(S)} \Big),
\end{displaymath}
and so
\begin{displaymath}
\frac{4}{5} |E f(x)| \leq |E f_{I,j,\mbox{\tiny trans}}(x)|
+ \frac{C'}{R} \Big( R^{-N+1} \sum_\tau \| f_\tau \|_{L^1(S)} \Big)
\end{displaymath}
provided $K^{-98} \leq 10^{-1}$, where $C'$ is the implicit constant in the
error term of the inequality before the last.

It remains to prove that $x$ is $(2\beta)$-broad for
$E f_{I,j,\mbox{\tiny trans}}(x)$. Since
$|E f(x)| > R^{-N+1} \sum_\tau \| f_\tau \|_{L^1(S)}$, we see that
\begin{displaymath}
|E f_{I,j,\mbox{\tiny trans}}(x)|
> \frac{1}{2} R^{-N+1} \sum_\tau \| f_\tau \|_{L^1(S)}.
\end{displaymath}
(We have $R \geq C K^\epsilon \geq C$, so for the last inequality we need
$C'/C \leq 3/10$.) Also, for $\tau \in I$,
$|E f_{\tau,j,\mbox{\tiny tang}}(x)| \leq K^{-100} |E f(x)|$, so
(\ref{equalitybeforethelast}) tells us that
\begin{eqnarray*}
|E f_{\tau,j,\mbox{\tiny trans}}(x)|
& \leq & |E f_\tau(x)| + |E f_{\tau,j,\mbox{\tiny tang}}(x)|
         + O \big( R^{-N} \| f_\tau \|_{L^1(S)} \big) \\
& \leq & (\beta + K^{-100}) |E f(x)|
         + O \big( R^{-N} \| f_\tau \|_{L^1(S)} \big),
\end{eqnarray*}
where we have also used the fact that $x$ is $\beta$-broad for $Ef$. Since
\begin{displaymath}
\frac{4}{5} |E f(x)| \leq |E f_{I,j,\mbox{\tiny trans}}(x)|
+ \frac{C'}{R} \Big( R^{-N+1} \sum_\tau \| f_\tau \|_{L^1(S)} \Big),
\end{displaymath}
$K^{-\epsilon} \leq \beta$, and $K^{-98} \leq 1/10$, we get
\begin{eqnarray*}
|E f_{\tau,j,\mbox{\tiny trans}}(x)|
& \leq & \frac{5}{4} \big( \beta+\frac{\beta^{2/\epsilon}}{10} \big)
         |E f_{I,j,\mbox{\tiny trans}}(x)| + \frac{C''}{R}
         \Big( R^{-N+1} \sum_\tau \| f_\tau \|_{L^1(S)} \Big) \\
& \leq & \frac{55}{40} \beta
         |E f_{I,j,\mbox{\tiny trans}}(x)| + \frac{2C''K^{-\epsilon}}{C}
         |E f_{I,j,\mbox{\tiny trans}}(x)| \\
& \leq & 2 \beta |E f_{I,j,\mbox{\tiny trans}}(x)|
\end{eqnarray*}
provided $2C''/C \leq 25/40$. Thus $x$ is $(2\beta)$-broad for
$|E f_{I,j,\mbox{\tiny trans}}(x)|$, as desired.
\end{proof}

We are now in position to state the main result of this section. For
$(R,K,m,b) \in [1,\infty)^4$, we let $\Lambda(R,K,m,b)$ be the set of all
functions $f \in L^1(S)$ such that $f= \sum_\tau f_\tau$ for some
decomposition $\{ \tau \}$ of $S$ of multiplicity $m$ with
$\mbox{supp} \, f_\tau \subset \tau$ and
\begin{equation}
\label{baverages}
\int_{B(\xi_0,R^{-1/2}) \cap S} |f_\tau(\xi)|^2 d\sigma(\xi)
\leq \frac{1}{R^{(b+1)/2}}
\end{equation}
for all $\xi_0 \in S$. Since $S$ can be covered by $\sim R$ of such balls
$B(\xi_0,R^{-1/2})$, (\ref{baverages}) tells us that
\begin{equation}
\label{bbdonL2(S)}
\int |f_\tau(\xi)|^2 d\sigma(\xi) \lct \frac{1}{R^{(b-1)/2}}.
\end{equation}

During the proof of Theorem \ref{mainjj} in the last section of the paper,
we are going to use two different values of $b$ depending on whether we
allow the final estimate to involve both $\| f \|_{L^2(S)}$ and
$\| f \|_{L^\infty(S)}$, or insist on only involving $\| f \|_{L^2(S)}$.

\begin{remark}
\label{implaterref}
It is important to note that all constants in the next theorem are allowed
to depend on $L$ and $\alpha$, but are uniform for all functions $h$
satisfying conditions (i)--(iv) of Assumption \ref{graphofh}, and all
weights $H$ of dimension $\alpha$. This fact will be crucial to the
induction argument that we will later use (see Theorem \ref{parabscaling})
to move from estimates on the broad part of $Ef$ to estimates on $Ef$
itself.
\end{remark}

\begin{thm}
\label{biltobr}
Let $3< p \leq 4$, $b \geq 1$, $\epsilon > 0$,
$0 \leq q_1 \leq 1 \leq 2 q_0$, $0< q_2 < q_0$, and $H$ be a weight of 
dimension $\alpha$. Also, let $\delta=\epsilon^2$, 
$\delta_{\mbox{\rm \tiny deg}}= \epsilon^4$, and 
$\delta_{\mbox{\rm \tiny trans}}= \epsilon^6$.

Suppose that
\begin{eqnarray}
\label{estonbil}
\lefteqn{\int_{B_j \cap W}
\mbox{\rm Bil}_{P,\delta} E f_{j,\mbox{\rm \tiny tang}}(x)^p H(x) dx}
\nonumber \\
& \leq & C_{\epsilon,K} R^{O(\delta)} R^{q_2\epsilon} A_\alpha(H)^{q_1}
        \Big( \sum_\tau \| f_\tau \|_{L^2(S)}^2 \Big)^{(3/2)+\epsilon}
\end{eqnarray}
whenever $R \geq C$, $K \geq 100$, $m \geq 1$, $f \in \Lambda(R,K,m,b)$,
$P$ is a polynomial of degree at most $D=R^{\delta_{\mbox{\rm \tiny deg}}}$,
and $P$ is a product of non-singular polynomial on $\mbb R^3$.

Then there is a constant $c_0$, which is independent of $q_0$, $q_1$, $b$,
and $p$, such that if $\epsilon \leq \min[ c_0,(p-3)/2]$, then there is a
$K=K(\epsilon)$ such that
\begin{eqnarray}
\label{underbil}
\lefteqn{\int_{B_R} \mbox{\rm Br}_\beta E f(x)^p H(x) dx} \nonumber \\
& \leq & C_\epsilon R^{q_0\epsilon} A_\alpha(H)^{q_1}
         \Big( \sum_\tau \| f_\tau \|_{L^2(S)}^2 \Big)^{(3/2)+\epsilon}
         R^{\delta_{\mbox{\rm \tiny trans}} \log(K^\epsilon \beta m)}
\end{eqnarray}
for all $\beta \geq K^{-\epsilon}$, $m \geq 1$, $R \geq 1$, and
$f \in \Lambda(R,K,m,b)$. Moreover,
$\lim_{\epsilon \to 0} K(\epsilon)=~\!\!\infty$.
\end{thm}

Our proof of Theorem \ref{biltobr}, which is the subject of the next
section, follows to a large extent the proof of Theorem 3.1 in Guth's paper
\cite{guth:poly}. There are three main differences between our theorem and
Guth's theorem. First, we are working in the weighted setting. Second, in
\cite{guth:poly}, (\ref{baverages}) is replaced by
\begin{displaymath}
\int_{B(\xi_0,R^{-1/2}) \cap S} |f_\tau(\xi)|^2 d\sigma(\xi) \leq
\frac{1}{R}.
\end{displaymath}
Third, the conclusion of Theorem \ref{biltobr} is conditional on
(\ref{estonbil}): if (\ref{estonbil}) holds, then (\ref{underbil}) holds. 

\section{Proof of Theorem \ref{biltobr}}

We alert the reader that in order to guarantee the independence of $c_0$
from $q_0$, $q_1$, $b$, and $p$, we will be careful to check that all the
constants (implicit and explicit) that appear in this proof are independent
of these parameters.

We may assume that the constant $C$ is large enough to satisfy the
requirements of Lemmas \ref{Lemma3.7} and \ref{Lemma3.8}. Lemma
\ref{Lemma3.8} also requires that $\beta m \leq 10^{-5}$. To meet this
requirement, we set
\begin{displaymath}
K=K(\epsilon)= e^{\epsilon^{-10}}
\end{displaymath}
and notice that
\begin{displaymath}
R^{\delta_{\mbox{\tiny trans}} \log(K^\epsilon \beta m)}
\geq R^{\epsilon^6 \log(K^\epsilon 10^{-6})} \geq R^{\epsilon^{-4}}
\end{displaymath}
if $\beta m \geq 10^{-6}$. So can assume that $\beta m \leq 10^{-6}$. In
fact,
\begin{eqnarray*}
\int_{B_R} H(x) dx
&  =   & \Big( \int_{B_R} H(x) dx \Big)^{q_1}
         \Big( \int_{B_R} H(x) dx \Big)^{1-q_1} \\
& \leq & A_\alpha(H)^{q_1} R^{\alpha q_1} |B(0,1)|^{1-q_1} R^{3-3q_1} \\
& \leq & 5 A_\alpha(H)^{q_1} R^3
\end{eqnarray*}
(because $|B(0,1)| \leq 5$, $0 \leq q_1 \leq 1$ and $\alpha \leq 3$) and
\begin{displaymath}
\| f \|_{L^2(S)} \leq \sum_\tau \| f_\tau \|_{L^2(S)}
\lct K \Big( \sum_\tau \| f_\tau \|_{L^2(S)}^2 \Big)^{1/2}
\end{displaymath}
(because the cardinality of $\{ \tau \}$ is $\sim K^2$), so
\begin{eqnarray}
\label{baseind}
\lefteqn{\int_{B_R} \mbox{\rm Br}_\beta E f(x)^p H(x) dx} \nonumber \\
& \lct & A_\alpha(H)^{q_1}
         \Big( \sum_\tau \| f_\tau \|_{L^2(S)}^2 \Big)^{p/2} R^3
         \nonumber \\
& \lct & A_\alpha(H)^{q_1}
        \Big( \sum_\tau \| f_\tau \|_{L^2(S)}^2 \Big)^{\frac{3}{2}+\epsilon}
      \Big( \sum_\tau \| f_\tau \|_{L^2(S)}^2 \Big)^{\frac{p-3}{2}-\epsilon}
         R^3 \\
& \lct & A_\alpha(H)^{q_1}
         \Big( \sum_\tau \| f_\tau \|_{L^2(S)}^2 \Big)^{(3/2)+\epsilon}
         R^{\delta_{\mbox{\tiny trans}} \log(K^\epsilon \beta m)} \nonumber
\end{eqnarray}
provided $p > 3 + 2 \epsilon$, $\beta m \geq 10^{-6}$, and
$R^3 \leq R^{\epsilon^{-4}}$, where we have used (\ref{bbdonL2(S)}) and the
fact that $(p-3)/2-\epsilon \leq 1/2$.

The theorem will be proved by induction. We see from (\ref{baseind})
(and (\ref{bbdonL2(S)})) that (\ref{underbil}) holds for
$R \leq CK^\epsilon$. So we assume that $R \geq CK^\epsilon$ and that the
theorem is true for all radii in the interval $[1,R/2]$. We also see from
(\ref{baseind}) that (\ref{underbil}) holds if
\begin{displaymath}
\sum_\tau \| f_\tau \|_{L^2(S)}^2
\leq R^{-3 \big( \frac{p-3}{2}-\epsilon \big)^{-1}},
\end{displaymath}
so we also assume that the theorem is true for all functions
$g \in \Lambda(R,K,m,b)$ with  $\sum_\tau \| g_\tau \|_{L^2(S)}^2 \leq
(1/2) \sum_\tau \| f_\tau \|_{L^2(S)}^2$.

In the discussion leading to the statement of Theorem \ref{biltobr}, $P$ was
a general polynomial on $\mbb R^3$ that was only required to be a product
of non-singular polynomials (see the paragraph following Theorem
\ref{hamsand}). Now we pick a specific polynomial $P$. We apply Theorem
\ref{hamsand} with $n=3$, $F= \chi_{B_R} (\mbox{Br}_\beta E f)^p H$, and
$D=R^{\delta_{\mbox{\tiny deg}}}$ to get a polynomial $P$ of degree at most
$D$ such that $P$ is a product of non-singular polynomials,
$\mbb R^3 \setminus Z(P)$ is a disjoint union of $\sim D^3$ cells $O_i$, and
\begin{displaymath}
\int_{O_i \cap B_R} \mbox{Br}_\beta E f(x)^p H(x) dx
\sim D^{-3} \int_{B_R} \mbox{Br}_\beta E f(x)^p H(x) dx.
\end{displaymath}
Let the modified cells $O_i'$ and the cell-wall $W$ be as defined above,
i.e.\ $W= N_{R^{(1/2)+\delta}} Z(P)$ and $O_i'=O_i \setminus W$. Then
\begin{eqnarray*}
\lefteqn{\int_{B_R} \!\! \mbox{Br}_\beta E f(x)^p H(x) dx} \\
& = & \sum_i \int_{B_R \cap O_i'} \mbox{Br}_\beta E f(x)^p H(x) dx +
      \int_{B_R \cap W} \mbox{Br}_\beta E f(x)^p H(x) dx.
\end{eqnarray*}

\subsection{The cellular case}

Suppose the cellular term dominates. Then there are $\sim D^3$ different
cells $O_i'$ such that
\begin{equation}
\label{fairdist}
\int_{B_R \cap O_i'} \mbox{Br}_\beta E f(x)^p H(x) dx \sim D^{-3}
\int_{B_R} \mbox{Br}_\beta E f(x)^p H(x) dx.
\end{equation}
We write $f_i=\sum_\tau f_{\tau,i}$, as in (\ref{eglemma}), and use
Lemma \ref{Lemma2.7} (applied to a single subset $\mbb T_i \subset \mbb T$)
to see that
\begin{displaymath}
\int_{B(\xi_0,R^{-1/2}) \cap S} |f_{\tau,i}|^2 d\sigma(\xi)
\leq c \int_{B(\xi_0,R^{-1/2}) \cap S} |f_\tau|^2 d\sigma(\xi)
\leq \frac{c}{R^{(b+1)/2}}
\end{displaymath}
for some absolute constant $c$. Using part (i) of Proposition \ref{wave}, we
notice that the supports of the $f_{\tau,i}$ are contained in neighborhoods
$\tau'$ of $\tau$, which, after choosing $C$ sufficiently large (in order
for $R^{-1/2}$ to be sufficiently small), define a decomposition
$\{ \tau' \}$ of $S$ of multiplicity $2m$. Letting
$g_{\tau,i}=c^{-1/2} f_{\tau,i}$ and $g_i=c^{-1/2} f_i$, we now see that
$g_i \in \Lambda(R,K,2m,b)$.

By Lemma \ref{fundthmalg} and Lemma \ref{Lemma2.7} (applied with $k=2D$), we
know that
\begin{displaymath}
\sum_{i \in I} \int |f_{\tau,i}|^2 d\sigma(\xi)
\lct D \int |f_\tau|^2 d\sigma(\xi),
\end{displaymath}
so that
\begin{displaymath}
\sum_{i \in I} \sum_\tau \int |f_{\tau,i}|^2 d\sigma(\xi)
\lct D \sum_\tau \int |f_\tau|^2 d\sigma(\xi).
\end{displaymath}
From among the $\sim D^3$ indices $i \in I$ that satisfy (\ref{fairdist}),
we can therefore pick a particular index $i$ that also satisfies
\begin{displaymath}
\sum_\tau \int |f_{\tau,i}|^2 d\sigma(\xi)
\lct D^{-2} \sum_\tau \int |f_\tau|^2 d\sigma(\xi).
\end{displaymath}
For sufficiently large $D$, we therefore have
\begin{displaymath}
\sum_\tau \int |g_{\tau,i}|^2 d\sigma(\xi)
\leq \frac{1}{2} \sum_\tau \int |f_\tau|^2 d\sigma(\xi).
\end{displaymath}
Lemma \ref{Lemma3.7}, now tells us that
\begin{eqnarray*}
\lefteqn{\int_{B_R} \mbox{Br}_\beta E f(x)^p H(x) dx \; \lct \;
         D^3 \int_{B_R \cap O_i'} \mbox{Br}_\beta E f(x)^p H(x) dx} \\
& & \lct \; D^3 \int_{B_R} \mbox{Br}_{2\beta} E g_i(x)^p H(x) dx + D^3
            \Big( R^{-N+1} \sum_\tau \| f_\tau \|_{L^1(S)} \Big)^p
            A_\alpha(H)^{q_1} R^3 \\
& & \lct \; D^3 \int_{B_R} \mbox{Br}_{2\beta} E g_i(x)^p H(x) dx + R^{-N+7}
            A_\alpha(H)^{q_1}
            \Big( \sum_\tau \| f_\tau \|_{L^2(S)}^2 \Big)^{(3/2)+\epsilon},
\end{eqnarray*}
where we have used the assumption $3 < p \leq 4$ as well as
(\ref{bbdonL2(S)}). By induction on
$\sum_\tau \int |f_\tau|^2 d\sigma(\xi)$, we can apply (\ref{underbil}) to
$g_i$. We get
\begin{eqnarray*}
\lefteqn{\int_{B_R} \mbox{Br}_{2\beta} E g_i(x)^p H(x) dx} \\
& \leq & C_\epsilon R^{q_0\epsilon} A_\alpha(H)^{q_1}
     \Big( \sum_\tau \int |g_{\tau,i}|^2 d\sigma(\xi) \Big)^{(3/2)+\epsilon}
         R^{\delta_{\mbox{\tiny trans}} \log(4K^\epsilon \beta m)},
\end{eqnarray*}
so that
\begin{eqnarray*}
\lefteqn{\int_{B_R} \mbox{Br}_\beta E f(x)^p H(x) dx \; \lct \;
         D^3 \int_{B_R \cap O_i'} \mbox{Br}_\beta E f(x)^p H(x) dx} \\
& \leq & C' D^3 C_\epsilon R^{q_0\epsilon} A_\alpha(H)^{q_1}
  \Big( D^{-2} \sum_\tau \int |f_\tau|^2 d\sigma(\xi) \Big)^{(3/2)+\epsilon}
        R^{\delta_{\mbox{\tiny trans}} \log(4K^\epsilon \beta m)} \\
&      & + \; C' R^{-N+7} A_\alpha(H)^{q_1}
           \Big( \sum_\tau \| f_\tau \|_{L^2(S)}^2 \Big)^{(3/2)+\epsilon} \\
& = & \Big( C' D^{-2\epsilon} R^{(\log 4) \delta_{\mbox{\tiny trans}}} \Big)
      C_\epsilon R^{q_0\epsilon} A_\alpha(H)^{q_1}
      \Big( \sum_\tau \int |f_\tau|^2 d\sigma(\xi) \Big)^{(3/2)+\epsilon} \\
&   & \times \; R^{\delta_{\mbox{\tiny trans}} \log(K^\epsilon \beta m)}
      + \; C' R^{-N+7} A_\alpha(H)^{q_1}
           \Big( \sum_\tau \| f_\tau \|_{L^2(S)}^2 \Big)^{(3/2)+\epsilon}
\end{eqnarray*}
To close the induction, it just suffices to prove that
\begin{displaymath}
C' D^{-2\epsilon} R^{(\log 4) \delta_{\mbox{\tiny trans}}} + C' R^{-N+7}
\leq 1.
\end{displaymath}
Since
\begin{displaymath}
D^{-2\epsilon} R^{(\log 4) \delta_{\mbox{\tiny trans}}}
= R^{-2 \epsilon \delta_{\mbox{\tiny deg}}}
  R^{(\log 4) \delta_{\mbox{\tiny trans}}}
= R^{-2 \epsilon^3 + (\log 4) \epsilon^6},
\end{displaymath}
it follows that the exponent of $R$ in both terms is negative and the
induction closes.

\subsection{The algebraic case}

Returning to the decomposition
\begin{eqnarray*}
\lefteqn{\int_{B_R} \!\! \mbox{Br}_\beta E f(x)^p H(x) dx} \\
& = & \sum_i \int_{B_R \cap O_i'} \mbox{Br}_\beta E f(x)^p H(x) dx +
      \int_{B_R \cap W} \mbox{Br}_\beta E f(x)^p H(x) dx,
\end{eqnarray*}
we now assume that the contribution from the cell-wall $W$ dominates. By
Lemma \ref{Lemma3.8}, we know that
\begin{eqnarray*}
\int_{B_R} \mbox{\rm Br}_\beta E f(x)^p H(x) dx
& \lct & \sum_{j,I} \int_{B_j \cap W} \mbox{Br}_{2\beta}
         E f_{I,j,\mbox{\tiny trans}}(x)^p H(x) dx \\
&      & + \; K^{100p} \sum_j \int_{B_j \cap W}
             \mbox{Bil}_{P,\delta} E f_{j,\mbox{\tiny tang}}(x)^p H(x) dx \\
&      & + \; O \Big( R^{-N+1} \sum_\tau \| f_\tau \|_{L^1(S)} \Big)^p
              A_\alpha(H)^{q_1} R^3.
\end{eqnarray*}

If the final $O$-term dominates, then the conclusion holds trivially (by
using the assumption $3 < p \leq 4$ and (\ref{bbdonL2(S)}) as before). 

If the tangential term dominates, then, recalling that 
$|\{ j \}| \lct R^{3\delta}$, we see that the conclusion holds by
(\ref{estonbil}). 

So we may assume that
\begin{equation}
\label{8thline}
\int_{B_R} \mbox{Br}_\beta E f(x)^p H(x) dx
\lct \sum_{j,I} \int_{B_j \cap W}
     \mbox{Br}_{2\beta} E f_{I,j,\mbox{\tiny trans}}(x)^p H(x) dx.
\end{equation}

Each ball $B_j$ has radius $R^{1-\delta}$. By induction on the radius, we
can therefore apply (\ref{underbil}) to each integral on the right-hand
side as soon as we verify that
$f_{I,j,\mbox{\tiny trans}} \in \Lambda(R^{1-\delta},K,2m,b)$. (The
multiplicity of the new decomposition of $S$ is $2m$ for the same reason
given above concerning the decomposition $\{ \tau' \}$ associated with the
function $g_i$.) By Lemma \ref{Lemma2.7} (applied to a single subset
$\mbb T_{j,\mbox{\tiny trans}} \subset \mbb T$), we have
\begin{displaymath}
\int_{B(\xi_0,R^{-(1-\delta)/2}) \cap S} |f_{\tau,j,\mbox{\tiny trans}}|^2
d\sigma(\xi) \lct \int_{B(\xi_0,R^{-(1-\delta)/2}) \cap S} |f_\tau|^2
d\sigma(\xi)
\end{displaymath}
for all $\xi_0 \in S$. Using (\ref{baverages}), we see that there is an
absolute constant $c$ such that
\begin{displaymath}
\int_{B(\xi_0,R^{-(1-\delta)/2}) \cap S} |f_{\tau,j,\mbox{\tiny trans}}|^2
d\sigma(\xi) \leq \frac{c R^\delta}{R^{(b+1)/2}}
\leq \frac{c R^{(b+1)\delta/2}}{R^{(b+1)/2}}
= \frac{c}{R^{(1-\delta)(b+1)/2}},
\end{displaymath}
where we have used the assumption that $b \geq 1$. Therefore,
\begin{displaymath}
c^{-1/2} f_{I,j,\mbox{\tiny trans}} \in \Lambda(R^{1-\delta},K,2m,b)
\end{displaymath}
and we may apply (\ref{underbil}) to each of the integrals on the
right-hand side of (\ref{8thline}) to get
\begin{eqnarray*}
\lefteqn{\int_{B_j} (\mbox{Br}_{2\beta}
         E f_{I,j,\mbox{\tiny trans}}(x)^p H(x) dx} \\
& \leq & c^2 C_\epsilon R^{(1-\delta)q_0\epsilon} A_\alpha(H)^{q_1}
         \Big( \sum_\tau \int |f_{\tau,j,\mbox{\tiny trans}}|^2 d\sigma(\xi)
         \Big)^{(3/2)+\epsilon} \\
&      & \times \;
         R^{(1-\delta)\delta_{\mbox{\tiny trans}} \log(4K^\epsilon\beta m)}.
\end{eqnarray*}

From Lemma \ref{Lemma3.5}, we know that a given tube in $\mbb T$ lies in
$\mbb T_{j,\mbox{\tiny trans}}$ for at most $\mbox{Poly}(D)$ values of $j$,
so Lemma \ref{Lemma2.7} (applied with $k=\mbox{Poly}(D)$) implies that
\begin{displaymath}
\sum_j \int |f_{\tau,j,\mbox{\tiny trans}}|^2 d\sigma(\xi)
\lct \mbox{Poly}(D) \int |f_\tau|^2 d\sigma(\xi),
\end{displaymath}
and hence
\begin{displaymath}
\Big( \sum_j \sum_{\tau \in I} \int |f_{\tau,j,\mbox{\tiny trans}}|^2
      d\sigma(\xi) \Big)^{\frac{3}{2}+\epsilon} \lct \mbox{Poly}(D)
\Big( \sum_\tau \int |f_\tau|^2 d\sigma(\xi) \Big)^{\frac{3}{2}+\epsilon},
\end{displaymath}
and hence
\begin{displaymath}
\sum_j \Big( \sum_{\tau \in I} \int |f_{\tau,j,\mbox{\tiny trans}}|^2
             d\sigma(\xi) \Big)^{\frac{3}{2}+\epsilon} \lct \mbox{Poly}(D)
\Big( \sum_\tau \int |f_\tau|^2 d\sigma(\xi) \Big)^{\frac{3}{2}+\epsilon}
\end{displaymath}
with the implicit constant independent of the cardinality of $\{ j \}$, and
hence
\begin{displaymath}
\sum_{j,I} \Big( \sum_{\tau \in I} \int |f_{\tau,j,\mbox{\tiny trans}}|^2
               d\sigma(\xi) \Big)^{\frac{3}{2}+\epsilon} \lct \mbox{Poly}(D)
\Big( \sum_\tau \int |f_\tau|^2 d\sigma(\xi) \Big)^{\frac{3}{2}+\epsilon}
\end{displaymath}
with the implicit constant depending on the cardinality of $\{ I \}$ (which
only depends on $K$, which is acceptable). Thus
\begin{eqnarray*}
\lefteqn{\int_{B_R} \mbox{Br}_\beta E f(x)^p H(x) dx} \\
& \leq & C' \, \mbox{Poly}(D) \, C_\epsilon R^{(1-\delta)q_0\epsilon}
         A_\alpha(H)^{q_1}
\Big( \sum_\tau \int |f_\tau|^2 d\sigma(\xi) \Big)^{\frac{3}{2}+\epsilon} \\
&   & \times \; R^{\delta_{\mbox{\tiny trans}} \log(4 K^\epsilon \beta m)}
      \\
& = & \Big( C' \, \mbox{Poly}(D) R^{-\delta q_0\epsilon}
            R^{(\log4) \delta_{\mbox{\tiny trans}}} \Big) C_\epsilon
      R^{q_0\epsilon} A_\alpha(H)^{q_1} \\
&   & \times \;
\Big( \sum_\tau \int |f_\tau|^2 d\sigma(\xi) \Big)^{\frac{3}{2}+\epsilon}
      R^{\delta_{\mbox{\tiny trans}} \log(K^\epsilon \beta m)}.
\end{eqnarray*}
To close the induction, we just have to prove that
\begin{displaymath}
C' \, \mbox{Poly}(D) R^{-\delta q_0\epsilon}
R^{(\log4) \delta_{\mbox{\tiny trans}}} \leq 1.
\end{displaymath}
But
\begin{displaymath}
C' \mbox{Poly}(D) R^{-\delta q_0\epsilon}
R^{(\log4) \delta_{\mbox{\tiny trans}}}
\leq R^{C'' \delta_{\mbox{\tiny deg}} - \delta q_0\epsilon
        + (\log 4) \delta_{\mbox{\tiny trans}}}
\leq R^{C'' \! \epsilon^4 - (1/2) \epsilon^3 + (\log 4) \epsilon^6},
\end{displaymath}
where we have used the assumption that $q_0 \geq 1/2$, so the induction
closes provided $\epsilon$ is sufficiently small.

\section{Estimates on the broad part}

This section and the next form the bulk of the proof of Theorem
\ref{mainjj}. In this section, we use Theorem \ref{biltobr} to estimate
various $L^p$ norms of the broad part of $Ef$ on the ball $B_R$ with respect
to the measure $H(x)dx$, where $H$ is a weight of dimension $\alpha$, as
defined in the Introduction. In view of the conditional formulation of 
Theorem \ref{biltobr}, this will be achieved by estimating the tangential 
part of $Ef$.

Following \cite{guth:poly}, we cover $B_j \cap W$ with cubes $Q$ of side
length $R^{1/2}$. For each cube $Q$, we let
$\mbb T_{j,\mbox{\tiny tang}, Q}$ be the set of tubes in
$\mbb T_{j,\mbox{\tiny tang}}$ that intersect $Q$. We know that
\begin{displaymath}
E f_{\tau,j,\mbox{\tiny tang}}
= \sum_{T \in \mbb T_{j,\mbox{\tiny tang}}} E f_{\tau,T}
= \sum_{T \in \mbb T_{j,\mbox{\tiny tang},Q}} E f_{\tau,T}
  + \sum_{T \in \mbb T_{j,\mbox{\tiny tang}}
          \setminus T_{j,\mbox{\tiny tang},Q}} E f_{\tau,T}.
\end{displaymath}

Let $x \in Q$. If $x \in T$, then $T$ must intersect $Q$, and it follows
that $T \in \mbb T_{j,\mbox{\tiny tang}, Q}$. If $x \not\in T$ and
$T \in \mbb T(\theta)$, then
$|E f_{\tau,T}(x)| \lct R^{-N} \| f_\tau \|_{L^1(\theta)}$, so
\begin{eqnarray*}
\sum_{T \in \mbb T_{j,\mbox{\tiny tang}}
      \setminus T_{j,\mbox{\tiny tang},Q}} |E f_{\tau,T}(x)|
& = & \sum_\theta
      \sum_{T \in (\mbb T(\theta) \cap \mbb T_{j,\mbox{\tiny tang}})
            \setminus T_{j,\mbox{\tiny tang},Q}} |E f_{\tau,T}(x)| \\
& \lct & \sum_\theta R^{-N} \| f_\tau \|_{L^1(\theta)} \\
&   =  & R^{-N} \| f_\tau \|_{L^1(S)},
\end{eqnarray*}
and so
\begin{displaymath}
E f_{\tau,j,\mbox{\tiny tang}}(x)
= \sum_{T \in \mbb T_{j,\mbox{\tiny tang},Q}} E f_{\tau,T}(x)
  + O \big( R^{-N} \| f_\tau \|_{L^1(S)} \big).
\end{displaymath}

Because of the definition of $T_{j,\mbox{\tiny tang}}$, it turns out that
all the tubes in $T_{j,\mbox{\tiny tang},Q}$ are nearly coplanar (provided
$\delta$ is sufficiently small for the radius of the tubes to be smaller
than the radius of $B_j$, i.e.\ provided
$R^{(1/2)+\delta} \leq R^{1-\delta}$). This led Guth to use the C\'{o}rdoba
$L^4$ argument and obtain the following bilinear estimate on $Q$.

\begin{alphlemma}[Lemma 3.10 in \cite{guth:poly}]
\label{Lemma3.10}
Suppose $0 < \delta \leq 1/4$. It $\tau_1$ and $\tau_2$ are non-adjacent
caps, then
\begin{eqnarray*}
\lefteqn{\int_Q |E f_{\tau_1,j,\mbox{\rm \tiny tang}}|^2
                |E f_{\tau_2,j,\mbox{\rm \tiny tang}}|^2 dx} \\
& \lct & R^{O(\delta)} R^{-1/2}
         \Big( \sum_{T_1 \in \mbb T_{j,\mbox{\rm \tiny tang},Q}}
               \| f_{\tau_1,T_1} \|_{L^2(S)}^2 \Big)
         \Big( \sum_{T_2 \in \mbb T_{j,\mbox{\rm \tiny tang},Q}}
               \| f_{\tau_2,T_2} \|_{L^2(S)}^2 \Big) \\
&      & + \;\;
O\Big( R^{-N+2} \big( \sum_\tau \| f_\tau \|_{L^2(S)}^2 \big)^2 \Big).
\end{eqnarray*}
\end{alphlemma}

To upgrade the estimate in Lemma \ref{Lemma3.10} from an estimate on $Q$ to
an estimate on $B_j \cap W$, Guth then considered the square function
\begin{displaymath}
S_{\tau,j,\mbox{\tiny tang}}=\Big( \sum_{T \in \mbb T_{j,\mbox{\tiny tang}}}
\big( \chi_{7T} R^{-1/2} \| f_{\tau,T} \|_{L^2(S)} \big)^2 \Big)^{1/2},
\end{displaymath}
where $7T$ is the tube with the same core line as $T$ but seven times the
radius. If $T \cap Q \not= \emptyset$, then $Q \subset 7T$. So
\begin{displaymath}
S_{\tau,j,\mbox{\tiny tang}}^2
\geq \sum_{T \in \mbb T_{j,\mbox{\tiny tang},Q}}
     \big( \chi_Q R^{-1/2} \| f_{\tau,T} \|_{L^2(S)} \big)^2
= \frac{\chi_Q}{R} \sum_{T \in \mbb T_{j,\mbox{\tiny tang},Q}}
  \| f_{\tau,T} \|_{L^2(S)}^2,
\end{displaymath}
and so
\begin{displaymath}
S_{\tau_1,j,\mbox{\tiny tang}}^2 S_{\tau_2,j,\mbox{\tiny tang}}^2
\geq \frac{\chi_Q}{R^2} \Big( \sum_{T_1 \in \mbb T_{j,\mbox{\tiny tang},Q}}
                              \| f_{\tau_1,T_1} \|_{L^2(S)}^2 \Big)
     \Big( \sum_{T_2 \in \mbb T_{j,\mbox{\tiny tang},Q}}
           \| f_{\tau_2,T_2} \|_{L^2(S)}^2 \Big),
\end{displaymath}
which gives
\begin{eqnarray*}
\lefteqn{\int_Q S_{\tau_1,j,\mbox{\tiny tang}}^2
                S_{\tau_2,j,\mbox{\tiny tang}}^2 dx} \\
& \geq & \frac{|Q|}{R^2} \Big( \sum_{T_1 \in \mbb T_{j,\mbox{\tiny tang},Q}}
                               \| f_{\tau_1,T_1} \|_{L^2(S)}^2 \Big)
         \Big( \sum_{T_2 \in \mbb T_{j,\mbox{\tiny tang},Q}}
               \| f_{\tau_2,T_2} \|_{L^2(S)}^2 \Big) \\
&  =   & R^{-1/2} \Big( \sum_{T_1 \in \mbb T_{j,\mbox{\tiny tang},Q}}
                               \| f_{\tau_1,T_1} \|_{L^2(S)}^2 \Big)
         \Big( \sum_{T_2 \in \mbb T_{j,\mbox{\tiny tang},Q}}
               \| f_{\tau_2,T_2} \|_{L^2(S)}^2 \Big).
\end{eqnarray*}
Lemma \ref{Lemma3.10} now implies that
\begin{eqnarray*}
\lefteqn{\int_Q |E f_{\tau_1,j,\mbox{\tiny tang}}|^2
         |E f_{\tau_2,j,\mbox{\tiny tang}}|^2 dx} \\
& \lct & R^{O(\delta)} \int_Q S_{\tau_1,j,\mbox{\tiny tang}}^2
         S_{\tau_2,j,\mbox{\tiny tang}}^2 dx +
O\Big( R^{-N+2} \big( \sum_\tau \| f_\tau \|_{L^2(S)}^2 \big)^2 \Big).
\end{eqnarray*}
Summing over all the $Q$ covering $B_j \cap W$, and expanding the definition
of the square function, this becomes
\begin{eqnarray*}
\lefteqn{\int_{B_j \cap W} |E f_{\tau_1,j,\mbox{\tiny tang}}|^2
         |E f_{\tau_2,j,\mbox{\tiny tang}}|^2 dx} \\
& \lct & R^{O(\delta)} \int_{B_j \cap W} S_{\tau_1,j,\mbox{\tiny tang}}^2
         S_{\tau_2,j,\mbox{\tiny tang}}^2 dx
   + O\Big( R^{-N+4} \big( \sum_\tau \| f_\tau \|_{L^2(S)}^2 \big)^2\Big) \\
& \lct & R^{O(\delta)}
         \sum_{T_1,T_2 \in \mbb T_{j,\mbox{\tiny tang}}} R^{-2}
         \| f_{\tau_1,T_1} \|_{L^2(S)}^2 \| f_{\tau_2,T_2} \|_{L^2(S)}^2
         \int_{B_j \cap W} \chi_{7T_1} \chi_{7T_2} dx \\
&      & + \;
      O\Big( R^{-N+4} \big( \sum_\tau \| f_\tau \|_{L^2(S)}^2 \big)^2 \Big).
\end{eqnarray*}
Since $T_1$ comes from the wave packet decomposition of $f_{\tau_1}$ and
$T_2$ comes from the wave packet decomposition of $f_{\tau_2}$, the angle
between $v(T_1)$ and $v(T_2)$ is $\gct K^{-1}$. So
\begin{displaymath}
\int_{\mbb R^3} \chi_{7T_1} \chi_{7T_2} dx
\lct \frac{R^{(1/2)+\delta}}{K^{-1}} R^{(1/2)+\delta} R^{(1/2)+\delta}
= K R^{(3/2)+3\delta},
\end{displaymath}
Inserting this bound in the last inequality, Guth obtained
\begin{eqnarray}
\label{startpttang}
\lefteqn{\int_{B_j \cap W} |E f_{\tau_1,j,\mbox{\tiny tang}}|^2
         |E f_{\tau_2,j,\mbox{\tiny tang}}|^2 dx} \nonumber \\
& \lct & R^{O(\delta)} R^{-1/2}
         \Big( \sum_{T_1 \in \mbb T_{j,\mbox{\tiny tang}}}
               \| f_{\tau_1,T_1} \|_{L^2(S)}^2 \Big)
         \Big( \sum_{T_2 \in \mbb T_{j,\mbox{\tiny tang}}}
               \| f_{\tau_2,T_2} \|_{L^2(S)}^2 \Big) \\
&      & + \; \nonumber
      O\Big( R^{-N+4} \big( \sum_\tau \| f_\tau \|_{L^2(S)}^2 \big)^2 \Big).
\end{eqnarray}

We are now in position to state our first estimate on the broad part of
$Ef$.

\begin{thm}
\label{thejapp}
Suppose $3/2 < \alpha \leq 3$, $H$ is a weight of dimension $\alpha$,
$b \geq 1$, and $p=2(4\alpha+3b)/(2\alpha+2b+1)$.

Then there is a constant $c_0$, which is independent of $b$, such that to
every $0 < \epsilon \leq \min[c_0,(p-3)/2]$ there are constants
$K=K(\epsilon)$ and $C_\epsilon$ so that
$\lim_{\epsilon \to 0} K(\epsilon)= \infty$ and
\begin{displaymath}
\int_{B_R} \mbox{\rm Br}_{K^{-\epsilon}} E f(x)^p H(x) dx
\leq C_\epsilon \, R^{(b+1)\epsilon/2} A_\alpha(H)^{1-(p/4)}
     \| f \|_{L^2(S)}^{3+2\epsilon}
\end{displaymath}
whenever $R \geq 1$, $f \in L^2(S)$, and
\begin{displaymath}
\int_{B(\xi_0,R^{-1/2}) \cap S} |f(\xi)|^2 d\sigma(\xi)
\leq \frac{1}{R^{(b+1)/2}}
\end{displaymath}
for all $\xi_0 \in S$.
\end{thm}

\begin{proof}
In view of Theorem \ref{biltobr}, we have to establish (\ref{estonbil}). To
guarantee the independence of $c_0$ from $b$, we will again be careful to
check that all constants appearing in this proof are independent of this
parameter.

Starting with (\ref{startpttang}) and noticing that $3 < p < 4$, we see by
H\"{o}lder's inequality that
\begin{eqnarray}
\label{fornewremark}
\lefteqn{\int_{B_j \cap W}
         \mbox{Bil}_{P,\delta} Ef_{j,\mbox{\tiny tang}}(x)^p H(x)dx}
         \nonumber \\
& \lct & \!\!\!\!\!\!\!\!\!\!\!\!\!\!\!\!
         \sum_{\tau_1, \tau_2 \, \mbox{\tiny non-adjacent}}
         \int_{B_j \cap W} |E f_{\tau_1,j,\mbox{\tiny tang}}|^{p/2}
                           |E f_{\tau_2,j,\mbox{\tiny tang}}|^{p/2} H(x)dx
        \nonumber \\
& \leq & \!\!\!\!\!\!\!\!\!\!\!\!\!\!\!\!
         \sum_{\tau_1, \tau_2 \, \mbox{\tiny non-adjacent}} \left(
         \int_{B_j \cap W} |E f_{\tau_1,j,\mbox{\tiny tang}}|^2
         |E f_{\tau_2,j,\mbox{\tiny tang}}|^2 dx \right)^{\frac{p}{4}}
         \Big( A_\alpha(H) R^\alpha \Big)^{1-\frac{p}{4}} \\
& \lct & R^{O(\delta)} A_\alpha(H)^{1-\frac{p}{4}}
         R^{-\frac{p}{8}+\alpha-\alpha\frac{p}{4}}
         \sum_{\tau_1,\tau_2} J^{\frac{p}{4}}
         \nonumber \\
&      & + \; O\Big( \big( A_\alpha(H) R^\alpha \big)^{1-\frac{p}{4}}
         R^{(-N+4)(p/4)}
         \big( \sum_\tau \| f_\tau \|_{L^2(S)}^2 \big)^{\frac{p}{2}} \Big),
         \nonumber
\end{eqnarray}
where
\begin{displaymath}
J= \Big( \sum_{T_1 \in \mbb T_{j,\mbox{\tiny tang}}}
         \| f_{\tau_1,T_1} \|_{L^2(S)}^2 \Big)
   \Big( \sum_{T_2 \in \mbb T_{j,\mbox{\tiny tang}}}
         \| f_{\tau_2,T_2} \|_{L^2(S)}^2 \Big).
\end{displaymath}
Following \cite{guth:poly}, we will bound $J$ in two different ways.

By part (i) of Proposition \ref{wave}, we have
\begin{eqnarray*}
\lefteqn{\sum_{T_1 \in \mbb T_{j,\mbox{\tiny tang}}}
         \| f_{\tau_1,T_1} \|_{L^2(S)}^2} \\
& & = \;
\sum_\theta \sum_{T_1 \in \mbb T(\theta) \cap \mbb T_{j,\mbox{\tiny tang}}}
\| f_{\tau_1,T_1} \|_{L^2(S)}^2
\; \lct \; \sum_\theta \| f_{\tau_1} \|_{L^2(\theta)}^2
           \; \lct \; \| f_{\tau_1} \|_{L^2(S)}^2.
\end{eqnarray*}
Likewise, $\sum_{T_2 \in \mbb T_{j,\mbox{\tiny tang}}}
\| f_{\tau_2,T_2} \|_{L^2(S)}^2 \lct \| f_{\tau_2} \|_{L^2(S)}^2$.
This gives the bound
\begin{displaymath}
J \lct \| f_{\tau_1} \|_{L^2(S)}^2 \| f_{\tau_2} \|_{L^2(S)}^2.
\end{displaymath}

On the other hand, Lemma \ref{Lemma3.6} tells us that
$\mbb T_{j,\mbox{\tiny tang}}$ contains tubes in only
$R^{O(\delta)} R^{1/2}$ different directions, so
\begin{eqnarray*}
\sum_{T_1 \in \mbb T_{j,\mbox{\tiny tang}}} \| f_{\tau_1,T_1} \|_{L^2(S)}^2
& \lct & \sum_{R^{(1/2)+O(\delta)} \mbox{\tiny caps } \theta} \;\;\;
         \sum_{T_1 \in \mbb T(\theta)} \| f_{\tau_1,T_1} \|_{L^2(S)}^2 \\
& \lct & \sum_{R^{(1/2)+O(\delta)} \mbox{\tiny caps } \theta}
         \int_\theta |f_{\tau_1}|^2 d\sigma(\xi) \\
& \lct & \frac{R^{(1/2)+O(\delta)}}{R^{(b+1)/2}}
         \; = \; R^{O(\delta)} R^{-b/2},
\end{eqnarray*}
where on the second line we used part (i) of Proposition \ref{wave} and on
the third line the fact that
$\int_\theta |f_\tau|^2 d\sigma(\xi) \lct R^{-(b+1)/2}$. Likewise with
$\tau_1, T_1$ replaced by $\tau_2, T_2$. Thus
\begin{displaymath}
J \lct R^{O(\delta)} R^{-b}.
\end{displaymath}

Putting the two bounds we now have on $J$ together, we see that
\begin{eqnarray*}
\lefteqn{\sum_{\tau_1,\tau_2} J^{\frac{p}{4}}
         \lct \Big( \sum_{\tau_1,\tau_2} J \Big)^{\frac{p}{4}}
         = \Big( \sum_{\tau_1,\tau_2} J \Big)^{(p-3-2\epsilon)/4}
           \Big( \sum_{\tau_1,\tau_2} J \Big)^{(3+2\epsilon)/4}} \\
& & \; \lct \; R^{O(\delta)} R^{-b(p-3-2\epsilon)/4}
               \Big( \sum_{\tau_1,\tau_2}
               \| f_{\tau_1} \|_{L^2(S)}^2 \| f_{\tau_2} \|_{L^2(S)}^2
               \Big)^{(3+2\epsilon)/4} \\
& & \; = \; R^{O(\delta)} R^{-b(p-3-2\epsilon)/4}
               \Big( \sum_\tau \| f_\tau \|_{L^2(S)}^2
               \Big)^{(3/2)+\epsilon}
\end{eqnarray*}
provided $p-3-2\epsilon \geq 0$. We note that since
$p - 3 - 2 \epsilon < 1 - 2 \epsilon$, all the implicit constants remain
independent of $b$. Therefore,
\begin{eqnarray*}
\lefteqn{\int_{B_j \cap W} \mbox{Bil}_{P,\delta}
         Ef_{j,\mbox{\tiny tang}}(x)^p H(x)dx} \\
& \lct & R^{O(\delta)} A_\alpha(H)^{1-\frac{p}{4}}
         R^{-\frac{p}{8}+\alpha-\alpha\frac{p}{4}}
         R^{-b \frac{p-3-2\epsilon}{4}}
         \Big( \sum_\tau \| f_\tau \|_{L^2(S)}^2\Big)^{\frac{3}{2}+\epsilon}
         \\
&      & + \; O\Big( \big( A_\alpha(H) R^\alpha \big)^{1-\frac{p}{4}}
         R^{(-N+4)(p/4)}
         \big( \sum_\tau \| f_\tau \|_{L^2(S)}^2 \big)^{\frac{p}{2}} \Big).
\end{eqnarray*}
Recalling from (\ref{bbdonL2(S)}) that
$\int |f_\tau|^2 d\sigma \lct R^{-(b-1)/2} \leq 1$, we have
\begin{eqnarray*}
\Big( \sum_\tau \| f_\tau \|_{L^2(S)}^2 \Big)^{\frac{p}{2}}
&  =   & \Big( \sum_\tau \| f_\tau \|_{L^2(S)}^2
         \Big)^{\frac{p-3-2\epsilon}{2}}
         \Big( \sum_\tau \| f_\tau \|_{L^2(S)}^2
         \Big)^{\frac{3}{2}+\epsilon} \\
& \lct & \Big( \sum_\tau \| f_\tau \|_{L^2(S)}^2
         \Big)^{\frac{3}{2}+\epsilon}
\end{eqnarray*}
provided $p-3-2\epsilon \geq 0$. Thus
\begin{eqnarray*}
\lefteqn{\int_{B_j \cap W} \mbox{Bil}_{P,\delta}
         Ef_{j,\mbox{\tiny tang}}(x)^p H(x)dx} \\
& \lct & R^{O(\delta)} A_\alpha(H)^{1-\frac{p}{4}}
         R^{\frac{b \epsilon}{2}} R^{\alpha+\frac{3b}{4}}
         R^{-\frac{p}{8}(1+2\alpha+2b)}
        \Big( \sum_\tau \| f_\tau \|_{L^2(S)}^2\Big)^{\frac{3}{2}+\epsilon}
        \\
&   =  & R^{O(\delta)} R^{b\epsilon/2} A_\alpha(H)^{1-(p/4)}
        \Big( \sum_\tau \| f_\tau \|_{L^2(S)}^2\Big)^{(3/2)+\epsilon}.
\end{eqnarray*}
We note that the implicit constant in the $R^{O(\delta)}$ factor does not 
depend on $b$.

Invoking Theorem \ref{biltobr} (with $q_2=b/2$, $q_1=1-(p/4)$, and
$q_0=(1/4)+(b/2)$), we conclude that to every sufficiently small $\epsilon$ 
there are constants $K=K(\epsilon)$ and $C_\epsilon$ such that
$\lim_{\epsilon \to 0} K(\epsilon)= \infty$ and
\begin{eqnarray*}
\lefteqn{\int_{B_R} \mbox{Br}_{\beta} E f(x)^p H(x) dx} \\
& \leq & C_\epsilon \, R^{\epsilon/4} R^{b\epsilon/2}
         A_\alpha(H)^{1-(p/4)}
         \Big( \sum_\tau \| f_\tau \|_{L^2(S)}^2 \Big)^{(3/2)+\epsilon}
         R^{\delta_{\mbox{\tiny trans}} \log(K^\epsilon \beta m)}
\end{eqnarray*}
for all $\beta \geq K^{-\epsilon}$, $m \geq 1$, $R \geq 1$, and
$f \in \Lambda(R,K,m,b)$.

Now suppose $f \in L^2(S)$ satisfies
\begin{displaymath}
\int_{B(\xi_0,R^{-1/2}) \cap S} |f(\xi)|^2 d\sigma(\xi)
\leq \frac{1}{R^{(b+1)/2}}
\end{displaymath}
for all $\xi_0 \in S$. Writing $f=\sum_\tau f_\tau$ with
$\mbox{supp} \, f_\tau \subset \tau$ and
$(\mbox{supp} \, f_\tau) \cap (\mbox{supp} \, f_{\tau'})= \emptyset$ if
$\tau \not= \tau'$, we see that $f \in \Lambda(R,K,m,b)$ and
$\sum_\tau \| f_\tau \|_{L^2(S)}^2= \| f \|_{L^2(S)}^2$. Applying the above
estimate with $\beta= K^{-\epsilon}$, we obtain the desired result.
\end{proof}

\begin{remark}
\label{soandso}
In dimension $n=2$, one needs to bound the broad part of $Ef$ by
$\big( \sum_\tau \| f_\tau \|_{L^2(S)}^2 \big)^{2+\epsilon}$ rather than
$\big( \sum_\tau \| f_\tau \|_{L^2(S)}^2 \big)^{(3/2)+\epsilon}$ for the
induction argument in the proof of Theorem \ref{biltobr} to work. In view of
the argument leading to (\ref{fornewremark}), however, one sees that
replacing $(3/2)+\epsilon$ by $2+\epsilon$ requires $p/4 > 1$. But when
$p/4 > 1$, one will not have the right exponent to apply H\"{o}lder's
inequality in the second paragraph of the proof of Theorem \ref{thejapp},
and as a result will not be able to exploit the dimensionality of $H$ and
arrive at the desired estimate.
\end{remark}

During the proof of the next theorem, we will need an estimate on
\begin{displaymath}
\int_{B_j} |E f_{\tau,j,\mbox{\tiny tang}}(x)|^2 H(x)dx.
\end{displaymath}
In \cite{guth:poly}, where the function $H$ was not present, this was
obtained via the standard local restriction estimate
\begin{displaymath}
\int_{B(0,R)} |E f(x)|^2 dx \lct R \, \| f \|_{L^2(S)}^2
\end{displaymath}
which holds for all $f \in L^2(S)$. One could still use this estimate here,
because $\| H \|_{L^\infty} \leq 1$, but then one would be repeating word
for word the argument from \cite{guth:poly} and would end up with the same
estimate on the broad part of $Ef$ as in that paper. In order to improve
matters -- via the induction argument of the next section -- we have to
involve the factor $A_\alpha(H)$ in our estimate, and we, therefore, have to
update the above local restriction estimate accordingly.

\begin{lemma}
\label{gentrace}
Suppose $0 < \alpha \leq 3$ and $H$ is a weight of dimension $\alpha$. Then
\begin{displaymath}
\int_{B(0,R)} |E f(x)|^2 H(x) dx \lct A_\alpha(H) R \, \| f \|_{L^2(S)}^2
\end{displaymath}
for all $R \geq 1$ and $f \in L^2(S)$.
\end{lemma}

\begin{proof}
We let $\eta$ be a $C_0^\infty$ function on $\mbb R^3$ such that
$|\widehat{\eta}| \geq 1$ on $B(0,1)$. Then
\begin{displaymath}
\int_{B(0,R)} |E f(x)|^2 H(x)dx
\leq \int |E f(x) \widehat{\eta_{R^{-1}}}(x)|^2 H(x)dx
= \int |\widehat{F}(x)|^2 H(x) dx,
\end{displaymath}
where $F=\eta_{R^{-1}} \ast f d\sigma$. Since $\eta$ is compactly supported
and $R \geq 1$, $F \ast F$ is supported in a ball $B(0,C)$. We let $\phi$ be
a Schwartz function on $\mbb R^3$ such that $\phi=1$ on $B(0,C)$. Then
$F \ast F= \phi \, (F \ast F)$, and, accordingly,
\begin{eqnarray*}
\int |\widehat{F}(x)|^2 H(x) dx
& = & \int \Big| \int \widehat{\phi}(x-y) \widehat{F \ast F}(y) dy \Big|
      H(x) dx \\
& \leq & \int |\widehat{F \ast F}(y)| \int |\widehat{\phi}(x-y)| H(x) dx dy.
\end{eqnarray*}
To estimate the inner integral, we let $B_l=B(y,2^l)$, and observe that
\begin{eqnarray*}
\int |\widehat{\phi}(x-y)| H(x) dx
&    = & \int_{B_0} |\widehat{\phi}(x-y)| H(x) dx + \sum_{l=1}^\infty
         \int_{B_l \setminus B_{l-1}} |\widehat{\phi}(x-y)| H(x) dx \\
& \leq & \int_{B_0} \frac{C_N H(x) dx}{(1+|x-y|)^N} + \sum_{l=1}^\infty
         \int_{B_l \setminus B_{l-1}} \frac{C_N H(x) dx}{(1+|x-y|)^N} \\
& \lct & A_\alpha(H).
\end{eqnarray*}
Therefore,
\begin{displaymath}
\int_{B(0,R)} |E f(x)|^2 H(x)dx
\lct A_\alpha(H) \int |\widehat{F \ast F}(y)| dy
= A_\alpha(H) \int |F(\xi)|^2 d\xi.
\end{displaymath}

By the definition of $F$, and the Cauchy-Schwarz inequality, we have
\begin{displaymath}
|F(\xi)|^2 \leq \Big( \int |\eta_{R^{-1}}(\xi-\zeta)| d\sigma(\zeta) \Big)
\Big( \int |\eta_{R^{-1}}(\xi-\zeta)| |f(\zeta)|^2 d\sigma(\zeta) \Big).
\end{displaymath}
Clearly,
\begin{displaymath}
\int |\eta_{R^{-1}}(\xi-\zeta)| d\sigma(\zeta)
= R^3 \int |\eta(R(\xi-\zeta))| d\sigma(\zeta)
\lct R^3 \sigma(B(\xi,R^{-1})) \lct R,
\end{displaymath}
so
\begin{displaymath}
\int |F(\xi)|^2 d\xi \lct
R \int |f(\zeta)|^2 \int |\eta_{R^{-1}}(\xi-\zeta)| d\xi d\sigma(\zeta)
= R \, \| \eta \|_{L^1} \| f \|_{L^2(S)}^2,
\end{displaymath}
and so
\begin{displaymath}
\int_{B(0,R)} |E f(x)|^2 H(x)dx \lct A_\alpha(H) R \, \| f \|_{L^2(S)}^2
\end{displaymath}
as claimed.
\end{proof}

This brings us to our second estimate on the broad part of $Ef$.

\begin{thm}
\label{themapp}
Suppose $0 < \alpha \leq 3$, $H$ is a weight of dimension $\alpha$, and
$p=13/4$.

Then there is a constant $c$, with $0 < c \leq (p-3)/2$, such that to every
$0 < \epsilon < c$ there are constants $K=K(\epsilon)$ and $C_\epsilon$ so
that $\lim_{\epsilon \to 0} K(\epsilon)= \infty$ and
\begin{displaymath}
\int_{B_R} \mbox{\rm Br}_{K^{-\epsilon}} E f(x)^p H(x) dx
\leq C_\epsilon \, R^\epsilon A_\alpha(H)^{2-(p/2)}
     \| f \|_{L^2(S)}^{3+2\epsilon}
\end{displaymath}
whenever $R \geq 1$, $f \in L^2(S)$, and
\begin{displaymath}
\int_{B(\xi_0,R^{-1/2}) \cap S} |f(\xi)|^2 d\sigma(\xi) \leq \frac{1}{R}
\end{displaymath}
for all $\xi_0 \in S$.
\end{thm}

\begin{proof}
The argument we use here is very close to the one presented in \S 3.4 of
\cite{guth:poly}. Our starting point is again (\ref{startpttang}), which
gives the following $L^4$ bound on the bilinear term:
\begin{eqnarray*}
\lefteqn{\int_{B_j \cap W} \mbox{Bil}_{P,\delta}
         Ef_{j,\mbox{\tiny tang}}(x)^4 H(x)dx
\; \leq \; \int_{B_j \cap W} \mbox{Bil}_{P,\delta}
           Ef_{j,\mbox{\tiny tang}}(x)^4 dx} \\
& \lct & \sum_{\tau_1, \tau_2 \, \mbox{\tiny non-adjacent}}
         \int_{B_j \cap W} |E f_{\tau_1,j,\mbox{\tiny tang}}|^2
                           |E f_{\tau_2,j,\mbox{\tiny tang}}|^2 dx
         \; \lct \; R^{O(\delta)} R^{-1/2} M + {\mathcal E},
\end{eqnarray*}
where
\begin{displaymath}
M= \sum_{\tau_1,\tau_2}
\Big( \sum_{T_1 \in \mbb T_{j,\mbox{\tiny tang}}}
      \| f_{\tau_1,T_1} \|_{L^2(S)}^2 \Big)
\Big( \sum_{T_2 \in \mbb T_{j,\mbox{\tiny tang}}}
      \| f_{\tau_2,T_2} \|_{L^2(S)}^2 \Big)
\end{displaymath}
and ${\mathcal E}=
O \Big( R^{-N+4} \big( \sum_\tau \| f_\tau \|_{L^2(S)}^2 \big)^2 \Big)$.

On the other hand, Lemma \ref{gentrace} tells us that
\begin{eqnarray*}
\lefteqn{\int_{B_j \cap W} \mbox{Bil}_{P,\delta}
         Ef_{j,\mbox{\tiny tang}}(x)^2 H(x) dx}
\\
& \lct & \sum_{\tau_1, \tau_2 \, \mbox{\tiny non-adjacent}}
         \int_{B_j \cap W} |E f_{\tau_1,j,\mbox{\tiny tang}}|
                           |E f_{\tau_2,j,\mbox{\tiny tang}}| H(x) dx \\
& \leq & \sum_{\tau_1,\tau_2}
         \| E f_{\tau_1,j,\mbox{\tiny tang}} \|_{L^2(B_j,H(x)dx)}
         \| E f_{\tau_2,j,\mbox{\tiny tang}} \|_{L^2(B_j,H(x)dx)} \\
& \lct & A_\alpha(H) R^{1-\delta} \sum_{\tau_1,\tau_2}
         \| f_{\tau_1,j,\mbox{\tiny tang}} \|_{L^2(S)}
         \| f_{\tau_2,j,\mbox{\tiny tang}} \|_{L^2(S)} \\
&   =  & A_\alpha(H) R^{1-\delta}
         \Big( \sum_\tau \| f_{\tau,j,\mbox{\tiny tang}} \|_{L^2(S)} \Big)^2
\; \lct \; A_\alpha(H) R
         \sum_\tau \| f_{\tau,j,\mbox{\tiny tang}} \|_{L^2(S)}^2.
\end{eqnarray*}

Interpolating between the $L^2$ estimate and the $L^4$ estimate, we get for
all $2 \leq p \leq 4$,
\begin{eqnarray*}
\lefteqn{\int_{B_j \cap W} \mbox{Bil}_{P,\delta}
         Ef_{j,\mbox{\tiny tang}}(x)^p H(x)dx} \\
&  =   & \int_{B_j \cap W} \mbox{Bil}_{P,\delta}
         Ef_{j,\mbox{\tiny tang}}(x)^{(2-p/2)(2)}
         \mbox{Bil}_{P,\delta}
         Ef_{j,\mbox{\tiny tang}}(x)^{(-1+p/2)(4)} H(x)dx \\
& \leq & \Big( \int_{B_j \cap W} \mbox{Bil}_{P,\delta}
         Ef_{j,\mbox{\tiny tang}}(x)^2 H(x)dx \Big)^{2-\frac{p}{2}} \\
&      & \times \; \Big( \int_{B_j \cap W} \mbox{Bil}_{P,\delta}
         Ef_{j,\mbox{\tiny tang}}(x)^4 H(x)dx
         \Big)^{-1+\frac{p}{2}} \\
& \lct & \Big( A_\alpha(H) R \sum_\tau
               \| f_{\tau,j,\mbox{\tiny tang}} \|_{L^2(S)}^2
         \Big)^{2-\frac{p}{2}}
         \Big( R^{O(\delta)} R^{-1/2} M + {\mathcal E} \Big)^{\frac{p}{2}-1}
         \\
& \leq & R^{O(\delta)} A_\alpha(H)^{2-\frac{p}{2}}
         R^{\frac{5}{2}-\frac{3}{4}p}
         \Big( \sum_\tau \| f_{\tau,j,\mbox{\tiny tang}} \|_{L^2(S)}^2
         \Big)^{2-\frac{p}{2}} M^{\frac{p}{2}-1} \\
&      & + \; A_\alpha(H)^{2-\frac{p}{2}} R^{2-\frac{p}{2}}
         \Big( \sum_\tau \| f_{\tau,j,\mbox{\tiny tang}} \|_{L^2(S)}^2
         \Big)^{2-\frac{p}{2}} {\mathcal E}^{\frac{p}{2}-1}
\end{eqnarray*}
(since $\frac{p}{2} - 1 \leq 1$). By Lemma \ref{Lemma2.7} (applied to a
single subset $\mbb T_{j,\mbox{\tiny tang}} \subset \mbb T$), we have
$\| f_{\tau,j,\mbox{\tiny tang}}\|_{L^2(S)}^2 \lct \| f_\tau \|_{L^2(S)}^2$,
so $\Big( \sum_\tau \| f_{\tau,j,\mbox{\tiny tang}} \|_{L^2(S)}^2
    \Big)^{2-\frac{p}{2}}
    \lct \Big( \sum_\tau \| f_\tau \|_{L^2(S)}^2 \Big)^{2-\frac{p}{2}}$,
and so
\begin{eqnarray*}
\lefteqn{A_\alpha(H)^{2-\frac{p}{2}} R^{2-\frac{p}{2}}
         \Big( \sum_\tau \| f_{\tau,j,\mbox{\tiny tang}} \|_{L^2(S)}^2
         \Big)^{2-\frac{p}{2}} {\mathcal E}^{\frac{p}{2}-1}} \\
& \lct & A_\alpha(H)^{2-\frac{p}{2}} R^{2-\frac{p}{2}}
         \Big( \sum_\tau \| f_\tau \|_{L^2(S)}^2 \Big)^{2-\frac{p}{2}}
        \Big( R^{-N+4} \big( \sum_\tau \| f_\tau \|_{L^2(S)}^2 \big)^2
        \Big)^{\frac{p}{2}-1} \\
& = & A_\alpha(H)^{2-\frac{p}{2}} R^{1-(N-3)(p-2)/2}
      \Big( \sum_\tau \| f_\tau \|_{L^2(S)}^2 \Big)^{\frac{p}{2}}.
\end{eqnarray*}
Thus the error term can be handled with the aid of (\ref{bbdonL2(S)})
as in the proof of the previous theorem.

We write the main term as
\begin{displaymath}
R^{O(\delta)} A_\alpha(H)^{2-\frac{p}{2}} R^{\frac{5}{2}-\frac{3}{4}p}
\Big( \sum_\tau \| f_{\tau,j,\mbox{\tiny tang}} \|_{L^2(S)}^2
\Big)^{2-\frac{p}{2}} M^\nu M^{\frac{p}{2}-1-\nu},
\end{displaymath}
where $\nu$ is a positive number that will be determined later. Following
\cite{guth:poly}, we will estimate $M$ in two different ways. As we saw
during the proof of Theorem \ref{thejapp}, part (i) of Proposition
\ref{wave} tells us that
\begin{displaymath}
\sum_{T_1 \in \mbb T_{j,\mbox{\tiny tang}}}
\| f_{\tau_1,T_1} \|_{L^2(S)}^2 \lct \| f_{\tau_1} \|_{L^2(S)}^2
\hspace{0.23in} \mbox{and} \hspace{0.23in}
\sum_{T_2 \in \mbb T_{j,\mbox{\tiny tang}}}
\| f_{\tau_2,T_2} \|_{L^2(S)}^2 \lct \| f_{\tau_2} \|_{L^2(S)}^2,
\end{displaymath}
so $M \lct \big( \sum_\tau \| f_\tau \|_{L^2(S)}^2 \big)^2$, and so
\begin{displaymath}
\Big( \sum_\tau \| f_{\tau,j,\mbox{\tiny tang}} \|_{L^2(S)}^2
\Big)^{2-\frac{p}{2}} M^\nu \lct
\Big( \sum_\tau \| f_\tau \|_{L^2(S)}^2 \Big)^{2-\frac{p}{2}+2\nu}.
\end{displaymath}
In order to estimate $M^{(p/2)-1-\nu}$, we use Lemma \ref{Lemma3.6} as in
the proof of the previous theorem to get
$\sum_{T_1 \in \mbb T_{j,\mbox{\tiny tang}}} \|f_{\tau_1,T_1}\|_{L^2(S)}^2
\lct R^{O(\delta)} R^{-1/2}$ (recall that $b=1$ in this theorem), and
likewise with $\tau_1, T_1$ replaced by $\tau_2, T_2$. Therefore,
\begin{displaymath}
M^{\frac{p}{2}-1-\nu}
\lct R^{O(\delta)} \Big( \frac{1}{R} \Big)^{\frac{p}{2}-1-\nu}
\end{displaymath}
provided $p \geq 2(1+\nu)$.

Putting the bounds together, we arrive at
\begin{eqnarray*}
\lefteqn{\int_{B_j \cap W} \mbox{Bil}_{P,\delta}
         Ef_{j,\mbox{\tiny tang}}(x)^p H(x)dx} \\
& \lct & R^{O(\delta)} A_\alpha(H)^{2-\frac{p}{2}}
         R^{\frac{5}{2}-\frac{3}{4}p}
         \Big( \sum_\tau \| f_\tau \|_{L^2(S)}^2 \Big)^{2-\frac{p}{2}+2\nu}
         \Big( \frac{1}{R} \Big)^{\frac{p}{2}-1-\nu} \\
&      & + \, R^{1-(N-3)(p-2)/2} A_\alpha(H)^{2-\frac{p}{2}}
           \Big( \sum_\tau \| f_\tau \|_{L^2(S)}^2 \Big)^{\frac{p}{2}} \\
& \lct & R^{O(\delta)} A_\alpha(H)^{2-\frac{p}{2}}
         R^{\frac{5}{2}-\frac{3}{4}p}
         \Big( \frac{1}{R} \Big)^{\frac{p}{2}-1-\nu}
        \Big( \sum_\tau \| f_\tau \|_{L^2(S)}^2 \Big)^{2-\frac{p}{2}+2\nu}
\end{eqnarray*}
provided $p \geq 2(1+\nu)$.

We now determine $\nu$. We need to have
$2-\frac{p}{2}+2\nu=\frac{3}{2}+\epsilon$, so
$\nu=\frac{p}{4}-\frac{1}{4}+\frac{\epsilon}{2}$. Then
\begin{displaymath}
\frac{p}{2}-1-\nu = \frac{p}{4}-\frac{3}{4}-\frac{\epsilon}{2} \geq 0
\;\; \Longleftrightarrow \;\; p \geq 3 + 2 \epsilon,
\end{displaymath}
and
\begin{displaymath}
R^{\frac{5}{2}-\frac{3}{4}p} \Big( \frac{1}{R} \Big)^{\frac{p}{2}-1-\nu}
= R^{\frac{5}{2}-\frac{3}{4}p}
        \Big( \frac{1}{R} \Big)^{\frac{p}{4}-\frac{3}{4}-\frac{\epsilon}{2}}
= R^{\epsilon/2} R^{\frac{13}{4}-p} = R^{\epsilon/2}.
\end{displaymath}
Thus
\begin{displaymath}
\int_{B_j \cap W} \mbox{Bil}_{P,\delta} Ef_{j,\mbox{\tiny tang}}(x)^p H(x)dx
\lct R^{O(\delta)} R^{\epsilon/2} A_\alpha(H)^{2-\frac{p}{2}}
     \Big( \sum_\tau \| f_\tau \|_{L^2(S)}^2 \Big)^{\frac{3}{2}+\epsilon}.
\end{displaymath}
We note that the implicit constant in the $R^{O(\delta)}$ factor does not 
depend on $b$.

Invoking Theorem \ref{biltobr} (with $q_2=1/2$, $q_1=2-(p/2)$, and
$q_0=3/4$), we conclude that to every sufficiently small $\epsilon$ there 
are constants $K=K(\epsilon)$ and $C_\epsilon$ such that
$\lim_{\epsilon \to 0} K(\epsilon)= \infty$ and
\begin{eqnarray*}
\lefteqn{\int_{B_R} \mbox{Br}_{\beta} E f(x)^p H(x) dx} \\
& \leq & C_\epsilon R^{3\epsilon/4} A_\alpha(H)^{2-(p/2)}
         \Big( \sum_\tau \| f_\tau \|_{L^2(S)}^2 \Big)^{(3/2)+\epsilon}
         R^{\delta_{\mbox{\tiny trans}} \log(K^\epsilon \beta m)}
\end{eqnarray*}
for all $\beta \geq K^{-\epsilon}$, $m \geq 1$, $R \geq 1$, and
$f \in \Lambda(R,K,m,1)$.

Given a function $f \in L^2(S)$ that satisfies
\begin{displaymath}
\int_{B(\xi_0,R^{-1/2}) \cap S} |f(\xi)|^2 d\sigma(\xi)
\leq \frac{1}{R}
\end{displaymath}
for all $\xi_0 \in S$, then, writing $f=\sum_\tau f_\tau$ with
$\mbox{supp} \, f_\tau \subset \tau$ and
$(\mbox{supp} \, f_\tau) \cap (\mbox{supp} \, f_{\tau'})= \emptyset$ if
$\tau \not= \tau'$, we see that $f \in \Lambda(R,K,m,1)$ and
$\sum_\tau \| f_\tau \|_{L^2(S)}^2= \| f \|_{L^2(S)}^2$. Applying the above
estimate with $\beta= K^{-\epsilon}$, we obtain the desired result.
\end{proof}

\section{Parabolic scaling and the main induction argument}

Suppose $\tau$ is a cap in $S$ of center $(\omega_0,h(\omega_0))$ and radius
$r \leq 1$. Following \cite{guth:poly}, for $\omega \in B^2(\omega_0,r)$, we
define
\begin{displaymath}
\widetilde{h}(\omega)
= h(\omega)-h(\omega_0)-(\omega-\omega_0) \cdot \nabla h(\omega_0).
\end{displaymath}
Also, for $|\eta| \leq 1$, we define
\begin{displaymath}
h_1(\eta)= r^{-2} \widetilde{h}(\omega_0+r \eta)
= r^{-2} \widetilde{h}(\omega)
\end{displaymath}
and we let $S_1$ be the graph of $h_1$ over $B^2(0,1)$.

To every function $f$ on $\tau$ we associate a function $g$ on $S_1$ defined
by
\begin{displaymath}
g(\eta,h_1(\eta))= r^2 f(\omega_0+r\eta,h(\omega_0+r\eta))
                   J_h(\omega_0+r\eta) J_{h_1}(\eta)^{-1},
\end{displaymath}
where $J_h=\sqrt{1+|\nabla h|^2}$ and $J_{h_1}=\sqrt{1+|\nabla h_1|^2}$.
Then
\begin{eqnarray*}
\lefteqn{\Big| \int_{B^2(\omega_0,r)} f(\omega,h(\omega))
e^{-2 \pi i \big( (x_1,x_2) \cdot \omega + x_3 h(\omega) \big)} J_h(\omega)
d\omega \Big|} \\
& = & \Big| \int_{B^2(\omega_0,r)} f(\omega,h(\omega))
      e^{-2 \pi  i \big(
      ((x_1,x_2) + x_3 \nabla h(\omega_0)) \cdot (\omega-\omega_0)
       +x_3 r^2 h_1((\omega-\omega_0)/r) \big)} \\
&   & \mbox{} \hspace{0.48in} \times J_h(\omega) d\omega \Big| \\
& = & \Big| \int_{B^2(0,1)} f(\omega_0 + r \eta,h(\omega_0 + r \eta))
      e^{-2 \pi  i \big( ((x_1,x_2) + x_3 \nabla h(\omega_0)) \cdot (r \eta)
       + x_3 r^2 h_1(\eta) \big)} \\
&   & \mbox{} \hspace{0.48in} \times J_h(\omega_0+r \eta) r^2 d\eta \Big| \\
& = & \Big| \int_{B^2(0,1)} g(\eta,h_1(\eta))
      e^{-2 \pi  i \big( ((rx_1,rx_2) + r x_3 \nabla h(\omega_0)) \cdot \eta
      + r^2 x_3 h_1(\eta) \big)} J_{h_1}(\eta) d\eta \Big|,
\end{eqnarray*}
where we have applied the change of variables $\eta=(\omega-\omega_0)/r$.
Thus
\begin{displaymath}
\big| E_S f(x) \big|
=\big| E_{S_1} g \big( (rx_1,rx_2) + r x_3 \nabla h(\omega_0), r^2 x_3 \big)
 \big|
\end{displaymath}
for all $x \in \mbb R^3$. Also,
\begin{eqnarray*}
\lefteqn{\int_{B^2(0,1)} |g(\eta,h_1(\eta))|^2 J_{h_1}(\eta) d\eta} \\
& = & r^2 \int_{B^2(0,1)}
      |f(\omega_0 + r \eta,h(\omega_0 + r \eta))|^2 J_h(\omega_0 + r \eta)^2
      J_{h_1}(\eta)^{-1} r^2 d\eta,
\end{eqnarray*}
so applying the change of variables $\omega=\omega_0+r\eta$, we see that
\begin{eqnarray*}
\lefteqn{\int_{B^2(0,1)} |g(\eta,h_1(\eta))|^2 J_{h_1}(\eta) d\eta} \\
& = & r^2 \int_{B^2(\omega_0,r)} |f(\omega,h(\omega))|^2 J_h(\omega)^2
      J_{h_1}((\omega-\omega_0)/r))^{-1} d\omega.
\end{eqnarray*}
Clearly, $J_{h_1}((\omega-\omega_0)/r)) \geq 1$ and (by (\ref{bdongrad}))
$J_h(\omega) \leq 3$ for all $\omega \in B^2(\omega_0,r)$, so
\begin{displaymath}
\int_{B^2(0,1)} |g(\eta,h_1(\eta))|^2 J_{h_1}(\eta) d\eta
\leq 3r^2 \int_{B^2(\omega_0,r)}|f(\omega,h(\omega))|^2 J_h(\omega) d\omega.
\end{displaymath}

Define the linear map $T: \mbb R^3 \to \mbb R^3$ by
\begin{displaymath}
Tx= \big( (rx_1,rx_2) + r x_3 \nabla h(\omega_0), r^2 x_3 \big).
\end{displaymath}
As we saw above, we have
\begin{displaymath}
|E_S f(x)|=|E_{S_1}g(Tx)|,
\end{displaymath}
so that
\begin{displaymath}
\int_{B(0,R)} |E_S f(x)|^p H(x) dx= \int_{B(0,R)} |E_{S_1}g(Tx)|^p H(x) dx.
\end{displaymath}
Applying the change of variables $u=Tx$, this becomes
\begin{displaymath}
\int_{B(0,R)} |E_S f(x)|^p H(x) dx
= \int_{T(B(0,R))} |E_{S_1} g(u)|^p H(T^{-1}u) r^{-4} du.
\end{displaymath}
Since (by (\ref{bdongrad}))
\begin{displaymath}
|Tx| \leq r |(x_1,x_2)|+ r |x_3| |\nabla h(\omega_0)| + r^2 |x_3|
\leq \Big( 1 + \frac{7}{4}|\omega_0| + r \Big) r |x|,
\end{displaymath}
it follows that $T(B(0,R)) \subset B(0,4rR)$, and hence
\begin{equation}
\label{bdesbyes1}
\int_{B(0,R)} |E_S f(x)|^p H(x) dx
\leq \int_{B(0,4rR)} |E_{S_1} g(u)|^p H(T^{-1}u) r^{-4} du.
\end{equation}

Let $H'=H \circ T^{-1}$. We need to study
\begin{displaymath}
\int_{B(u_0,t)} H'(u) du = \int_{B(u_0,t)} H(T^{-1}u) du
\end{displaymath}
for $u_0 \in \mbb R^3$ and $t \geq 1$. We begin by noticing that
\begin{displaymath}
u \in B(u_0,t) \;\; \Longleftrightarrow \;\; |Tv| \leq t
\;\; \Longleftrightarrow \;\; v \cdot A v \leq t^2,
\end{displaymath}
where $v= T^{-1}u - T^{-1}u_0$ and
\begin{displaymath}
A= T^t T = \left( \begin{array}{ccc}
                  r^2 \; & \; 0   \; & \; r^2 \partial_1 h(\omega_0) \\
                  0   \; & \; r^2 \; & \; r^2 \partial_2 h(\omega_0) \\
                  r^2 \partial_1 h(\omega_0) \; & \;
                  r^2 \partial_2 h(\omega_0) \; & \;
                  r^2 |\nabla h(\omega_0)|^2 + r^4
                  \end{array} \right).
\end{displaymath}
The characteristic polynomial of this matrix is
\begin{displaymath}
|A - \lambda I| = \big( r^2-\lambda \big)
\big( \lambda^2 - (r^2 J_h(\omega_0)^2 + r^4) \lambda +r^6).
\end{displaymath}
Since $J_h(\omega_0) \geq 1$, we have
\begin{displaymath}
(r^2 J_h(\omega_0)^2+r^4)^2-4r^6 \geq (r^2+r^4)^2-4r^6 = r^4+r^8-2r^6
=(r^2-r^4)^2 \geq 0,
\end{displaymath}
so the eigenvalues of $A$ are
\begin{displaymath}
\lambda_1= r^2, \hspace{0.25in}
\lambda_2=
\frac{2r^6}{r^2J_h(\omega_0)^2+r^4+\sqrt{(r^2 J_h(\omega_0)^2+r^4)^2-4r^6}},
\end{displaymath}
and
\begin{displaymath}
\lambda_3=
\frac{r^2 J_h(\omega_0)^2+r^4+\sqrt{(r^2 J_h(\omega_0)^2 + r^4)^2-4r^6}}{2}.
\end{displaymath}
Since $1 \leq J_h(\omega_0) \leq 3$, it follows that
\begin{displaymath}
\lambda_2 \geq \frac{r^6}{r^2 J_h(\omega_0)^2+r^4} \geq \frac{r^6}{9r^2+r^4}
\geq \frac{r^4}{10} \hspace{0.25in} \mbox{and} \hspace{0.25in}
\lambda_3 \geq \frac{r^2}{2}.
\end{displaymath}
Therefore,
the image of $B(u_0,t)$ under $T^{-1}$ is contained in an ellipsoid
of center $x_0=T^{-1}u_0$, two short principal axes of length
$2\sqrt{2} \, t/r$, and long principal axis of length $2\sqrt{10} \, t/r^2$.
This ellipsoid can be covered by balls
$B(x_1,4t/r), \ldots, B(x_N,4t/r)$ with $N \leq 3/r$, so after applying
the change of variables $x=T^{-1}u$ we see that
\begin{eqnarray*}
\lefteqn{\int_{B(u_0,t)} H'(u) du = \int_{T^{-1}(B(u_0,t))} H(x) r^4 dx}
\\
& & \leq r^4 \sum_{j=1}^N \int_{B(x_j,4t/r)} H(x) dx
    \leq r^4 N A_\alpha(H) \Big( \frac{4t}{r} \Big)^\alpha
    \leq (3)(4^3) r^{3-\alpha} A_\alpha(H) t^\alpha
\end{eqnarray*}
for all $u_0 \in \mbb R^3$ and $t \geq 1$. Thus
\begin{equation}
\label{hprime192h}
A_\alpha(H') \leq (192) r^{3-\alpha} A_\alpha(H).
\end{equation}

\begin{thm}
\label{parabscaling}
Let $0 < \alpha \leq 3$, $3 \leq p \leq 4$, $2 \leq \gamma \leq 3$,
$0 \leq q_1 \leq 1$, $q_2 \geq 0$, and $c > 0$.

Suppose that we have the following estimate on the broad part of $Ef$: to
every $\epsilon \in (0,c)$ there are constants $K(\epsilon)$ and
$\bar{C}_\epsilon$ such that $\lim_{\epsilon \to 0} K(\epsilon)= \infty$ and
\begin{displaymath}
\int_{B(0,R)} \mbox{\rm Br}_{K^{-\epsilon}} Ef(x)^p H(x)dx
\leq \bar{C}_\epsilon R^\epsilon A_\alpha(H)^{q_1}
     R^{q_2} \| f \|_{L^2(S)}^{\gamma} \| f \|_{L^\infty(S)}^{p-\gamma}
\end{displaymath}
for all radii $R \geq 1$, weights $H$ of dimension $\alpha$, functions $h$
satisfying conditions {\rm (i)--(iv)} of {\rm Assumption \ref{graphofh}},
and functions $f \in L^\infty(S)$.

If $2p-\alpha-1-\gamma > 0$, then there is a constant $c'$, which only
depends on $\alpha, p$, and $\gamma$, such that for $0 < \epsilon < c'$ we
have
\begin{displaymath}
\int_{B(0,R)} |Ef(x)|^p H(x)dx \leq C_\epsilon R^\epsilon
\Big( \max \big[ A_\alpha(H), A_\alpha(H)^{q_1} \big] \Big) R^{q_2}
\| f \|_{L^2(S)}^{\gamma} \| f \|_{L^\infty(S)}^{p-\gamma},
\end{displaymath}
with
\begin{displaymath}
C_\epsilon=
2 \big( \bar{C}_\epsilon + 10^4 \sigma(S)^4 \big),
\end{displaymath}
for all radii $R \geq 1$, weights $H$ of dimension $\alpha$, functions $h$
satisfying conditions {\rm (i)--(iv)} of {\rm Assumption \ref{graphofh}},
and functions $f \in L^\infty(S)$.
\end{thm}

\begin{proof}
We are going to prove the theorem by induction on $R$. The estimate is true
for $1 \leq R \leq 10$:
\begin{eqnarray*}
\int_{B(0,R)} |Ef(x)|^p H(x)dx
& \leq & \| f \|_{L^1(S)}^p \int_{B(0,10)} H(x)dx \\
&  =   & \Big( \int_{B(0,10)} H(x)dx \Big)^{1-q_1+q_1} \| f \|_{L^1(S)}^p \\
& \leq & R^{\epsilon+q_2} |B(0,10)|^{1-q_1} A_\alpha(H)^{q_1}
         (10^{\alpha q_1})
         \| f \|_{L^1(S)}^{\gamma} \| f \|_{L^1(S)}^{p-\gamma} \\
& \leq & (5^{1-q_1}) (10^3) \sigma(S)^4 R^\epsilon A_\alpha(H)^{q_1} R^{q_2}
         \| f \|_{L^2(S)}^{\gamma} \| f \|_{L^\infty(S)}^{p-\gamma},
\end{eqnarray*}
where we have used the fact that $\sigma(S) \geq 1$ for all $h$.

Suppose $R \geq 10$ and our estimate is true for all functions $h$
satisfying conditions (i)--(iv) of Assumption \ref{graphofh}, weights $H$ of
dimension $\alpha$, and all radii in the interval $[1,R/2]$. Then
\begin{eqnarray*}
\lefteqn{\int_{B(0,R)} |Ef(x)|^p H(x)dx} \\
& \leq & \int_{B(0,R)} \mbox{Br}_{K^{-\epsilon}} Ef(x)^p H(x)dx
            + K^\epsilon \sum_\tau \int_{B(0,R)} |Ef_\tau(x)|^p H(x)dx \\
& \leq & \bar{C}_\epsilon R^\epsilon A_{\alpha}(H)^{q_1} R^{q_2}
         \| f \|_{L^2(S)}^{\gamma} \| f \|_{L^\infty(S)}^{p-\gamma}
         + K^\epsilon \sum_\tau \int_{B(0,R)} |Ef_\tau(x)|^p H(x)dx.
\end{eqnarray*}
We have $K$ functions $f_\tau$ each supported in a cap of diameter $r=1/K$,
and hence in a set of the form $B(\xi_\tau,\rho) \cap S$ with
$r \leq \rho \leq 3 r$ and $\xi_\tau \in S$. We are going to use parabolic
scaling and the induction hypothesis to bound
$\sum_\tau \int_{B(0,R)} |E f_\tau|^p H(x) dx$. We let $\phi$ be a
non-negative Schwartz function on $\mbb R^3$ such that $\phi \geq 1$ on the
unit ball and $\widehat{\phi}$ is supported in the unit ball, and we observe
that
\begin{displaymath}
|\phi_\rho(\xi - \xi_\tau)| \geq \frac{1}{\rho^3}
\hspace{0.25in} \mbox{on} \hspace{0.25in} B(\xi_\tau,\rho).
\end{displaymath}
We also define the function $F_\tau$ on $B(\xi_\tau,\rho) \cap S$ by the
equation
\begin{displaymath}
f_\tau(\xi) = \phi_\rho(\xi - \xi_\tau) F_\tau(\xi)
\end{displaymath}
and we observe that $|F_\tau| \leq \rho^3 |f_\tau|$. Then
\begin{displaymath}
E f_\tau (x)
= \Big( \phi_\rho(\cdot - \xi_\tau) \widehat{\Big)\;} \ast E F_\tau (x)
= \int e^{-2 \pi i (x-y) \cdot \xi_\tau}
       \widehat{\phi}(\rho(x-y)) E F_\tau (y) dy,
\end{displaymath}
so that
\begin{eqnarray*}
|E f_\tau (x)|^p
& \leq & \Big( \int |\widehat{\phi}(\rho(x-y))| \; |E F_\tau (y)| dy \Big)^p
                                                                          \\
& \leq & \Big( \int |\widehat{\phi}(\rho(x-y))| dy \Big)^{p-1}
         \Big( \int |E F_\tau (y)|^p |\widehat{\phi}(\rho(x-y))| dy \Big) \\
&   =  & \frac{1}{\rho^{3(p-1)}} \big\| \widehat{\phi} \, \big\|_{L^1}^{p-1}
          \int |E F_\tau (y)|^p |\widehat{\phi}(\rho(x-y))| dy,
\end{eqnarray*}
which gives
\begin{eqnarray*}
\lefteqn{\int_{B(0,R)} |E f_\tau (x)|^p H(x) dx} \\
& \leq & \frac{\rho^3}{\rho^{3p}}\big\| \widehat{\phi} \, \big\|_{L^1}^{p-1}
         \int |E F_\tau (y)|^p
         \int_{B(0,R)} |\widehat{\phi}(\rho(x-y))| H(x) dx dy.
\end{eqnarray*}
We now define the function ${\mathcal H}$ on $\mbb R^3$ by
\begin{displaymath}
{\mathcal H}(y)=
\big\| \widehat{\phi} \, \big\|_{L^\infty}^{-1} A_\alpha(H)^{-1} \rho^\alpha
\int_{B(0,R)} |\widehat{\phi}(\rho(x-y))| H(x) dx,
\end{displaymath}
notice that ${\mathcal H}$ is supported in the ball $B(0,R+\rho^{-1})$, and
conclude that
\begin{eqnarray*}
\lefteqn{\int_{B(0,R)} |E f_\tau (x)|^p H(x) dx} \\
& \leq & \big\| \widehat{\phi} \, \big\|_{L^\infty} \, A_\alpha(H) \,
         \big\| \widehat{\phi} \, \big\|_{L^1}^{p-1} \rho^{3 - \alpha}
         \rho^{-3 p} \int_{B(0,R+\rho^{-1})} |E F_\tau (y)|^p
         {\mathcal H}(y)dy.
\end{eqnarray*}
Since $\widehat{\phi}$ is supported in the unit ball, we have
\begin{displaymath}
\big\| \widehat{\phi} \, \big\|_{L^1}
\leq |B(0,1)| \, \big\| \widehat{\phi} \, \big\|_{L^\infty}
\leq 5 \big\| \widehat{\phi} \, \big\|_{L^\infty}.
\end{displaymath}
Thus
\begin{eqnarray}
\label{EfbdbyEF}
\lefteqn{\int_{B(0,R)} |E f_\tau (x)|^p H(x) dx} \nonumber \\
& \leq & 5^3 \big\| \widehat{\phi} \, \big\|_{L^\infty}^4
         A_\alpha(H) \rho^{3 - \alpha - 3 p}
         \int_{B(0,R+\rho^{-1})} |E F_\tau (y)|^p {\mathcal H}(y)dy
\end{eqnarray}
($\| \widehat{\phi} \|_{L^\infty} \geq 1$ because
$\widehat{\phi}(0)= \| \phi \|_{L^1} \geq 1$).

The function ${\mathcal H}$ is a weight on $\mbb R^3$ of the same dimension
as $H$. In fact,
\begin{displaymath}
{\mathcal H}(y) \leq
\frac{\rho^\alpha \big\| \widehat{\phi} \,\big\|_{L^\infty}^{-1}}
     {A_\alpha(H)} \big\| \widehat{\phi} \, \big\|_{L^\infty}
\int_{B(y,1/\rho)} H(x)dx
\leq \frac{\rho^\alpha}{A_\alpha(H)} A_\alpha(H)
\Big( \frac{1}{\rho} \Big)^\alpha = 1
\end{displaymath}
(provided $1/\rho \geq 1$) for all $y \in \mbb R^3$, so
$\| {\mathcal H} \|_{L^\infty} \leq 1$. Also,
\begin{eqnarray*}
\int_{B(y_0,t)} {\mathcal H}(y) dy
& = & \int \chi_{B(y_0,t)}(y) {\mathcal H}(y) dy \\
& = & \frac{\rho^\alpha \big\| \widehat{\phi} \,\big\|_{L^\infty}^{-1}}
           {A_\alpha(H)}
      \int_{B(0,R)} \int \chi_{B(y_0,t)}(y) \,
      |\widehat{\phi}(\rho(x-y))| dy H(x)dx.
\end{eqnarray*}
Applying the change of variables $v=\rho(x-y)$ to the inner integral, we get
\begin{eqnarray*}
\int_{B(y_0,t)} {\mathcal H}(y) dy
& = & \frac{\rho^{\alpha-3} \big\| \widehat{\phi} \,\big\|_{L^\infty}^{-1}}
           {A_\alpha(H)}
      \int_{B(0,R)} \int \chi_{B(y_0,t)}(x-\rho^{-1}v) \,
      |\widehat{\phi}(v)| dv H(x) dx \\
& = & \frac{\rho^{\alpha-3} \big\| \widehat{\phi} \,\big\|_{L^\infty}^{-1}}
           {A_\alpha(H)}
      \int_{B(0,R)} |\widehat{\phi}(v)|
      \int \chi_{B(y_0,t)}(x-\rho^{-1}v) H(x) dx dv \\
& = & \frac{\rho^{\alpha-3} \big\| \widehat{\phi} \,\big\|_{L^\infty}^{-1}}
           {A_\alpha(H)}
      \int_{B(0,R)} |\widehat{\phi}(v)|
      \int \chi_{B(y_0+\rho^{-1}v,t)}(x) H(x) dx dv \\
& \leq & \frac{\rho^{\alpha-3}
               \big\| \widehat{\phi} \,\big\|_{L^\infty}^{-1}}
              {A_\alpha(H)}
     \int_{B(0,R)} |\widehat{\phi}(v)| A_\alpha(H) t^\alpha dv \\
& \leq & \big\| \widehat{\phi} \, \big\|_{L^\infty}^{-1}
         \big\| \widehat{\phi} \, \big\|_{L^1}
         \, \rho^{\alpha-3} t^\alpha
\end{eqnarray*}
for all $y_0 \in \mbb R^3$ and $t \geq 1$, so that
\begin{equation}
\label{bdonfancyh}
A_\alpha({\mathcal H})
\leq \big\| \widehat{\phi} \, \big\|_{L^\infty}^{-1}
     \big\| \widehat{\phi} \, \big\|_{L^1} \, \rho^{\alpha-3}
\leq 5 \, \rho^{\alpha-3}.
\end{equation}

We know that $\tau$ is the graph of $h$ over $B^2(\omega_0,r)$, so, by
(\ref{bdesbyes1}),
\begin{eqnarray*}
\int_{B(0,R+\rho^{-1})} |E F_\tau(x)|^p {\mathcal H}(x) dx
&   =  & \int_{B(0,R+\rho^{-1})} |E_S F_\tau(x)|^p {\mathcal H}(x) dx \\
& \leq & r^{-4} \int_{B(0,4rR+4)} |E_{S_1} G(u)|^p {\mathcal H}'(u) du
\end{eqnarray*}
($B(0,4rR+4r\rho^{-1}) \subset B(0,4rR+4)$ because $r \leq \rho$), and so
(choosing $K$ large enough for $4rR+4< R/2$) the induction hypothesis tells
us that
\begin{eqnarray*}
\lefteqn{\int_{B(0,R+\rho^{-1})} |EF_\tau(x)|^p {\mathcal H}(x) dx}\\
& \leq & r^{-4} C_\epsilon (4rR+4)^\epsilon \Big( \max
         \big[ A_\alpha({\mathcal H}'), A_\alpha({\mathcal H}')^{q_1} \big]
         \Big) (4rR+4)^{q_2}
         \| G \|_{L^2(S_1)}^{\gamma} \| G \|_{L^\infty(S_1)}^{p-\gamma}.
\end{eqnarray*}
Since
\begin{displaymath}
G(\eta,h_1(\eta))= r^2 F_\tau(\omega_0+r\eta,h(\omega_0+r\eta))
                   J_h(\omega_0+r\eta) J_{h_1}(\eta)^{-1}
\end{displaymath}
and $1 \leq J \leq 3$, it follows that
\begin{displaymath}
\| G \|_{L^\infty(S_1)} \leq 3 r^2 \| F_\tau \|_{L^\infty(S)}.
\end{displaymath}
Also, since
\begin{displaymath}
\int_{B^2(0,1)} |G(\eta,h_1(\eta))|^2 J_{h_1}(\eta) d\eta
\leq 3r^2 \int_{B^2(\omega_0,r)}
     |F_\tau(\omega,h(\omega))|^2 J_h(\omega) d\omega,
\end{displaymath}
we have
\begin{displaymath}
\| G \|_{L^2(S_1)}^2 \leq 3 r^2 \| F_\tau \|_{L^2(S)}^2.
\end{displaymath}
Therefore,
\begin{eqnarray*}
\lefteqn{\int_{B(0,R+\rho^{-1})} |EF_\tau(x)|^p {\mathcal H}(x)dx} \\
& \leq & r^{-4} C_\epsilon \Big( \frac{R}{2} \Big)^\epsilon \Big( \max
         \big[ A_\alpha({\mathcal H}'), A_\alpha({\mathcal H}')^{q_1} \big]
         \Big) \Big( \frac{R}{2} \Big)^{q_2}
         \| G \|_{L^2(S_1)}^{\gamma} \| G \|_{L^\infty(S_1)}^{p-\gamma} \\
& \leq & 3^4 r^{2p-4-\gamma} C_\epsilon R^\epsilon \Big( \max
         \big[ A_\alpha({\mathcal H}'), A_\alpha({\mathcal H}')^{q_1} \big]
         \Big) R^{q_2} \| F_\tau \|_{L^2(S)}^{\gamma}
         \| F_\tau \|_{L^\infty(S)}^{p-\gamma}.
\end{eqnarray*}
Since (by (\ref{hprime192h}) and (\ref{bdonfancyh}))
\begin{displaymath}
A_\alpha({\mathcal H}') \leq (192) r^{3-\alpha} A_\alpha({\mathcal H})
\leq (192) r^{3-\alpha} (5 \rho^{\alpha-3}) \leq 960
\end{displaymath}
and $0 \leq q_1 \leq 1$, it follows that
\begin{displaymath}
\max \big[ A_\alpha({\mathcal H}'), A_\alpha({\mathcal H}')^{q_1} \big]
\leq 960.
\end{displaymath}
So
\begin{eqnarray*}
\lefteqn{\int_{B(0,R+\rho^{-1})} |EF_\tau(x)|^p {\mathcal H}(x)dx} \\
& \leq & (3^4 \times 960) r^{2p-4-\gamma} C_\epsilon R^\epsilon R^{q_2}
         \| F_\tau \|_{L^2(S)}^{\gamma}
         \| F_\tau \|_{L^\infty(S)}^{p-\gamma},
\end{eqnarray*}
and so (using (\ref{EfbdbyEF}))
\begin{eqnarray*}
\lefteqn{\int_{B(0,R)} |Ef_\tau(x)|^p H(x)dx} \\
& \leq & 10^7 \big\| \widehat{\phi} \, \big\|_{L^\infty}^4
         \rho^{3-\alpha-3p} r^{2p-4-\gamma} C_\epsilon R^\epsilon
         A_\alpha(H) R^{q_2} \| F_\tau \|_{L^2(S)}^{\gamma}
         \| F_\tau \|_{L^\infty(S)}^{p-\gamma}.
\end{eqnarray*}
Recalling that $|F_\tau| \leq \rho^3 |f_\tau|$, this becomes
\begin{eqnarray*}
\lefteqn{\int_{B(0,R)} |Ef_\tau(x)|^p H(x)dx} \\
& \leq & 10^7 \big\| \widehat{\phi} \, \big\|_{L^\infty}^4 \rho^{3-\alpha}
         r^{2p-4-\gamma} C_\epsilon R^\epsilon A_\alpha(H) R^{q_2}
         \| f_\tau \|_{L^2(S)}^{\gamma}
         \| f_\tau \|_{L^\infty(S)}^{p-\gamma} \\
& \leq & (10^7 \times 3^3) \big\| \widehat{\phi} \, \big\|_{L^\infty}^4
         r^{2p-\alpha-1-\gamma} C_\epsilon R^\epsilon A_\alpha(H) R^{q_2}
         \| f_\tau \|_{L^2(S)}^{\gamma} \| f \|_{L^\infty(S)}^{p-\gamma}
\end{eqnarray*}
(recall that $\rho \leq 3r$). Thus
\begin{eqnarray*}
\lefteqn{K^\epsilon \sum_\tau \int_{B(0,R)} |Ef_\tau(x)|^p H(x)dx} \\
& \leq & 10^9 \big\| \widehat{\phi} \, \big\|_{L^\infty}^4
         r^{2p-\alpha-1-\gamma-\epsilon} C_\epsilon R^\epsilon A_\alpha(H)
         R^{q_2} \| f \|_{L^\infty(S)}^{p-\gamma}
         \sum_\tau \| f_\tau \|_{L^2(S)}^{\gamma}.
\end{eqnarray*}
Now
\begin{displaymath}
\sum_\tau \| f_\tau \|_{L^2(S)}^{\gamma}
= \sum_\tau \Big( \int |f_\tau|^2 d\sigma \Big)^{\gamma/2}
\leq \Big( \sum_\tau \int |f_\tau|^2 d\sigma \Big)^{\gamma/2}
= \| f \|_{L^2(S)}^{\gamma}
\end{displaymath}
provided $\gamma \geq 2$, and
\begin{displaymath}
10^9 \big\| \widehat{\phi} \, \big\|_{L^\infty}^4
r^{2p-\alpha-1-\gamma-\epsilon} C_\epsilon
\leq \bar{C}_\epsilon + 10^4 \sigma(S)^4
\end{displaymath}
provided
\begin{displaymath}
10^9 \big\| \widehat{\phi} \, \big\|_{L^\infty}^4
r^{2p-\alpha-1-\gamma-\epsilon} \leq \frac{1}{2},
\end{displaymath}
so
\begin{eqnarray*}
\lefteqn{K^\epsilon \sum_\tau \int_{B_R(0)} |Ef_\tau(x)|^p H(x)dx} \\
& \leq & \big( \bar{C}_\epsilon + 10^4 \sigma(S)^4 \big) R^\epsilon
         A_\alpha(H) R^{q_2}
         \| f \|_{L^2(S)}^{\gamma} \| f \|_{L^\infty(S)}^{p-\gamma}
\end{eqnarray*}
provided
\begin{displaymath}
r^{2p-\alpha-1-\gamma-\epsilon} \leq 10^{-10}
\big\| \widehat{\phi} \, \big\|_{L^\infty}^{-4}.
\end{displaymath}
Since $\lim_{\epsilon \to 0} K(\epsilon)= \infty$, the induction closes if
$2p-\alpha-1-\gamma > 0$, and we obtain
\begin{displaymath}
\int_{B(0,R)} |Ef(x)|^p H(x)dx
\leq C_\epsilon R^\epsilon
     \Big( \max \big[ A_\alpha(H), A_\alpha(H)^{q_1} \big] \Big)
     R^{q_2} \| f \|_{L^2(S)}^{\gamma} \| f \|_{L^\infty(S)}^{p-\gamma},
\end{displaymath}
as desired.
\end{proof}

\section{Proof of Theorem \ref{mainjj}}

(i) We let $b=1$ in Theorem \ref{thejapp}. Then
$p=2(4\alpha+3)/(2\alpha+3)$, and to every
$0 < \epsilon \leq \min[c_0,(p-3)/2]$ there are constants $K=K(\epsilon)$
and $C_\epsilon$ such that $\lim_{\epsilon \to 0} K= \infty$ and
\begin{displaymath}
\int_{B(0,R)} \mbox{Br}_{K^{-\epsilon}} Ef(x)^p H(x)dx
\leq C_\epsilon R^\epsilon A_\alpha(H)^{1-\frac{p}{4}}
     \| f \|_{L^2(S)}^{3+2\epsilon}
\end{displaymath}
for all functions $f \in L^2(S)$ that satisfy the inequality
\begin{equation}
\label{lastsecb1}
\int_{B(\xi_0,R^{-1/2}) \cap S} |f|^2 d\sigma \leq \frac{1}{R}
\end{equation}
for all $\xi_0 \in S$.

Given a non-zero function $f \in L^\infty(S)$, we see that the function
$\| f \|_{L^\infty(S)}^{-1} f$ satisfies (\ref{lastsecb1}), and the above
estimate becomes
\begin{displaymath}
\int_{B(0,R)} \mbox{Br}_{K^{-\epsilon}} Ef(x)^p H(x)dx
\leq \bar{C}_\epsilon R^\epsilon A_\alpha(H)^{1-\frac{p}{4}}
     \| f \|_{L^2(S)}^3 \| f \|_{L^\infty(S)}^{p-3}.
\end{displaymath}
Applying Theorem \ref{parabscaling} with $q_1=1-(p/4)$, $q_2=0$, and
$\gamma=3$, we get the required result provided $2p>\alpha+4$. Solving this
inequality for $\alpha$, we get $3/2 < \alpha < 5/2$. We have thus proved
part (i) except for the case $\alpha=3/2$.

We remind the reader about what we mentioned in \S 1.1 concerning the case
$\alpha = 3/2$ of Theorem \ref{mainjj}. When $\alpha=3/2$, parts (i) and
(ii) of Theorem \ref{mainjj} agree, but the proof belongs to part (ii).

(ii) We suppose first that $\alpha > 3/2$. We let
\begin{displaymath}
b = b_\epsilon = \frac{\alpha-(3/2)}{2 \epsilon}-\alpha-\frac{1}{2},
\end{displaymath}
in Theorem \ref{thejapp}. This requires some explanation. The conclusion of
Theorem \ref{thejapp} holds for $0 < \epsilon \leq \min[c_0,(p-3)/2]$. Since
$b \geq 1$, we have
\begin{displaymath}
\frac{p-3}{2}
= \frac{1}{2} \Big( \frac{8\alpha+6b}{2\alpha+2b+1} - 3 \Big)
\leq \frac{\alpha-(3/2)}{2\alpha+3}.
\end{displaymath}
So we assume that $c_0 \leq (\alpha-(3/2))/(2\alpha+3)$ and choose $b$ to
satisfy $\epsilon = (p-3)/2$. Solving this equation for $b$, we arrive at
the solution $b=b_\epsilon$ as above.

We, therefore, have the following estimate on the broad part of $Ef$:
\begin{displaymath}
\int_{B(0,R)} \mbox{Br}_{K^{-\epsilon}} Ef(x)^p H(x)dx
\leq C_\epsilon R^{(b+1)\epsilon/2} A_\alpha(H)^{1-\frac{p}{4}}
\| f \|_{L^2(S)}^{3+2\epsilon}
\end{displaymath}
whenever $R \geq 1$, $f \in L^2(S)$, and
\begin{equation}
\label{ridofbL2}
\int_{B(\xi_0,R^{-1/2}) \cap S} |f|^2 d\sigma \leq
\frac{1}{R^{(b+1)/2}}
\end{equation}
for all $\xi_0 \in S$.

Given a non-zero function $f \in L^2(S)$, we see that the function
$R^{-(b+1)/4} \| f \|_{L^2(S)}^{-1} f$ satisfies (\ref{ridofbL2}), and the
above estimate becomes
\begin{displaymath}
\int_{B(0,R)} \mbox{Br}_{K^{-\epsilon}} Ef(x)^p H(x)dx \leq C_\epsilon
A_\alpha(H)^{1-\frac{p}{4}} R^{(p-3)(b+1)/4} \| f \|_{L^2(S)}^p.
\end{displaymath}

By H\"{o}lder's inequality, we have
\begin{displaymath}
\int_{B(0,R)} \mbox{Br}_{K^{-\epsilon}} Ef(x)^3 H(x)dx
\leq \Big( A_\alpha(H) R^\alpha \Big)^{1-\frac{3}{p}}
\Big( \int_{B_R(0)} \mbox{Br}_{K^{-\epsilon}} Ef(x)^p H(x)dx
\Big)^{\frac{3}{p}},
\end{displaymath}
and hence
\begin{displaymath}
\Big( \int_{B(0,R)} \mbox{Br}_{K^{-\epsilon}} Ef(x)^3 H(x)dx
\Big)^{\frac{1}{3}}
\leq C_\epsilon^{\frac{1}{p}} R^{\alpha \frac{p-3}{3p}}
     A_\alpha(H)^{\frac{1}{12}} R^{\frac{(p-3)(b+1)}{4p}} \| f \|_{L^2(S)}.
\end{displaymath}
Inserting for $b$ its value in term of $\epsilon$, we get
\begin{displaymath}
(p-3)(b+1) = \alpha-\frac{3}{2}- \Big( \alpha-\frac{1}{2} \Big) (2\epsilon),
\end{displaymath}
so that
\begin{eqnarray*}
\alpha \frac{p-3}{3p} + \frac{(p-3)(b+1)}{4p}
&   =  & \alpha \frac{2\epsilon}{3p} -
         \Big( \alpha-\frac{1}{2} \Big) \frac{\epsilon}{2p}
         + \frac{1}{4p} \Big( \alpha-\frac{3}{2} \Big) \\
&   =  & \frac{\epsilon}{4p} + \alpha \frac{\epsilon}{6p}
         + \frac{1}{4p} \Big( \alpha-\frac{3}{2} \Big) \\
& \leq & \frac{\epsilon}{4p} + \frac{\epsilon}{2p}
         + \frac{1}{12} \Big( \alpha-\frac{3}{2} \Big)
\end{eqnarray*}
(because $\alpha \leq 3 \leq p$), and hence
\begin{displaymath}
\int_{B(0,R)} \mbox{Br}_{K^{-\epsilon}} Ef(x)^3 H(x)dx
\leq \bar{C}_\epsilon R^\epsilon A_\alpha(H)^{\frac{1}{4}}
     R^{\frac{1}{4}(\alpha-\frac{3}{2})} \| f \|_{L^2(S)}^3.
\end{displaymath}
Applying Theorem \ref{parabscaling} with $q_1=1/4$,
$q_2=(1/4)(\alpha-(3/2))$, $p=3$, and $\gamma=3$, we arrive at the desired
result provided $2p > \alpha + 4$, i.e.\ provided $\alpha < 2$.

We have proved part (ii) of Theorem \ref{mainjj} in the regime
$3/2 < \alpha < 2$:
\begin{displaymath}
\int_{B(0,R)} |E f(x)|^3 H(x)dx
\leq C_\epsilon(\alpha,S) R^\epsilon A_{\alpha,3}(H)
     R^{\frac{1}{4}(\alpha - \frac{3}{2})} \| f \|_{L^2(S)}^3
\end{displaymath}
for all $f \in L^2(S)$ and $R \geq 1$. In particular, when
$\alpha=(3/2)+\epsilon$, this becomes
\begin{displaymath}
\int_{B(0,R)} |E f(x)|^3 H(x)dx
\leq C_\epsilon(S) R^\epsilon A_{(3/2)+\epsilon,3}(H)
     R^{\epsilon/4} \| f \|_{L^2(S)}^3.
\end{displaymath}
But from the definition of $A_\alpha(H)$, we see that
$A_\beta(H) \leq A_\alpha(H)$ if $\beta \geq \alpha$, so the same is true
for $A_{\alpha,p}(H)$, and so
\begin{displaymath}
\int_{B(0,R)} |E f(x)|^3 H(x)dx
\leq C_\epsilon(S) R^{2\epsilon} A_{(3/2),3}(H) \| f \|_{L^2(S)}^3.
\end{displaymath}

(iii) In this part we use Theorem \ref{themapp}. We have the following
estimate on the broad part of $Ef$:
\begin{displaymath}
\int_{B(0,R)} \mbox{Br}_{K^{-\epsilon}} Ef(x)^p H(x)dx
\leq C_\epsilon R^\epsilon A_\alpha(H)^{2-\frac{p}{2}}
     \| f \|_{L^2(S)}^3 \| f \|_{L^\infty(S)}^{p-3}
\end{displaymath}
for all $f \in L^\infty(S)$, where $p=13/4$. Of course,
\begin{displaymath}
\| f \|_{L^2(S)}^3
= \| f \|_{L^2(S)}^{\gamma} \| f \|_{L^2(S)}^{3-\gamma}
\lct \| f \|_{L^2(S)}^{\gamma} \| f \|_{L^\infty(S)}^{3-\gamma}
\end{displaymath}
whenever $0 \leq \gamma \leq 3$, so the above estimate implies that
\begin{displaymath}
\int_{B(0,R)} \mbox{Br}_{K^{-\epsilon}} Ef(x)^p H(x)dx
\leq \bar{C}_\epsilon R^\epsilon A_\alpha(H)^{2-\frac{p}{2}}
     \| f \|_{L^2(S)}^{\gamma} \| f \|_{L^\infty(S)}^{p-\gamma}.
\end{displaymath}
Applying Theorem \ref{parabscaling} with $q_1=2-(p/2)$, $q_2=0$, and
$2 \leq \gamma \leq 3$, we arrive at the desired conclusion provided
$2p-\alpha-1-\gamma > 0$, i.e.\ provided $\gamma < (11/2)-\alpha$.

\end{document}